\documentclass{article}

\usepackage{macros}

\title{Metric Linear Orders and O-Minimality}
\author{Aaron Anderson \\ University of Pennsylvania \and Diego Bejarano \\ York University}

\begin{document}

\maketitle

\begin{abstract}   
    In continuous logic, there are plenty of examples of interesting stable metric structures. However, on the other side of the SOP line, there are only a few metric structures where order is relevant, and orders often appear in different ways. We now present a unified approach to linear and cyclic orders in continuous logic.
    We axiomatize theories of metric linear and cyclic orders, and in the ultrametric case, find generic completions, analogous to the complete theory $\Th{DLO}$. 
    
    We then characterize which expansions of metric linear orders should be considered o-minimal in terms of regulated functions. We prove versions several key properties of o-minimal structures, such as definable completeness and distality, in the context of o-minimal metric structures. Finally we find examples of o-minimal (and weakly o-minimal) metric structures by turning to real closed metric valued fields.
\end{abstract}

\tableofcontents

\section{Introduction}

Continuous logic, as laid out in \cite{mtfms}, adapts the tools of model theory to better study analytic objects. It differs from standard (or as we will call it, \emph{discrete}) first-order logic in two ways. Firstly, it allows formulas to take values in a real interval $[0,1]$ instead of simply true and false. Secondly, it allows for equality to be replaced with a metric on the elements of a structure, with respect to which all other aspects of the structure must be continuous. These objects, \emph{metric structures}, include the full variety of structures from discrete logic as well as a growing menagerie of inherently continuous examples.

In terms of the standard classification of first-order structures, most of the examples of metric structures presented in \cite{mtfms} are \emph{stable} - informally, they display no order-like behavior \cite{schrodinger, localstab}.
Other well-studied examples tend to have been either stable or not \emph{NIP} - informally, combinatorially complicated
\cite{urysohn, rtrees, hansonIso}.
While in principle, we may expect non-discrete metric structures to exhibit all the same possibilities of behavior that discrete structures do and more, a strictly simple theory of metric structures was discovered only recently \cite{simple}. In the world of NIP but unstable structures, one may search for distal metric structures, but examples were only identified in \cite{distalexamples}.

In discrete logic, most examples of distal structures, or NIP-but-unstable structures in general, prominently feature a \emph{linear order}, but linear orders have been almost entirely absent from continuous logic. When ordered structures have been studied in continuous logic, such as ordered real closed metric valued fields \cite{mvf} or (portions of) value groups \cite{tilting}, they have been studied in ad-hoc ways, rather than with a general theory of linear orders.

This means that while strongly minimal structures, the best understood discrete structures, have an analog in the continuous logic literature \cite{hanson_smsets}, their ordered counterparts, the celebrated \emph{o-minimal} structures, do not.
This is particularly unfortunate, as the most celebrated examples of o-minimal structures are expansions of the real ordered field \cite{wilkie, r_an_exp}.
These structures have proved extremely useful in studying analytic objects (see \cite{vdd_omin} for general background and \cite{omin_gaga} for an example). These o-minimal expansions of the reals featuring analytic functions are central to powerful new applications of model theory, including the recent proof of the Andr\'e-Oort conjecture \cite{andre_oort}. In order to understand these analytic, real-valued functions using the lens of continuous logic, we must first understand linear orders in continuous logic in general, and o-minimal metric structures in particular.

\subsection{Theories of Metric Linear Orders}
In this paper, we will present several new theories whose models are metric structures imbued with a linear or cyclic order.
In order for the \emph{continuous} aspect of continuous logic to tell us something meaningful about these ``metric linear orders,'' rather than separate information about a metric and a linear order, we need some compatibility condition relating the metric and the order appropriately. We use the following convexity condition:
\begin{defn*}[\ref{meta-def}]
    We say that a set $M$ is a \emph{metric linear order} if it is assigned a linear ordering together with a metric such that each open ball in the metric is order-convex.
\end{defn*}

This compatibility condition facilitates a number of basic observations, including that the metric topology coarsens the order topology. This definition has appeared before (in the ultrametric case), in the structural Ramsey theory literature, as applied to ultrametric spaces taking values in a particular countable set, which can thus be studied in discrete logic \cite{vanthe1, vanthe2}.

To describe a linear order as a metric structure at all,
the set $\{(x,y) : x \leq y\}$ must be definable in the sense of continuous logic. As continuous logic is about real-valued predicates, a set $D$ is definable when it is closed and the distance predicate $d(x,D) = \inf_{y \in D}d(x,y)$ is definable. Thus one choice of symbol to represent $\leq$ is
$$d_\leq(x,y) = \inf_{w \leq z}d\left((x,y),(w,z)\right) = \inf_{w \leq z}\max\left(d(x,w),d(y,z)\right).$$

While we definitely want $d_\leq$ to be definable, another real-valued representation of $\leq$ is easier to work with. This is the distance \emph{to a ray}:
$$r(x,y) = d(x,(-\infty,y]) = \inf_{z \leq y}d(x,z).$$
The behavior of this ray predicate is easy to describe \emph{precisely under the compatibility condition of metric linear orders}, and using this ray predicate, we were able to axiomatize a straightforward theory $\MLO$ of metric linear orders:
\begin{thm*}[\ref{thm:mlo_models}]
    The models of $\MLO{}$ are exactly the complete, bounded metric linear orders.
\end{thm*}

The models of this theory are still very diverse, including, for instance, any linear order equipped with the discrete $\{0,1\}$-valued metric.
In order to isolate new \emph{metric} behavior that does not occur in discrete logic,
and focus on the \emph{dense} linear order case, we propose a theory $\MDLO$ of metric dense linear orders.
\begin{thm*}[Lemmas \ref{lem:top_agree}, \ref{lem:dense_ordered_distances}, \ref{lem:mdlo_dlo}]
    The models of \MDLO{} are precisely the metric linear orders $M$ such that for every $a \in M$, the sets $\{d(a,b): b < a\}$ and $\{d(a,b): b > a\}$ are both dense in $[0,1]$.

    Each such model is a dense linear order without endpoints where the metric and order topologies agree.
\end{thm*}

However, $\MDLO$ is still not precise enough to be complete - to find our first non-discrete complete theory of metric linear orders, we must place another constraint on the behavior of the metric. Motivated by trying to understand metric valued fields, we focus on the ultrametric case.
Adding the ultrametric axiom to $\MDLO$ gives us a theory $\UDLO$ of ultrametric dense linear orders, which indeed turns out to be complete.
In order to prove this, we needed to prove analogous facts about the pure metric reduct of this theory.
Thus in Section \ref{sec:ultra}, we study the metric reduct theory $\DU$ of ``dense ultrametrics'' and $\UDLO$ in parallel.
The theory $\DU$ had already been studied in \cite{conant_um}, not with an axiomatization, but as the complete theory of the ultrametric Urysohn spaces constructed in \cite{gao_shao_ultra}. These constructions were modified in \cite{vanthe1} to add a lexicographic linear ordering, inducing models of $\UDLO$.
We show, using a framework generalized to continuous logic in \cite{rtrees}, that both of these theories are model companions of more general theories, from which other useful results follow.
\begin{thm*}[Corollary \ref{cor:du_model_companion}, Theorem \ref{thm:udlo_model_companion}]
    The theory $\DU$ is complete, eliminates quantifiers, and is the model companion of the theory $\UM$ of ultrametric spaces in the empty language of metric structures.

    The theory $\UDLO$ is complete, eliminates quantifiers, and is the model companion of the theory $\ULO$ of ultrametric linear orders.
\end{thm*}

Unfortunately, unlike $\Th{DLO}$ in discrete logic, neither of these theories is separably categorical, as was observed for $\DU$ in \cite{conant_um}.
However, they are not far from that.
To recover some version of categoricity, we aim at approximate categoricity, a weaker condition that is sufficient to obtain a Ryll-Nardzewski-like result on the type space. This closely follows the framework of distortion systems laid out in \cite{hansonIso, hansonCat}, generalizing \cite{by_perturb,metricscott}.
As such, we recall many basic definitions and facts about approximate categoricity in Section \ref{sec:cat}, although we strongly recommend referring to \cite{hansonIso, hansonCat} for detailed background.
Ultimately, we do prove that they are approximately categorical - more precisely, \emph{Gromov-Hausdorff categorical} and an ordered modification of Gromov-Hausdorff categorical respectively:
\begin{thm*}[\ref{thm:DU_cat}, \ref{thm:UDLO_cat}]
    The theory \DU{} is strongly $\overline{GH}$-$\omega$-categorical, while \UDLO{} is strongly $\overline{OGH}$-$\omega$-categorical.
\end{thm*}

Motivated by understanding the projective lines of ordered metric valued fields, we also examined metric cyclic orders.
Cyclic orders have been frequently studied in model theory as reducts of linear orders \cite{mennen,tw_dp-min,ntp2}, and naturally capture the order information retained by, for instance, the real projective line $\R\mathbb{P}^1$.
Cyclic orders are also usually axiomatized in terms of the strict relation (which we denote $\cyc(x,y,z)$) analogous to $<$, while in continuous logic, definable sets must be closed. Thus we present an axiomatization of cyclic orders in discrete logic in terms of the nonstrict relation $\ceq(x,y,z) \iff \cyc(x,y,z) \vee x = y \vee y = z \vee x = z$. Once we have done this, we modify it into a metric version, analogous to metric linear orders, such that any metric linear order admits a reduct to a metric cyclic order. This can be envisioned as ``rolling up'' the linear order into a circle. Using this relation, we axiomatize a complete theory $\UDCO$ of ultrametric dense cyclic orders.
\begin{thm*}[Lemma \ref{lem:udlo_iff_udco}, Theorem \ref{thm:UDCO_complete}]
    The theory $\UDCO$ is complete, and is the theory of any cyclic reduct of $\UDLO$.
\end{thm*}

\subsection{o-Minimal Metric Structures}
With the groundwork laid for metric linear orders, we can explore what makes a metric linear order o-minimal.
In discrete logic, o-minimality is classified by all definable sets (with parameters) in dimension 1 being quantifier-free definable (with parameters) in just the order language.
In continuous logic, definable sets are of secondary importance to real-valued definable predicates, so we seek to define o-minimality in terms of all definable \emph{predicates} in dimension 1 being quantifier-free definable in just the order language $\{r(x,y)\}$.

This property, quantifier-free definability in just the order language, is usually stated in simpler terms: such a set is a finite union of intervals and points.
We propose a corresponding definition for real-valued predicates: that they be \emph{strongly regulated}.
\begin{defn*}[\ref{defn_strong_reg}]
Let $X$ be a linear order.
Let the set of \emph{strong step functions} on $X$ be the $\R$-linear span of the set of characteristic functions of intervals in $X \to \R$,
    and let the set of \emph{weak step functions} be the $\R$-linear span of the set of characteristic functions of order-convex sets.

Let $f : X \to \R$.
We say that $f$ is \emph{strongly regulated} if it is a uniform limit of strong step functions, and \emph{weakly regulated} if it is a uniform limit of weak step functions.
\end{defn*}
These are two possible generalizations to a general linear order $X$ of the definition of regulated functions in \cite{bourbaki_FVR}, where they occur in the theory of Riemann integration. (In fact, the inspiration to use regulated functions in this characterization came from \cite[Lemma 3.14]{anderson2}, where certain functions on indiscernible sequences are shown to be regulated, in order to integrate them.) The difference in these two corresponds to the difference between o-minimal and weakly o-minimal structures.

To justify the relevance of regulated functions to o-minimality, we show that they characterize definable predicates in discrete o-minimal structures, which we treat as metric structures with the $\{0,1\}$-valued metric:
\begin{thm*}[\ref{thm:discrete_omin_reg}]
    Suppose $M$ is a structure of discrete logic expanding a linear order.
    Then if $f : M \to [0,1]$ is strongly regulated, then using continuous logic, there is some definable predicate $\phi(x)$ with $\phi^M = f$.

    Conversely, if $M$ is o-minimal and $\phi(x)$ is a definable predicate, then $\phi^M$ is strongly regulated. If $M$ is weakly o-minimal and $\phi(x)$ is a definable predicate, then $\phi^M$ is weakly regulated.
\end{thm*}

We then find that in any metric linear order, quantifier-free definable predicates with parameters in dimension 1 are strongly regulated, indeed all strongly regulated functions \emph{that are continuous with respect to the metric} 
\begin{thm*}[Lemma \ref{lem:qf_reg}, Theorem \ref{thm:reg_qf}]
    Let $M$ be a metric linear order and let $f : M \to [0,1]$. Then $f$ is given by a quantifier-free definable predicate with parameters in the language $\{r\}$ if and only if $f$ is strongly regulated and continuous with respect to the metric.
\end{thm*}
We finally define o-minimality and weak o-minimality accordingly:
\begin{defn*}[\ref{defn:omin}]
    Let $\mathcal{L}$ be a language including $r$, and let $M$ be an $\mathcal{L}$-structure expanding a metric linear order. We say that $M$ is \emph{o-minimal} when either of the following two conditions holds for every definable predicate $\phi(x)$ with $|x| = 1$:
    \begin{itemize}
        \item $\phi(x)$ is equivalent to a quantifier-free definable predicate in the reduced language $\{r\}$
        \item the interpretation of $\phi(x)$ is strongly regulated.
    \end{itemize}

    We call $M$ \emph{weakly o-minimal} when the interpretation of every definable predicate $\phi(x)$ with $|x| = 1$ is weakly regulated.
\end{defn*}

We then prove versions of some classic o-minimality results
for o-minimal or weakly o-minimal expansions of appropriate theories of metric linear orders. These include a characterization of definable sets in dimension 1, including definable completeness, as well as distality of weakly o-minimal metric theories.
\begin{thm*}[Lemma \ref{lem:def_complete}, Theorem \ref{thm:omin_def_set}]
    In an o-minimal expansion $M$ of an $\MDLO$, a set $D \subseteq M$ is definable if and only if $M \setminus D$ is a countable union of disjoint open intervals, of which only finitely many have diameter $\geq \varepsilon$ for any fixed $\varepsilon > 0$.

    In particular, if $D$ is bounded above (resp. below), then $D$ has a least upper bound (resp. greatest lower bound).
\end{thm*}

\begin{thm*}[\ref{thm:distal}]
    If $T$ is a theory expanding $\MLO$ such that all models of $T$ are weakly $o$-minimal, then $T$ is distal.
\end{thm*}

Finally, we apply these definitions to metric valued fields.
We show that models of $\ORCMVF$, the projective lines of ordered real closed metric valued fields defined in \cite{mvf}, have a cyclic order reduct satisfying $\UDCO$.
If we rephrase our definition of o-minimality for cyclic orders, these are o-minimal:
\begin{thm*}[\ref{thm:ORCMVF_reg}]
    Every model of $\ORCMVF{}$ is a cyclically o-minimal expansion of $\UDCO{}$.
\end{thm*}

The formalism of metric valued fields in \cite{tilting} produces linear, rather than cyclic orders. However, these structures, which we refer to as \emph{metric valuation rings}, are more model theoretically wild.
We develop a theory of ordered real closed metric valuation rings, expanding $\UDLO$, but these are, like real closed valued fields in discrete logic, only weakly o-minimal:
\begin{thm*}[\ref{thm:ORCMVR_omin}]
    Every model of $\ORCMVR$ is weakly o-minimal, but not o-minimal.
\end{thm*}

Along the way, we observe that the value groups (or rather, value monoids-with-zero) of metric valuation rings, interpretable in the structure described in \cite{tilting}, are ultrametric linear orders.
\begin{thm*}[\ref{thm:value_ulo}]
    Suppose that $M$ is a model of 
    The imaginary $v\mathcal{O}$ is a metric linear order, and for all $x,y$, $r(v(x),v(y)) = D(x,y)$, so this metric linear order structure is interpretable. In particular, $v\mathcal{O} \vDash \ULO$.
\end{thm*}
By studying these linear orders, it may be possible to develop an Ax-Kochen-Ershov principle for metric valuation rings, analogous to the version of the principle for projective lines of metric valued fields in \cite{ake}.

\subsection*{Acknowledgements}
We would like to thank Artem Chernikov, Isaac Goldbring, C. Ward Henson, Nick Ramsey, and Tom Scanlon for advising and helpful conversations during the writing of this paper.
We would additionally like to thank Ita\"i Ben Yaacov and James Hanson for inspiration and helpful references.


\section{Continuous Logic Preliminaries}

In this section we give a concise presentation of the basics of continuous logic for readers who might be unfamiliar with the subject. The reader should be aware that this presentation provides only a cursory exposition of the basic definitions and results we will use throughout the paper. For a standard introduction to the subject, we refer the reader to \cite{mtfms}. 

\begin{defn}
A metric structure $M=(M,d;\{c_i\}_{i\in I}, \{f_j\}_{j\in J}, \{P_k\}_{k\in K})$ consists of a complete, bounded metric space $(M,d)$ of diameter $1$ augmented with an additional structure given by collection of elements $c_i\in M$, of uniformly continuous functions $f_j:M^n\to M$, and of uniformly continuous functions $P_k:M^n\to [0,1]$ (we call these functions predicates).
\end{defn}

An illustrative example is that of Banach spaces: If $B$ is the unit ball of a Banach space $X$ over $\mathbb{R}$ or $\mathbb{C}$, we can name the point $0_X\in B$, and add the norm $||\cdot||:B\to [0,1]$ as a predicate. We can also add functions $f_{\alpha,\beta}(x,y) = \alpha x + \beta y$ for each pair of scalars with $|\alpha| + |\beta| \leq 1$. In addition, if one wishes to capture the structure of a C$^*$-algebra, we can include multiplication and the $^*$-map as functions. 

\begin{rmk}
Metric structures can be defined more generally by allowing the range of the predicates to live inside arbitrary closed intervals of the reals. However, we do not require these level of generality in this article.
\end{rmk}

A model-theoretic study of metric structures can be achieved by developing a suitable logic. As in the case of first-order logic, the notions of atomic formulas, formulas, sentences, and truth in a model are defined inductively. Here we give a short explanation:
\begin{defn}
\begin{enumerate}
\item We start with a \textbf{vocabulary}: a collection of formal symbols. We allow three kinds of symbols: constants, functions, and predicates. As with first-order logic, functions and predicates have a prescribed number of arguments they take, called the \textbf{arity}, but in continuous logic the language also specifies the \textbf{modulus of continuity} for functions and predicates.  For example, $\tau=\{0_X,||\cdot||,\{f_{\alpha,\beta}\}_{|\alpha|+|\beta|\leq 1}\}$ is the vocabulary for unit balls of Banach spaces.

\item The \textbf{formulas} of continuous logic are constructed in an analogous way to the first-order setting: we apply \textbf{predicates} to \textbf{terms} to create \textbf{atomic formulas}. Then, we can combine those formulas using continuous functions of the form $u:[0,1]^n\to[0,1]$ as our \textbf{connectives}. One important example is the restricted subtraction function $x\ \dot-\ y=\max(x-y,0)$. For \textbf{quantifiers}, continuous logic uses $\inf_x$ and $\sup_x$. \textbf{Free and bound} variables are defined analogously as in first order logic, and a formula without free variables is called a \textbf{sentence}.

\item Formulas are then \textbf{uniformly continuous} functions $\phi:M^n\to [0,1]$. We usually write $\phi^M(\bar a)=r$ to mean that applying $\phi$ to the tuple $\bar a$ of elements from $M$ gives the value $r$ as a result. For example, consider the sentence $\phi := \sup_{x,y} |d(x,y)-d(y,x)|.$ Given a metric structure $M$, we have that $\phi^M=0$ since the metric is symmetric. The guiding principle is that a sentence is \textbf{true} in $M$ if $\phi^M=0$, and otherwise the value of $\phi^M$ is a measure of the degree of failure. A \textbf{language} $L$, on some vocabulary $\tau$, is the set of all formulas in that vocabulary. We will often omit mentioning both if they are clear from context.

\item One difference with first-order logic is that we will often make use of the uniform limits of formulas, where the uniform limit is taken with respect to a particular metric structure $M$.  Following the convention of \cite{mtfms}, a uniform limit of formulas is called a \textbf{definable predicate}. We usually make a slight abuse of notation and write definable predicates as functions $\phi:M^n\to[0,1]$. We should note that definable predicates can also be derived purely syntactically without referencing a metric structure using forced limits (see \cite{mtfms}).  

\item By a \textbf{quantifier-free definable predicate}, we mean a definable predicate that is the uniform limit of quantifier-free formulas. 

\item Given a metric structure $M$ in a language $L$, and a set $B\subset M$, we can construct the language $L(B)$ by adding the elements of $B$ as constants in the vocabulary, and treat $M$ as an $L(B)$-metric structure. If a sequence of $L(B)$-formulas converges uniformly in $M$ to a function $\phi:M^n\to[0,1]$, we call $\phi$ an $L(B)$-definable predicate.

\item A \textbf{theory} $T$ is a set of sentences. We say that a metric structure $M$ is a \textbf{model of a theory} $T$, and write $M\models T$, if $\phi^M=0$ for every $\phi\in T$. Also, the theory of $\mathrm{Th}(M)$ is the set of all sentences $\phi$ such that $\phi^M=0$. A theory is \textbf{complete} if $T=\mathrm{Th}(M)$ for some metric structure $M$. We call the expression $\phi=0$, for a sentence $\phi$, a \textbf{condition}.

\item Given a theory $T$ on a language $L$, we write $\mathrm{Mod}_L(T)$ for the set of all metric structures that are models of $T$. Given a class of $L$-structures $\mathcal{C}$, we say that $\mathcal{C}$ is \textbf{axiomatizable} if the is theory $T$ such that $\mathcal{C}=\mathrm{Mod}_L(T)$, and say that $T$ is the set of axioms for $\mathcal{C}$ in $L$. 

\item A theory $T$ \textbf{eliminates quantifiers} if for every definable predicate $\phi(x_1,\dotsc, x_n)$ there is a quantifier-free definable predicate $\psi(x_1,\dotsc,x_n)$ such that for every $M\models T$ we have $$|\phi-\psi|^M=0.$$
\end{enumerate}
\end{defn}

The remainder of this section explores some basic concepts of continuous logic that will be used in future sections. We start with definable sets:

\begin{rmk}
Given a metric structure $M$ with distance $d$, we treat $M^n$ as a (complete, bounded) metric space equipped with the metric 
$$d_n(\bar a, \bar b) = \max_{1\leq i\leq n}d(a_i,b_i).$$
\end{rmk}

\begin{defn}
Let $M$ be a metric $L$-structure, $D$ a non-empty, metrically closed subset of $M^n$ and $B\subset M$ a set of parameters. We say that $D$ is an $L(B)$-definable set (or that $D$ is definable in $M$ with parameters from $B$) if there is an $L(B)$-definable predicate $\phi$ such that 
$$\phi^M(\bar x) = d_n^M(\bar x,D): \bar x\mapsto \inf_{\bar a\in D} d_n^M(\bar x,\bar a).$$
\end{defn}
This conditions might seem quite strong to the reader. However, in continuous logic, this condition is equivalent to quantification over the set $D$:

\begin{fact}[{\cite[Theorem  9.17]{mtfms}}]
For a closed set $D\subset M^n$ the following are equivalent:
\begin{enumerate}
\item $D$ is definable in $M$ with parameters from $B$.
\item For any $L(B)$-definable predicate $\phi:M^m\times M^n\to[0,1]$, the predicate $\psi:M^m\to[0,1]$ defined by
$$\psi(\bar x)= \inf_{\bar y\in D}\phi(\bar x,\bar y)$$
is $L(B)$-definable (and similarly for sup). 
\end{enumerate}
\end{fact}

With this in mind, we can give a basic definition of the definable and algebraic closure of a set:

\begin{defn}
Let $M$ be a metric $L$-structure, a set $B\subset M$, and a tuple $\bar a\in M^n$. We say that $\bar a$ is definable over $B$ if $\{\bar a\}$ is $L(B)$-definable (i.e. $d_n(\bar x,\bar a)$ is an $L(B)$-definable predicate). We say that $\bar a$ is algebraic over $B$ if there is a compact set $C\subset M^n$ such that $\bar a\in C$ and $C$ is $L(B)$-definable. Thus, we define:
$$\mathrm{dcl}_M(B)=\{\bar a\mid \text{ $\bar a$ is definable over $B$}\}$$
and
$$\mathrm{acl}_M(B)=\{\bar a\mid \text{ $\bar a$ is algebraic over $B$}\}$$
\end{defn}

\section{A Language For Ordered Metric Structures}
In discrete logic, definable sets in dimension 1 in o-minimal structures are exactly those that are quantifier-free definable in the language $\{\leq\}$,
so to generalize the definition of o-minimality, we will first look for the proper generalization of the language $\{\leq\}$.

\begin{defn}\label{meta-def}
    We say that a set $M$ is a \emph{metric linear order} if it is assigned a linear ordering together with a metric such that each open ball in the metric is order-convex.
    
    In a metric linear order $M$, we define the following functions:
    Let $d_\leq(x,y)$ be the distance predicate to $\{(x,y) : x \leq y\}$,
    and let $r(x,y)$ denote the distance from $x$ to the ray $(-\infty,y]$.
\end{defn}

Metric linear orders were defined (under the name of ``convex linear orders'') in \cite{vanthe1,vanthe2}, where they are studied from a structural Ramsey theory perspective, without using the framework of continuous logic.

\begin{lem}\label{meta-basic}
    Let $M$ be a metric linear order.
    Then 
    \begin{itemize}
        \item Closed balls are also order-convex.
        \item For any $a \in M$, $d(x,a)$ is nonincreasing on $(-\infty,a]$, and nondecreasing on $[a,\infty)$.
        \item The set $\{(x,y) : x \leq y\}$ is closed in the metric topology.
        \item The function $r(x,y)$ is $2$-Lipschitz and given by $$ r(x,y) = \begin{cases}
            0 & x \leq y \\
            d(x,y) & y \leq x.
        \end{cases}
        $$
        \item The identity $r(x,y) + r(y,x) = d(x,y)$ holds.
        \item For any $a \in M$, $r(x,a)$ is nondecreasing.
    \end{itemize}
\end{lem}
\begin{proof}
    Closed balls are intersections of families of order-convex open balls, and are thus order-convex.

    Fix $a \in M$. If $b \leq c \leq a$, then the closed ball of radius $d(a,b)$ around $b$ also contains $c$ by order-convexity, so $d(a,c) \leq d(a,b)$.

    Suppose that $(x_0,y_0) \in \{(x,y) : x > y\}$.
    For any $\delta < \frac{1}{2}d(x_0,y_0)$,
    as the open balls $B_\delta(x_0)$ and $B_\delta(y_0)$ are order-convex, do not intersect, and contain points $x_0,y_0$ respectively with $x_0 > y_0$,
    we see that every point of $B_\delta(x_0)$ is greater than every point of $B_\delta(y_0)$, so 
    $$(x_0,y_0) \in B_{\delta}(x_0) \times B_{\delta}(y_0) \subseteq \{(x,y) : x > y\}.$$
    This shows that that $\{(x,y) : x > y\}$ is open, and $\{(x,y) : x \leq y\}$ is closed.
    
    Now we check claims about $r$.
    Clearly if $x \leq y$, then $r(x,y) = 0$. If instead $x > y$, then for any $z \leq y$, by monotonicity, $d(x,z) \geq d(x,y)$, so $d(x,y) = r(x,y)$.
    
    From this, we will check the 2-Lipschitz property. If $(x_0,y_0),(x_1,y_1)$ are pairs, then $|r(x_0,y_0)-r(x_1,y_1)|$ is either 0, $|d(x_0,y_0)-d(x_1,y_1)|$, or $d(x_i,y_i)$ for $i = 0,1$.
    As $d(x,y)$ is itself 2-Lipschitz, without loss of generality we can assume this is $d(x_0,y_0)$ in the case where $x_0 > y_0$ and $x_1 \leq y_1$.
    If $x_1 \leq y_0$, then $d(x_0,x_1) \geq d(x_0,y_0)$, so $d(x_0,y_0) \leq d((x_0,y_0),(x_1,y_1))$.
    If $x_1 \geq y_0$, then $y_0 \leq x_1 \leq y_1$ so by monotonicity $d(y_0,x_1) \leq d(y_0,y_1)$, so
    $$d(x_0,y_0) \leq d(x_0,x_1) + d(x_1,y_0) \leq d(x_0,x_1) + d(y_0,y_1) \leq 2d((x_0,y_0),(x_1,y_1)).$$

    The remaining claims follow from the piecewise form of $r$ and the fact that $d(x,a)$ is nondecreasing on $[a,\infty)$.
\end{proof}

Before giving a piecewise description of $d_\leq$, we introduce a notation that will also be useful when discussing metric cyclic orders:
\begin{defn}\label{defn:d_delta}
    In any metric language, let $d_\Delta$ be the formula $$d_\Delta(x,y) = \inf_z(\max(d(x,z),d(y,z)))$$
    which is the distance predicate to the definable set $\Delta = \{(a,a) : a \in M\}$, the \emph{diagonal} of $M^2$.
\end{defn}
\begin{lem}
    In any metric structure, for all $x,y$, $d(x,y) \leq d_\Delta(x,y)$, and $M$ is an ultrametric if and only if for all $x,y$, $d_\Delta(x,y) = d(x,y)$.
\end{lem}
\begin{proof}
    We see that $d(x,y) \leq d_\Delta(x,y)$ by observing that 
    $$\inf_z(\max(d(x,z),d(y,z))) \leq \max(d(x,y),d(y,y)) = d(x,y).$$
    In the ultrametric case, we see that for any $z$, $\max(d(x,z),d(y,z)) \geq d(x,y)$, so $$\inf_z(\max(d(x,z),d(y,z))) \geq d(x,y).$$
    The statement that for all $x,y$, $\inf_z(\max(d(x,z),d(y,z))) \geq d(x,y)$ also states that for all $x,y,z,$ $\max(d(x,z),d(y,z)) \geq d(x,y)$, so this statement is equivalent to the ultrametric inequality.
\end{proof}

\begin{lem}\label{lem:dleq_c}
    In any metric linear order $M$, then for all $a,b \in M$,
    $$d_\leq(a,b)
    = \inf_{z}\max(r(a,z),r(z,b))
   = \begin{cases}
        0 & \textrm{if }a \leq b\\
        d_\Delta(a,b) & \textrm{if }b \leq a
    \end{cases}
    $$
\end{lem}
\begin{proof}
    If $a \leq b$, the first two of these of these are clearly 0.
    We also find, by plugging in $z = a$,
    that $$\inf_{z}\max(r(a,z),r(z,b)) \leq \max(r(a,a),r(a,b)) = 0.$$

    As $\Delta \subseteq \{(x,y): x \leq y\}$, we know that
    $d_\leq(a,b) \leq d_\Delta(a,b)$.
    
    If $b < a$, then assume that $(c,d) \in \{(x,y): x \leq y\}$ minimizes $d((a,b),(c,d))$.
    
    We do not increase the distance by replacing $(c,d)$ with $(a,d)$ or $(c,b)$, as long as these points are in $\{(x,y): x \leq y\}$.
    Thus we may assume that $b \leq c \leq d \leq a$.
    If this is the case, then $d(b,c) \leq d(b,d)$, so $d((a,b),(c,d)) \geq d((a,b),(c,c))$, so we can assume that $c = d$. Thus $d_\leq(a,b) = d((a,b),(c,c)) \geq d_\Delta(a,b)$ for some $c \in [b,a]$.

    We have seen that $d_\leq(a,b) = \inf_{z \in [b,a]}\max(d(a,z),d(z,b)) = \inf_{z \in [b,a]}\max(r(a,z),r(z,b))$, as for any $c' \in [b,a]$, $r(a,c') = d(a,c')$ and $r(c',b) = d(c',b)$.
    Fix $c' \not \in [b,a]$, and assume without loss of generality that $c' < b$. we find that $r(a,c') = d(a,c')$ while $r(c',b) = 0$, so $\max(r(a,c'),r(c',b)) = d(c',a) \geq d(b,a) \geq d_\leq(a,b)$ by convexity. Thus the infimum does not change when we extend to $z \in M$ instead of $z \in [b,a]$, and 
    $$\inf_{z}\max(r(a,z),r(z,b)) = \inf_{z \in [b,a]}\max(r(a,z),r(z,b)) = d_\leq(a,b).$$
\end{proof}

\begin{lem}
    Let $M$ be both an $\mathcal{L}$-metric structure for some language $\mathcal{L}$ and a metric linear order, with the same metric.
    If any one of $d_\leq,r$ is a definable predicate, then so is the other one.
\end{lem}
\begin{proof}
    If $r$ is definable, then $\inf_{z}\max(r(x,z),r(z,y))$ is definable, and this is $d_\leq(x,y)$ by Lemma \ref{lem:dleq_c}.

    If $d_\leq$ is definable, then 
    consider $\phi(x,y) = \inf_{w,z : w \leq z}d(x,w) + d(z,y)$ is a definable predicate, and we claim that $\phi(x,y) = r(x,y)$.
    
    If $a \leq b$, then $\phi(a,b) = 0 = r(a,b)$.
    If $a > b$, then by setting $w = z = a$ we see that this is at most $d(a,b) = r(a,b)$.
    For any $c \leq d$, the intervals $[c,a]$ and $[b,d]$ must intersect, and if $e$ is in the intersection, we see that $d(a,c) \geq d(a,e)$ and $d(b,d) \geq d(b,e)$ so $d(a,c) + d(b,d) \geq d(a,e) + d(b,e) \geq d(a,b)$, so $\phi(x,y) \geq d(a,b)$.
\end{proof}

As in a reasonable metric theory of linear orders, $\{(x,y) : x \leq y\}$ should be definable.
As both $d_\leq, r$ will be definable, we include them all in our language, so that they are both quantifier-free definable:
\begin{defn}
    Let $\Lleq$ be the metric language consisting of the binary predicates $\{d_\leq,r\}$.
\end{defn}

\subsection{Axiomatizing The Theory of Metric Linear Orders}

\begin{defn}
Let $\MLO{}$ be the $\Lleq$-theory consisting of the following conditions:
\begin{enumerate}[label=(\arabic*)]
\item\label{axiom:asymm} $\sup_{x,y} |(r(x,y) + r(y,x)) - d(x,y)| = 0$ 
\item\label{axiom:total} $\sup_{x,y} \min\{r(x,y), r(y,x)\} = 0$
\item\label{axiom:trans} $\sup_{x,y,z} r(x,z)\ \dot- \ (r(x,y) + r(y,z)) = 0$ 
\item\label{axiom:dleq} $\sup_{x,y}|d_\leq(x,y) - \inf_{z}\max(r(x,z),r(z,y))| = 0$.
\end{enumerate}
\end{defn}
Conditions \ref{axiom:asymm} and \ref{axiom:total} together imply that exactly one of $r(x,y)$ and $r(y,x)$ is equal to $d(x,y)$ and the other one is equal to zero. Hence our ordering is \textbf{total}. Condition \ref{axiom:trans} is a version of the triangle inequality that gives us the following properties:
\begin{itemize}
\item \textbf{Transitivity:} If $x\leq y$ and $y\leq z$, then $r(x,y)=r(y,z)=0$, so condition \ref{axiom:trans} forces that $r(x,z)=0$. Thus, $x\leq z$. 
\item \textbf{Monotonicity:} If $y\leq z\leq x$, then the condition forces $r(x,z) \leq r(x,y)$. Meaning that $r$ is nondecreasing on $(-\infty, x]$.
\end{itemize}
Note that $r$ is not symmetric, so condition \ref{axiom:trans} is not a full triangle inequality condition.\par
One can think of $r$ as an ``asymmetrization" of the distance predicate, where for each pair of elements $x$ and $y$, we pick one of the pairs $(x,y)$ and $(y,x)$ to give the full weight of the distance $d(x,y)$ and zero to the other, and do so in a way that is transitive and monotonic. 

\begin{remark}
Let $(M;\leq)$ is a discrete structure equipped with a linear order $\leq$. Now consider $(M;\leq)$ as a continuous logic structure with the discrete $\{0,1\}$-metric, and with $0$ being true and $1$ false. Then $\leq$ satisfies the axioms for $r$ in the previous definition.
\end{remark}

\begin{thm}\label{thm:mlo_models}
The models of $\MLO{}$ are exactly the complete, bounded metric linear orders.
\end{thm}

\begin{proof}
Suppose $(M,d)$ is complete, bounded metric space with a linear order $\leq$ such that each open ball in the metric is order convex. Then interpreting $d_\leq,r$ as in \ref{meta-def} yields a metric structure in $\Lleq$. Conditions \ref{axiom:asymm} and \ref{axiom:total} follow directly from \ref{meta-basic}. Condition \ref{axiom:trans} is justified as follows by cases: 
\begin{itemize}
\item $x\leq y\leq z$: then all the $r$ predicates in condition \ref{axiom:trans} are zero, and the condition holds.
\item $x\leq z\leq y$, and $y\leq x\leq z$: since $x\leq z$, the left-hand side is zero, and in the right-hand side we have either zero or some distance predicate, which is at least zero.
\item $z\leq y\leq x$: then all the $r$ predicates are actually equal to the distance, and the condition follows from the triangle inequality.
\item $y\leq z\leq x$: the condition simplifies to $d(x,z)\leq d(x,y)$. This equality follows from the fact that $d(x,\cdot)$ is nonincreasing on $(-\infty, x]$ by \ref{meta-basic}.
\item $z\leq x\leq y$:  the condition simplifies to $d(x,z)\leq d(y,z)$. This equality follows from the fact that $d(z,\cdot)$ is nondecreasing on $[z,\infty)$ by \ref{meta-basic}.
\end{itemize}
Once this is done, condition \ref{axiom:dleq} is satisfied by Lemma \ref{lem:dleq_c}.

On the other hand, suppose $(M; d, d_\leq, r)\models \MLO{}$. Define a linear order $\leq$ on $M$ by saying that $x\leq y$ iff $r(x,y)=0$. We need to show that $\leq$ is a linear order and that every open ball in the metric is order-convex.\par
As for the axioms of linear orders: reflexivity follows from the fact that $d(x,x)=0$, transitivity from condition \ref{axiom:trans}, totality from condition \ref{axiom:total}. Antisymmetry follows from condition\ref{axiom:asymm}: if $x\leq y$ and $y\leq x$ then $d(x,y)=r(x,y)+r(y,x)=0$, so $x=y$. Here we need that $d$ be a metric.\par
Fix $a\in M$ and $\delta>0$. We need to show that $B_\delta(a)$ is order-convex. Fix $x,y\in B_\delta(a)$ and $z\in M$ such that $x\leq z\leq y$. If $a\leq z$, then $a\leq z\leq y$ and, since $d(a,y)<\delta$, condition \ref{axiom:trans} gives us that
$$d(a,z)=r(z,a)\leq r(z,y)+r(y,a) = d(a,y)<\delta.$$
The case $a\geq z$ is handled similarly by noticing that $x\leq z\leq a$ and using condition \ref{axiom:trans}.
\end{proof}

\subsection{Topologies}
\begin{lem}\label{lem:top_refine}
    In any metric linear order $M$, the metric topology refines the order topology.
\end{lem}
\begin{proof}
    As seen in Lemma \ref{meta-basic}, the set $\{(x,y) : x > y\}$ is metric-open. The rays $(-\infty,a),(a,\infty)$ are preimages of this set under the maps $x \mapsto (x,a)$ and $y \mapsto (a,y)$, so they are also metric-open, as are all open intervals, and in turn, any order-open set.
\end{proof}

This means that the identity function from $M$ equipped with the metric topology to $M$ equipped with the order topology is continuous, allowing us to make the following translations between the two topologies:

\begin{cor}
    Suppose $M$ is metric linear order and $X \subseteq M$.

    If $X$ is open or closed in the order topology, it is in the metric topology as well.
    
    If $X$ is compact or connected in the metric topology, it is in the order topology as well.
\end{cor}

\begin{lem}\label{lem:top_agree}
    In a metric linear order $M$ which is a dense linear order without endpoints, the metric and order topologies agree if and only if for every $a \in M$ and $\varepsilon > 0$, there are $b< a < c$ with $d(b,a),d(a,c) < \varepsilon$.
\end{lem}
\begin{proof}
    By Lemma \ref{lem:top_refine}, the topologies agree if and only if every open ball is order-open.

    This is equivalent to saying for every $a \in M$, $\varepsilon > 0$, $a$ is contained in an open interval $(b,c)$ which is in turn contained in the ball. This happens precisely when $b < a < c$ and $d(b,a),d(a,c) < \varepsilon$.
\end{proof}

This condition on the topologies agreeing is, perhaps inconveniently, not elementary.

\begin{eg}
    If $M$ is a metric linear order which is a dense linear order without endpoints where the metric and order topologies agree, and $n \in \N$, define $M_n$ to be a copy of $M$ with the same order, where all distances are scaled down by a factor of $\frac{1}{n}$.

    The sequence $(M_n : n \in \N)$ is a sequence of metric linear orders which are dense linear orders without endpoints where the metric and order topologies agree, but their ultraproduct is a single point.
\end{eg}
\begin{proof}
    The diameter of each $M_n$ is $\frac{1}{n}$, so the ultraproduct satisfies the condition $\sup_{x,y}d(x,y) = 0$.
\end{proof}

As a way around this, we propose the following theory:
\begin{defn}\label{defn:mdlo}
    Let \MDLO{} be the theory of \emph{metric dense linear orders} consisting of the two countable families of axioms \begin{itemize}
        \item For any rational $p \in \Q \cap [0,1]$,
        $\sup_x\inf_y |r(x,y) - p| = 0$
        \item For any rational $p \in \Q \cap [0,1]$,
        $\sup_x\inf_y |r(y,x) - p| = 0$.
    \end{itemize}
\end{defn}

The name of this theory is justified by the following two density statements, which together imply that the metric and order topologies agree in models of \MDLO{}:

\begin{lem}\label{lem:dense_ordered_distances}
    The models of \MDLO{} are precisely the metric linear orders $M$ such that for every $a \in M$, the sets $\{d(a,b): b < a\}$ and $\{d(a,b): b > a\}$ are both dense in $[0,1]$.
\end{lem}
\begin{proof}
    For any rational $p$, $M \vDash \sup_x\inf_y |r(x,y) - p| = 0$ if and only if for all $a \in M$,
    $M \vDash \inf_y |r(a,y) - p| = 0$. This means that for any $\varepsilon > 0$, there is some $b \in M$ with $|r(a,b) - p| < \varepsilon$. For $p > 0$, we can choose $0 < \varepsilon < p$, and see that this implies $r(a,b) > 0$, so $b < a$. This means that there are values $b \in M$ with $b < a$ and thus $d(a,b) = r(a,b)$, with $r(a,b)$ arbitrarily close to $p$ - when we consider this for all $p$, this is equivalent to the set $\{d(a,b): b < a\}$ being dense in $[0,1]$.
    
    A symmetric argument shows that the remaining family of axioms is equivalent to the set $\{d(a,b): a < b\}$ being dense in $[0,1]$.
\end{proof}

\begin{lem}\label{lem:mdlo_dlo}
    Any model of \MDLO{} is a dense linear order without endpoints.
\end{lem}
\begin{proof}
    Suppose $M \vDash \MDLO{}$ and $a < b$ are in $M$. Then by Lemma \ref{lem:dense_ordered_distances}, there is some $c > a$ with $d(a,c) < d(a,b)$, so $a < c < b$ as $M \vDash \MLO{}$.

    Lemma \ref{lem:dense_ordered_distances} clearly contradicts $M$ having endpoints.
\end{proof}

\section{The Ultrametric Case}\label{sec:ultra}
In this section, we draw our attention to the case of metric linear orders whose metrics are ultrametrics. In particular, we study the model companion of the theory of ultrametric linear orders, which we call the theory of ultrametric dense linear orders. First, though, we review some background on quantifer elimination and model companions in contnuous logic:

\subsection{Quantifer Elimination and Model Companions}

We will show that \DU{} is in fact complete and thus equivalent to $\mathrm{Th}(U_S)$, by showing that it is also the model companion of \UM{}. We take this approach, as we will want to take the same approach to finding the model companion of the theory of ultrametric linear orders. The blueprint we use for finding model companions in continuous logic is the one used in \cite{rtrees}, as $\R$-trees and ultrametrics are inherently related. We start by recalling the definition of existentially closed models in continuous logic, and then work to show that the existentially closed models of \UM{} are precisely the models of \DU{}.

\begin{defn}[{\cite[Defn. 3.4]{rtrees}}]
    If $T$ is a theory, a model $M \vDash T$ is an \emph{existentially closed} model of $T$ if and only if for all $N \vDash T$ with $M \subseteq N$, and all quantifier-free formulas $\phi(\bar x,\bar y)$, and tuples $\bar a \in M^m$, if 
    $$\psi(\bar x) = \inf_{\bar y}\phi(\bar x,\bar y),$$
    then
    $$\psi^M(\bar a) = \psi^N(\bar a).$$
\end{defn}

\begin{fact}[{\cite[Prop. 3.5]{rtrees}}]
    A theory $T$ is model complete if any embedding between model of $T$ is an elementary embedding. This is equivalent to the condition that every model of $T$ is existentially closed.
\end{fact}

\begin{defn}[{\cite[Defn. 3.4]{rtrees}}]
    Let $T$ be a theory. A theory $S$, in the same language, is the model companion of $T$ if:
    \begin{itemize} 
    \item Every model of $T$ embeds into a model of $S$.
    \item Every model of $S$ embeds into a model of $T$.
    \item $S$ is model complete.
    \end{itemize}
\end{defn}

\begin{defn}
   A theory $T$ has \emph{amalgamation of substructures} when $M_0, M_1, M_2$ are substructures of models of $T$, and $f_1 : M_0 \hookrightarrow M_1$, $f_2 : M_0 \hookrightarrow M_2$ are embeddings, then there is a model $N \vDash T$ with embeddings $g_1 : M_1 \hookrightarrow N, g_2 : M_2 \hookrightarrow N$ such that $g_1 \circ f_1 = g_2 \circ f_2$.
\end{defn}

\begin{fact}[{\cite[Prop. 3.8]{rtrees}}]\label{fact:model_companion_qe}
    Assume that $T_1$ and $T_2$ are theories in the same language such that $T_2$ is the model companion of $T_1$ and $T_1$ has amalgamation over substructures. Then $T_2$ has quantifier elimination.
\end{fact}

We now prove a test that we will use to check existential closedness.
\begin{lem}\label{lem:ec_test}
    If $T$ is a theory, a model $M \vDash T$ is an existentially closed model of $T$ if for every model $N \vDash T$ with $M \subseteq N$, all $\bar a \in M^m$, $\bar b \in N^n$, all $\delta > 0$, there exists $\bar c \in M^n$ such that for all atomic formulas $\theta(\bar x,\bar y)$,
    $$|\theta(\bar a, \bar b) - \theta(\bar a,\bar b)| < \delta.$$
\end{lem}
\begin{proof}
    Let $\psi(\bar x) = \inf_{\bar y}\phi(\bar x, \bar y),$
    where $\phi$ is quantifier-free.
    Fix $\bar a$.
    For all $\varepsilon > 0$, we will show that $|\psi^M(\bar a) - \psi^N(\bar a)| < \varepsilon$.
    Fix $\bar b$ such that $\phi(\bar a,\bar b) \leq \psi^N(\bar a) + \frac{\varepsilon}{2}$.
    As $\phi$ is a continuous combination of atomic formulas, there is some $\delta > 0$ such that if for all atomic formulas $\theta(\bar x,\bar y)$, if $\bar c$ is such that $|\theta(\bar a,\bar b) - \theta(\bar a,\bar c)|<\delta$, then $|\phi(\bar a,\bar b) - \phi(\bar a,\bar c)| < \frac{\varepsilon}{2}$.
    By assumption, there is some $\bar c$ meeting this hypothesis, so 
    $\psi^M(\bar a) \leq \phi(\bar a,\bar c) < \psi^N(\bar a) + \varepsilon$.
\end{proof}

\subsection{Ultrametric Metric Structures}
Recall that a metric space is \emph{ultrametric} if it satisfies
$$d(x,z) \leq \max(d(x,y),d(y,z))$$
for all $x,y,z$. It is straightforward to check that this property can be represented with a condition in continuous logic, so metric structures whose metrics are ultrametric are axiomatized by the following theory:
\begin{defn}
    Let \UM{} be the theory, in the empty language of metric structures, of \emph{ultrametric spaces}, consisting of the ultrametric axiom
    $$\sup_{x,y,z}d(x,z) \dot{-} \max(d(x,y),d(y,z))= 0.$$
\end{defn}

We particularly care about a particular theory of ultrametric spaces, which we will later combine with an order. This theory was studied in \cite{conant_um} as the complete theory of a particular ultrametric space constructed in \cite{gao_shao_ultra}:
\begin{defn}
    If $S \subset (0,1]$ is countable and dense, let $U_{S}$ be the set $\{f \in \Q^S : \forall s > 0, \{x \geq s : f(x)\neq 0\}\textrm{ is finite.}\}$
    Give $U_S$ the metric
    $$d(f,g) = \sup\{x : f(x) \neq g(x)\}.$$
\end{defn}
\begin{fact}[\cite{conant_um}]\label{fact:ThUmax}
    The following are true of the theory $\mathrm{Th}(U_S)$ for any countable dense $S \subset (0,1]$:
    \begin{itemize}
        \item $\mathrm{Th}(U_S)$ is complete
        \item $\mathrm{Th}(U_S)$ eliminates quantifiers
        \item Any model of $\mathrm{Th}(U_S)$ is an ultrametric space
        \item Any model of \UM{} embeds (isometrically or equivalently, model-theoretically) into a model of $\mathrm{Th}(U_S)$.
    \end{itemize}
\end{fact}
Because $\mathrm{Th}(U_S)$ eliminates quantifiers, it is model complete, which together with the fact that any ultrametric space embeds into a model, tells us that this is the model companion of \UM{}. As this theory is unique, all structures $U_S$ are elementarily equivalent.
\begin{cor}
    The theory \UM{} has a model companion, which is the theory $\mathrm{Th}(U_S)$ for any countable dense $S \subset (0,1]$.
\end{cor}
We will simply refer to this theory as $\mathrm{Th}(U_S)$.

We propose the following theory, which will turn out to be another axiomatization of $\mathrm{Th}(U_S)$:
\begin{defn}
    Let \DU{} be the theory, in the empty language of metric structures, extending \UM{} by the following countable family of axioms: For every rational $p \in \Q \cap [0,1]$, include the axiom
    $$\sup_x\inf_y|d(x,y) - p| = 0.$$
\end{defn}
We call models of this theory \emph{dense ultrametrics}, because of the following observation:
\begin{lem}
    IF $M$ is a metric structure in the empty language, then $M \vDash \DU{}$ if and only if for every $a \in M,$ the set $\{d(a,b) : b \in M\}$ is dense in $[0,1]$.
\end{lem}
(The proof is a simplification of that of Lemma \ref{lem:dense_ordered_distances} below.)

Firstly, we check that $U_S \vDash \DU{}$:
\begin{lem}\label{lem:du_example}
    If $S \subset (0,1]$ is countable and dense, then for any $f \in U_S,$
    $\{d(f,g) : g \in U_S\} = \Q \cap [0,1]$, and in particular, $U_S \vDash \DU{}.$
\end{lem}
\begin{proof}
    This is a slight generalization of \cite[Proposition 2.7]{conant_um}.
    Fix $f \in U_S$.

    For any $s \in S$, we can define $g \in U_S$ such that $f(x) = g(x)$ if and only if $x \neq s$, so $d(f,g) = \sup\{s\} = s$.
\end{proof}

We will prove existential closure through a series of lemmas which we can reuse when linear orders are present.
\begin{lem}\label{lem:um_max_closest}
    Let $M$ be an ultrametric space, let $A \subseteq M$, let $b \in M$, and assume $a_0 \in A$ is such that $d(a_0,b) \leq d(a,b)$ for all $a \in A$.

    Then for all $a \in A$, $d(a,b) = \max\left(d(a,a_0),d(a_0,b)\right)$.
\end{lem}
\begin{proof}
    By the ultrametric inequality, if $x,y,z$ are in an ultrametric space and $d(x,y) < d(x,z)$, then $d(y,z) = d(x,z)$.
    Generalizing, if $d(x,y) \neq d(x,z)$, then $d(y,z) = \max\left(d(x,y),d(x,z)\right)$.

    In our case, if $a \in A$ is such that $d(a,a_0)\neq d(a_0,b)$, then the result follows.
    Otherwise, if $d(a,a_0)\neq d(a_0,b)$, we find that $d(a,b) \leq \max\left(d(a,a_0),d(a_0,b)\right) = d(a_0,b)$. By the assumption of minimality of $d(a_0,b)$, we have $d(a,b) = d(a_0,b) = \max\left(d(a,a_0),d(a_0,b)\right)$.
\end{proof}

\begin{defn}\label{isometric}
    If $M, N$ are ultrametric spaces, $\varepsilon > 0$, and $(a_1,\dots,a_n) \in M^n, (b_1,\dots,b_n) \in N^n$ are tuples of the same length, call these tuples $\varepsilon$-\emph{isometric} if for all $1 \leq i,j \leq n$, $|d(a_i,a_j) - d(b_i,b_j)|< \varepsilon$.
\end{defn}
\begin{lem}\label{lem:delta_isometric_induction}
    If $M \vDash \UM, N \vDash \DU$, $\varepsilon > 0$, and $\bar a \in M^n, \bar b \in N^n$ are $\varepsilon$-isometric tuples of the same length, then for every $a_{n+1} \in M$, there is an open ball $B\subset N$ such that for every $b_{n+1} \in B$, $(a_1,\dots,a_{n+1})$ and $(b_1,\dots,b_{n+1})$ are $\varepsilon$-isometric.

    In particular, if $a_{n+1} \not \in \{a_1,\dots,a_n\}$ and $1 \leq i_0 \leq n$ minimizes $d(a_{i_0},a_{n+1})$,
    then there is an open set $D \subseteq (0,1)$ such that for any $b_{n+1} \in N$ with $d(b_{i_0},b_{n+1}) \in D$,
    it will also be true that for all $1 \leq i \leq n$, $d(b_{i_0},b_{n+1})\leq d(b_i,b_{n+1})$,
    and $(a_1,\dots,a_{n+1})$ and $(b_1,\dots,b_{n+1})$ will be $\varepsilon$-isometric.
\end{lem}
\begin{proof}
    Let $1 \leq i_0 \leq n$ minimize $d(a_{i_0},a_{n+1})$ (this need not be unique).
    If $a_{n+1} = a_{i_0}$, we may simply let $b_{n+1} = b_{i_0}$.
    Otherwise, let
    $$D = (d(a_{i_0},a_{n+1}) - \varepsilon, d(a_{i_0},a_{n+1}) + \varepsilon) \setminus
    \{d(b_i,b_{i_0}) : 1 \leq i \leq n\}.$$
    This is a nonempty open set, so by the density axioms of \DU{}, there is some $b^* \in N$ with $d(b_{i_0},b^*) \in D$. Let $t$ be strictly smaller than all the distances mentioned in the construction so far, and consider $B=B_t(b^*)\subset N$. Note that for every $b_{n+1}\in B$ we have that $d(b_i,b_{n+1})=d(b_i,b^*)$ for $1\leq i\leq n$ by the ultrametric inequality.

    Now fix $b_{n+1}\in B$. First, we establish that for all $1 \leq i \leq n$, $d(b_{i_0},b_{n+1}) \leq d(b_i,b_{n+1})$.
    By construction, $d(b_{i_0},b_{n+1}) \neq d(b_i,b_{i_0})$, so the ultrametric inequality will be an equality, and
    $$d(b_i,b_{n+1}) = \max\left(d(b_{i_0},b_{n+1}) ,d(b_i,b_{i_0})\right) \geq d(b_{i_0},b_{n+1}).$$

    Thus by Lemma \ref{lem:um_max_closest}, for all $1 \leq i \leq n$,
    we have 
    \begin{align*}
        d(b_i,b_{n+1}) &= \max\left(d(b_i,b_{i_0}),d(b_{i_0},b_{n+1})\right)\\
        d(a_i,a_{n+1}) &= \max\left(d(a_i,a_{i_0}),d(a_{i_0},a_{n+1})\right),
    \end{align*}
    so as
    $$|d(a_i,a_{i_0}) - d(b_i,b_{i_0})|, |d(a_{i_0},a_{n+1}) - d(b_{i_0},b_{n+1})| < \varepsilon,$$
    and the function $\max(\cdot,\cdot)$ is 1-Lipschitz,
    we have
    $|d(a_i,a_{n+1}) - d(b_i,b_{n+1})| < \varepsilon$ as well.
\end{proof}

\begin{thm}\label{thm:du_ec}
    The existentially closed models of \UM{} are precisely the models of \DU{}.
\end{thm}
\begin{proof}
    Let $M \vDash \UM{}$ be an existentially closed model of \UM{}, and consider the quantifier-free formula $\phi(x,y) = |d(x,y) - q|$ for some $q\in [0,1]$.
    By Fact \ref{fact:ThUmax}, $M$ embeds into some $N \vDash \mathrm{Th}(U_S)$, and by Lemma \ref{lem:du_example}, $N \vDash \DU{}$.
    Fix $a \in M$.
    By the axioms of $\DU{}$, $\inf_{y \in N} \phi(a,y) = 0$, so by existential closure, $\inf_{y \in M} \phi(a,y) = 0$ as well, so $M \vDash \DU{}$.

    Conversely, we use Lemma \ref{lem:ec_test} to show that any $M \vDash \DU$ is an existentially closed model of $\UM{}$. Let $N \vDash \UM{}$ extend $M$, and fix $\varepsilon > 0$. We claim that for all $\bar a \in M^m$, $\bar b \in N^n$, there is $\bar c \in M^n$ such that for all $1 \leq i \leq m, 1 \leq j \leq n$, $|d(a_i,b_j) - d(a_i,c_j)| < \varepsilon$,
    and for all $1 \leq i,j \leq n$, $|d(b_i,b_j) - d(c_i,c_j)| < \varepsilon$.
    In other words, we need to find $\bar c$ such that the tuples $(\bar a,\bar b)$ and $(\bar a,\bar c)$ are $\varepsilon$-isometric.
    This follows by induction on the length of $\bar b$. The base case, that $\bar a$ and $\bar a$ are $\varepsilon$-isometric is trivial, and Lemma \ref{lem:delta_isometric_induction} provides the inductive step.
\end{proof}

We finish the model companion proof with the following fact, which is stated for inductive theories as \cite[Prop. 3.7]{rtrees}. For the proof in classical logic, see \cite[Theorem 8.3.6]{hodges}.
\begin{fact}\label{fact:universal_model_companion}
    If $T$ is a universal theory (in continuous logic, axiomatized by $\sup$-formulas), and another theory $T'$ axiomatizes the existentially closed models of $T$, then $T'$ is the model companion of $T$.
\end{fact}

This completes our proof that \DU{} is the model companion of \UM{}, which implies that it is equivalent to $\mathrm{Th}(U_S)$.
\begin{cor}\label{cor:du_model_companion}
    The theory \DU{} is complete, eliminates quantifiers, is the model companion of \UM{}, and is equivalent to $\mathrm{Th}(U_S)$ for any countable dense $S \subset (0,1]$.
\end{cor}

\subsection{Ultrametric Linear Orders}
\begin{defn}
    Let \ULO{} be the $\Lleq$-theory of \emph{ultrametric linear orders} consisting of \MLO{} together with the ultrametric axiom \UM{}.
\end{defn}

One advantage of working with ultrametrics is that the language can be simplified:
\begin{lem}\label{lem:ulo_dist_ray}
    If $M$ is an ultrametric linear order, then $d_\leq = r$.
\end{lem}
\begin{proof}
    By the ultrametric inequality, for any $a,b \in M$, $d_\Delta(a,b) = d(a,b)$. Thus by Lemma \ref{lem:dleq_c},
    $$d_\leq(a,b)
    = \begin{cases}
        0 & \textrm{if }a \leq b\\
        d_\Delta(a,b) & \textrm{if }b \leq a
    \end{cases}
    = \begin{cases}
        0 & \textrm{if }a \leq b\\
        d(a,b) & \textrm{if }b \leq a
    \end{cases}
    = r(a,b).$$
\end{proof}

The ultrametric triangle inequality itself interacts with the ordering in a reliable way:
\begin{lem}\label{lem:ultra_ordered_triangle}
    If $M$ is an ultrametric linear order, and $a,b,c \in M$ satisfy $a \leq b \leq c$, then $d(a,c) = \max((d(a,b),d(a,c))$.
\end{lem}
\begin{proof}
    Monotonicity already tells us that $d(a,b),d(b,c) \leq d(a,c)$, so $\max(d(a,b),d(b,c))\leq d(a,c)$.
    Conversely, the ultrametric triangle inequality states that $d(a,c) \leq \max(d(a,b),d(b,c))$, so 
    $d(a,c) = \max(d(a,b),d(b,c))$.
\end{proof}

We can induct on this lemma to produce the following generalization, which makes it straightforward to characterize finite ultrametric linear orders:
\begin{lem}\label{lem:ultra_ordered_sequence}
    If $M$ is an ultrametric linear order, and 
    $a_1 < a_2 < \dots < a_n$ is an increasing sequence in $M$ with $n \geq 2$, then $d(a_1,a_n) = \max_{1 \leq i < n} d(a_i,a_{i + 1})$.
\end{lem}
\begin{proof}
    This is trivial for $n = 2,$ and for $n = 3$, it is Lemma \ref{lem:ultra_ordered_triangle}.
    Given a sequence $a_1 < a_2 < \dots < a_n$ for which this holds, if $a_{n+1} > a_n$, then Lemma \ref{lem:ultra_ordered_triangle} shows that
    \begin{align*}
        d(a_1,a_{n+1}) 
        &= \max(d(a_1,a_n),d(a_n,a_{n+1})\\
        &= \max\left(\max_{1 \leq i < n} d(a_i,a_{i + 1}),d(a_n,a_{n+1})\right)\\
        &=\max_{1 \leq i < n + 1} d(a_i,a_{i + 1}),
    \end{align*}
    using our induction hypothesis to finish the induction step.
\end{proof}

To prove one more property of \ULO{}, useful in understanding its model companion, we will cite the following fact about finite models of \ULO{}. This was shown in the context of classical structural Ramsey theory, without any continuous logic.
\begin{fact}{{\cite[Theorem 2]{vanthe1}}}
    If $S \subset (0,1]$ is countable and dense, then the set of finite ultrametric linear orders with distances in $S$ has the Ramsey property.
\end{fact}
Rather than define the Ramsey property itself, we note the following consequence, which follows easily due to \cite[Theorem 4.2]{nesetril2005ramsey}, noting that the hereditary property and joint embedding property are easy to check. This amalgamation property is fundamentally a fact about embeddings, and in both the classical and the continuous logic setting, an embedding between metric linear orders is just an order-preserving isometry. Thus this amalgamation property translates from the classical logic setting of \cite{vanthe1,nesetril2005ramsey} to ours:
\begin{cor}\label{cor:fin_ulo_amalgamation}
    If $S \subset (0,1]$ is countable and dense, then the set of finite ultrametric linear orders with distances in $S$ is an amalgamation class, meaning that for any such structures $M_0,M_1,M_2$ and embeddings $f_1 : M_0 \hookrightarrow M_1$, $f_2 : M_0 \hookrightarrow M_2$, there is a finite ultrametric linear order $N$ with distances in $S$ with embeddings $g_1 : M_1 \hookrightarrow N, g_2 : M_2 \hookrightarrow N$ with $g_1 \circ f_1 = g_2 \circ f_2$.
\end{cor}

We extrapolate from this to amalgamation over substructures:
\begin{lem}\label{lem:ulo_amalgamation}
    The theory \ULO{} has amalgamation over substructures.
\end{lem}
\begin{proof}
    First, we observe that \ULO{} is a universal theory, so its substructures are also models of \ULO{}.

    The model $N$ we seek is precisely a model of \ULO{} that also satisfies the atomic diagrams of the structures $M_0, M_1, M_2$, as well as sentences identifying the constant representing $m \in M_0$ with the constants representing $f_1(m),f_2(m)$. We thus only need to show that this theory is satisfiable, or finitely satisfiable. To show finite satisfiability, it will suffice to reduce to the portion of this theory containing \ULO{} and the sentences pertaining to only finitely many constants representing elements of $M_0, M_1, M_2$. As $\Lleq$ is a relational language, this amounts to assuming that $M_0, M_1, M_2$ are finite.

    To consider finite structures, we turn to Corollary \ref{cor:fin_ulo_amalgamation}.
    For any finite ultrametric linear orders $M_0, M_1, M_2$, let $S \subset (0,1]$ be countable and dense, containing all distances between pairs in $M_0,M_1,M_2$. By Corollary \ref{cor:fin_ulo_amalgamation}, there is some finite ultrametric linear order $N$ with embeddings $g_1 : M_1 \hookrightarrow N, g_2 : M_2 \hookrightarrow N$ with $g_1 \circ f_1 = g_2 \circ f_2$.
\end{proof}

\subsection{Ultrametric Dense Linear Orders}
We now study a theory of \emph{ultrametric dense linear orders}, whose orders will be dense linear orders, and whose metrics will be dense ultrametrics modelling \DU{}.
\begin{defn}
    Let \UDLO{} be the $\Lleq$-theory of \emph{ultrametric dense linear orders} consisting of \ULO{} together with \MDLO{}.
\end{defn}

We then see that the models of \UDLO{} are precisely the ultrametric linear orders $M$ such that for every $a \in M$, the sets $\{d(a,b): b < a\}$ and $\{d(a,b): b > a\}$ are both dense in $[0,1]$.

By ignoring the ordering in the distance sets in Lemma \ref{lem:dense_ordered_distances}, we get this corollary: 
\begin{cor}
    The reduct of any model of \UDLO{} to the empty language satisfies \DU{}.
\end{cor}

To make this theory more concrete, we provide an example, by showing how $U_S \vDash \DU{}$ (see Lemma \ref{lem:du_example} can be expanded to make a model of \UDLO{}. This metric linear order first appears in \cite{vanthe1}.
\begin{lem}
    For countable dense $S \subset (0,1]$, give $U_S$ the ``lexicographical'' order where $f < g$ if and only if when $x$ is maximal such that $f(x) \neq g(x)$, $f(x) < g(x)$.

    Then $U_S \vDash \UDLO$.
\end{lem}
\begin{proof}
    The metric open ball around $f$ of radius $r$ is defined as all $g$ such that for $x \geq r$, $f(x) = g(x)$, which is convex in this lexicographical order.

    For any $f$, the set $\{g \in U_S : f < g\}$ is $S$, because for any $s \in S$, we may simply define $g \in U_S$ such that if $x \neq s$,  then $f(x) = g(x)$, but $g(s) = f(s) + 1$. Then $d(f,g) = \sup\{s\} = s$, while $f < g$. A similar argument shows $\{g \in U_S : g < f\} = S$.
\end{proof}

\begin{lem}\label{lem:order_isometric_induction}
    If $M,N\models \UDLO{}$, $\varepsilon > 0$, $n\in\N$, and $\bar a \in M^n, \bar b \in N^n$ are $\varepsilon$-isometric $n$-tuples with the same order type, then for every $a_{n+1} \in M$, there is an open ball $B\subset N$ such that for any $b_{n+1} \in B$ we have that $(a_1,\dots,a_{n+1})$ and $(b_1,\dots,b_{n+1})$ are $\varepsilon$-isometric with the same order type.

\end{lem}
\begin{proof}
The proof is quite similar to Lemma \ref{lem:delta_isometric_induction}, but we need to preserve the order type. As before, we let $i_0$ be such that $d(a_{i_0},a_{n+1})$ is as small as possible. Without loss of generality, let us assume that $a_{n+1}<a_{i_0}$ and that $i_1$ is such that $a_{i_1}$ is the largest element of $\bar a$ satisfying $a_{i_1}<a_{n+1}<a_{i_0}$. If no such element exists, the construction is easier. Consider now the open set

$$D = ((d(a_{i_0},a_{n+1}) - \varepsilon, d(a_{i_0},a_{n+1}) + \varepsilon) \cap (0, d(a_{i_0},a_{i_1}))\setminus
    \{d(b_i,b_{i_0}) : 1 \leq i \leq n\}.$$

as $a_{i_1}<a_{n+1}<a_{i_0}$, it follows that $d(a_{n+1},a_{i_0})\leq d(a_{i_1},a_{i_0})$, so this interval is nonempty. By \UDLO{}, the set $\{d(b_{i_0},x): x\in N, x < b_{i_0}\}$ is dense in $(0,1)$, so we can find $b^*\in N$ with the approximately correct distance and with the correct order type. We let $B$ be an open ball centered around $b^*$ with radius smaller than all the nonzero distances that are realized between elements of $\bar a$ and $\bar b$.
\end{proof}

\begin{thm}\label{thm:udlo_ec}
    The existentially closed models of \ULO{} are precisely the models of \UDLO{}.
\end{thm}
\begin{proof}
    Suppose that $M$ is an existentially closed model of \ULO{}. To show that $M \vDash \UDLO{}$, we will show that for any $a \in M$ and $q \in [0,1]$,
    that $M \vDash \inf_x|r(x,a) - q| = 0$. (The analogous result for the other axiom follows by the same argument.)
    To do this, we show that there is $N \vDash \ULO$ extending $M$ containing an element $b$ with $r(b,a) = q$ - then $N \vDash \inf_x|r(x,a) - q| = 0$, and so by existential closedness, $M \vDash \inf_x|r(x,a) - q| = 0$.
    We can construct $N$ through Lemma \ref{lem:ulo_amalgamation}, by amalgamating $M$ with the ultrametric linear order $\{a,b\}$ with $a < b$ and $d(a,b) = q$.

    It now suffices to show that for every model $N \vDash \ULO$ with $M \subseteq N$, all $\bar a \in M^m$, $\bar b \in N^n$, all $\varepsilon > 0$, there exists $\bar c \in M^n$ such that for all atomic formulas $\theta(\bar x,\bar y)$,
    $$|\theta(\bar a, \bar b) - \theta(\bar a,\bar b)| < \varepsilon.$$
    Because all atomic formulas may be assumed to be of the form $d(w,v)$ or $r(w,v)$, it suffices to find $\bar c$ such that $(\bar a, \bar b)$ and $(\bar a, \bar c)$ have the same order type and are $\varepsilon$-isometric.

    As in the proof of Theorem \ref{thm:du_ec}, we construct this by induction on the length of $\bar b$. As a base case, we see that $\bar a$ and $\bar b$ have the same order type and are $\varepsilon$-isometric.
    It now suffices to show that if $(\bar a, \bar b)$ and $(\bar a, \bar c)$ have the same order type and are $\varepsilon$-isometric, then for any $b_{n+1} \in N$, there is $c_{n+1} \in M$ such that $(\bar a,\bar b,b_{n+1})$ and $(\bar a, \bar c,c_{n+1})$ have the same order type and are $\varepsilon$-isometric.
    By Lemma \ref{lem:delta_isometric_induction}, for the $\varepsilon$-isometric condition, if $e \in \{a_i : 1 \leq i \leq m\} \cup \{b_j : 1 \leq j \leq n\}$ minimizes $d(e,c_{n+1})$, and $f$ is the corresponding element of the tuple $(\bar a, \bar c)$ (if $e = a_i$, then $f = a_i$, if $e = b_j$, then $f = c_j$), then it suffices to find $c_{n+1}$ such that $d(f,c_{n+1}) \in D$, where $D$ is a nonempty open subset of $[0,1]$. After doing this, $f$ will still be the closest element of $\{a_i : 1 \leq i \leq m\} \cup \{c_j : 1 \leq j \leq n\}$ to $c_{n+1}$.
    By the axioms of \UDLO{}, it is possible to choose $c_{n+1} \in M$ such that $(f,c_{n+1})$ has the same order type as $(e,b_{n+1})$, and $d(f,c_{n+1}) \in D$. It now suffices to check that $(\bar a, \bar b, b_{n+1})$ and $(\bar a, \bar c, c_{n+1})$ have the same order type. As $f$ is the closest element of $\{a_i : 1 \leq i \leq m\} \cup \{c_j : 1 \leq j \leq n\}$ to $c_{n+1}$, $c_{n+1}$ is adjacent to $f$ in the order, and similarly $b_{n+1}$ is adjacent to $e$ in its order. As $c_{n+1}$ is adjacent to $f$ on the correct side, it has been added to the correct gap in the order.
\end{proof}

\begin{thm}\label{thm:udlo_model_companion}
    The theory \UDLO{} is the model companion of \ULO{}, and \UDLO{} has quantifier elimination.
\end{thm}
\begin{proof}
    By Fact \ref{fact:universal_model_companion}, as \ULO{} is universal, Theorem \ref{thm:udlo_ec} implies that \UDLO{} is the model companion of \ULO{}.

    By Lemma \ref{lem:ulo_amalgamation}, $\ULO$ has amalgamation over substructures, so by Fact \ref{fact:model_companion_qe}, $\UDLO$ has quantifier elimination.
\end{proof}

Now that we have quantifier elimination,
we can show that $\mathrm{acl}$ is trivial in $\UDLO$.

\begin{lem}\label{lem:locally_constant}
    Let $M \vDash \UDLO{}$, $A \subseteq M$, and $\phi(x)$ an $A$-predicate.

    If $b,c \in M$ are such that $d(b,c) < d(b,A)$, then $\phi(b) = \phi(c)$.
    
    Consequently, if $b \in M$ is not in the metric closure of $A$, then $\phi(x)$ is constant on the open ball of radius $d(b,A)$ around $b$.
\end{lem}
\begin{proof}
    By quantifier elimination, we can express $\phi(x)$ as a continuous combination of countably many $A$-definable atomic formulas $\psi_n(x)$. If each of these satisfy $\psi_n(b) = \psi_n(c)$, then $\phi(b) = \phi(c)$ as well, so we may focus on an $A$-definable atomic formula $\psi(x)$.

    Any such formula takes one of the forms $d(x,a),r(x,a),r(a,x)$ for some $a \in A$.
    As $d(b,c) < d(b,a)$, the ultrametric property tells us that $d(a,b) = d(a,c)$. For the other predicates, we observe that the points $b$ and $c$ are on the same side of $a$ in the order. Thus either $r(b,a) = 0 = r(c,a)$, or $r(b,a) = d(b,a) = d(c,a) = r(c,a)$, and similarly for $r(a,x)$.
\end{proof}

\begin{thm}
    If $M \vDash \UDLO{}$ and $A \subseteq M$, then $\mathrm{acl}(A)$ is the metric closure of $A$.
\end{thm}
\begin{proof}
    Suppose $b \in \mathrm{acl}(A)$. 
    Then $b$ is in some compact set with $A$-definable distance predicate $\phi(x)$.

    If $b$ is not in the metric closure of $A$, then $d(b,A) > 0$. By Lemma \ref{lem:locally_constant}, $\phi(x)$ is constant on the the open ball defined by $d(x,b) < d(b,A)$. As the set defined by $\phi(x) = 0$ contains this entire open ball, it is not compact, a contradiction.
    (Any open ball in $\UDLO{}$ has infinite $\varepsilon$-covering number for any $\varepsilon$ smaller than the radius of the ball.)
    
    Meanwhile, the metric closure of a set in any metric structure is contained in its definable closure, and thus its algebraic closure (\cite[Exercise 10.10]{mtfms}).
\end{proof}

\section{Approximate Categoricity and Type Spaces}\label{sec:cat}

The aim of this section is to show that $\DU$ and $\UDLO$ satisfy a weak notion of separable categoricity developed by Hanson in \cite{hansonIso} and \cite{hansonCat} using the notion of distortion systems. Although this notion of categoricity does not necessarily give rise to nice invertible maps between separable models of $\DU$ and $\UDLO$, it is strong enough to proof a Ryll-Nardzewski-like result on the type space. This first subsection recall the necessary definitions, but we refer the reader to  \cite{hansonIso} and \cite{hansonCat} for the full details.

\begin{defn}
\begin{enumerate}
\item (\cite[Definition 2.4]{hansonIso}) Given a set of formulas $\Delta$, and fixing a continuous theory $T$, $\overline\Delta$ is the closure of $\Delta$ under renaming variables, quantification, 1-Lipschitz
connectives, logical equivalence modulo T, and uniform limits.
\item (\cite[Definition 2.6]{hansonIso}) Given $n\in\N$, a collection of formulas $\Delta$ separates distinct $n$-types $p$ and $q$ if there is a formula $\phi\in \Delta$ such that $\phi(p)\neq\phi(q)$ (after possibly renaming the variables in $\phi$).
\item (\cite[Definition 2.8]{hansonIso}) A set of formulas $\Delta$ is a \textit{distortion system} for a continuous theory $T$ if it separates any two distinct $n$-types for every $n\in\N$, and $\overline\Delta=\Delta$.
\end{enumerate}
\end{defn}

\begin{fact}[{\cite[Proposition 2.10]{hansonIso}}] 
If $\Delta$ separates any two distinct quantifier free $n$-types for every $n\in\N$, then $\overline\Delta$ is a distortion system.
\end{fact}

We can then use distortion systems to measure how distinct two models of a theory are: 

\begin{defn} Let $\Delta$ be a set of formulas in a metric language $\mathcal{L}$ with a set of sorts $\mathcal{S}$:
\begin{enumerate}
\item A correlation between metric structures $M$ and $N$ is a set $R$ such that $R\upharpoonright s \subset s(M)\times s(N)$ is a total surjective relation for every sort $s\in\mathcal{S}$.
\item An almost correlation  between metric structures $M$ and $N$ is a correlation between dense subsets of $M$ and $N$.
\item Given $\bar m \in M$ and  $\bar n \in N$, $cor(M,\bar m;N,\bar n)$ is the set of all correlations between $M$ and $N$ such that $(\bar m,\bar n)\in R$. Similarly, $acor(M,\bar m;N,\bar n)$ is the set of all almost correlations between $M$ and $N$ such that $(\bar m,\bar n)\in R$.
\item If $R\in cor(M;N)$, then the distortion of $R$ according to $\Delta$ is 
$$dis_\Delta(R) = \sup\{|\phi^M(\bar m)-\phi^N(\bar n)|\mid (\bar m,\bar n)\in R, \phi\in \Delta].$$
The same holds if  $R\in acor(M;N)$.
\item $\rho_\Delta(M,\bar m;N,\bar n)= \inf\{dis_\Delta(R)\mid R\in cor(M,\bar m;N,\bar n)\}$.
\item $a_\Delta(M,\bar m;N,\bar n)= \inf\{dis_\Delta(R)\mid R\in acor(M,\bar m;N,\bar n)\}$.
\end{enumerate}
\end{defn}

The next few facts give us basic results about the distortions systems and the distortion of a correlation that we will use in the next subsection: 
\begin{fact}[{\cite[Proposition 2.5]{hansonIso}}]\label{easyClosure}
For any $M,N\models T$ and $\bar m\in M$, $\bar n\in N$, $dis_\Delta(R) = dis_{\overline{\Delta}}(R)$
\end{fact}

\begin{fact}
Let $\Delta$ be a distortion system for $T$. Then:
\begin{enumerate}
\item{\cite[After Definition 1.5]{hansonCat}} If $A\subset M$ is some set of parameters, $\Delta(A)=\{\phi(\bar x, \bar a)\mid \bar a\in A\}$ is a distortion system for $T(A)$
\item{\cite[Proposition 1.7]{hansonCat}} For any set of fresh constant symbols $C$, let $D_0(\Delta, C) = \Delta \cup \{d(x,c)\}_{c\in C}$ and $D(\Delta,C)$ its closure. Then, $D(\Delta,C)$ is a distortion system for $T(C)$
\end{enumerate}
\end{fact}

\begin{remark}\label{simplify}
Let $\Delta$ be a distortion system for $T$, and $R\in cor(M,N)$ for $M,N\models T$ both containing a set of parameters $A$. Then
$$ dis_\Delta(R)\leq dis_{\Delta(A)}(R) \leq dis_{D_0(\Delta(A), C)}(R) = dis_{D(\Delta(A), C)}(R)$$
\end{remark}

\begin{fact}[{\cite[Lemma 1.12]{hansonIso}}]
Let $\Delta$ be a distortion system for $T$. For every predicate symbol $P\in\mathcal{L}$ and every $\varepsilon>0$ there is a $\delta>0$ such that $\rho_\Delta(M,\bar m;N,\bar n)<\delta$ implies $|P^M(\bar m)-P^N(\bar n)|<\varepsilon$. 
\end{fact}

We are now ready to state Hanson's notion of categoricity: 
\begin{defn}[{\cite[Definition 3.1]{hansonCat}}]
Fix a continuous theory $T$ and a distortion system $\Delta$ for $T$:
\begin{enumerate}
\item $T$ is $\Delta$-$\kappa$-categorical if $a_\Delta(M;N)=0$ for any two $M,N\models T$ of density character $\kappa$.
\item $T$ is strongly $\Delta$-$\kappa$-categorical if $\rho_\Delta(M;N)=0$ for any two $M,N\models T$ of density character $\kappa$.
\end{enumerate}
\end{defn}

Although this notion of cateogricty does not necessarily imply the existence of nice invertible maps, it does have nice consequences for the type space:

\begin{defn}[{\cite[Definition 2.18]{hansonIso}}]
Let $\Delta$ be a distortion system for $T$. For each $\lambda$ and types $p,q\in S_\lambda(T)$, we can define the distance
$$\delta^\lambda_\Delta(p,q) = \inf \{ \rho_\Delta(M,\bar m;N,\bar n) \mid \bar m\models p, \bar n\models q\}$$
\end{defn}

\begin{defn}[{\cite[Definition 1.8]{hansonCat}}]
Let $\Delta$ be a distortion system for $T$, $A$ a set of parameters in some model of $T$, and $C$ a set of fresh new constants. For each $\lambda$ and types $p,q\in S_\lambda(A)$, let
$$d_{\Delta,A}(p,q) = \delta^0_{D(\Delta(A),\bar c)}(p(\bar c,A), q(\bar c,A)).$$
Where we drop $A$ if it is empty.
\end{defn}

Calculating the value of $d_{\Delta,A}(p,q)$ can be a difficult task. The following lemma gives us a tool for such a task: 
\begin{fact}[{\cite[Proposition 1.10]{hansonCat}}]
Let $T$ be a complete theory, let $\Delta$ be a distortion system for $T$.
\begin{enumerate}
\item For every $\varepsilon>0$ there is a $\delta>0$ such that if \begin{itemize}
\item there are models $M,N\models T$ both containing a set of parameters $A$, with tuples $\bar m\in M$ and $\bar n,\bar b\in N$ such that $\bar m\models p$ and $\bar n\models q$, and an $R\in cor(M,\bar m;N,\bar b)$ with $dis_\Delta(R)\leq \delta$ and $d^N(\bar n,\bar b)\leq \delta$.
\end{itemize}
Then $d_{\Delta,A}(p,q)\leq \varepsilon$.
\item For any set of parameters $A$, if $d_{\Delta,A}(p,q)\leq \varepsilon$ then there are models $M,N\models T$ both containing $A$, with tuples $\bar m\in M$ and $\bar n,\bar b\in N$ such that $\bar m\models p$ and $\bar n\models q$, and an $R\in cor(M,\bar m;N,\bar b)$ with $dis_\Delta(R)\leq \varepsilon$ and $d^N(\bar n,\bar b)\leq \varepsilon$.
\end{enumerate}
\end{fact}

We can now state the Ryll-Nardzewski-like result obtained by Hanson:

\begin{defn}[{\cite[Definition 2.1]{hansonCat}}]
Let $X$ be a topological space and $d:X^2\to\mathbb{R}$ be a metric on $X$ not necessarily related to the topology. A point $x\in X$ is weakly $d$-atomic if the closed ball 
$$B_{\leq \varepsilon}(x)=\{y\in X\mid d(x,y)\leq \varepsilon\}$$
has nonempty interior for every $\varepsilon>0$. If $Y$ is a subspace of $X$ and $x\in Y$, we say that $x$ is weakly $d$-atomic-in-$Y$ if $x$ is weakly $d$-atomic in $Y$ using the subspace topology and computing the interiors in $Y$. 
\end{defn}

\begin{fact}[{\cite[Theorem 3.13]{hansonCat}}]\label{fact:approx_cat}
For any countable complete theory $T$ and distortion system $\Delta$ of T: 
\begin{enumerate}
\item $T$ is $\Delta$-$\omega$-categorical if and only if every type $p\in S_n(\bar a)$ is weakly $d_\Delta$-atomic-in-$S_n(\bar a)$, seen as a subspace of $S_{n+|\bar a|}(T)$, for every tuple of parameters $\bar a$ and every $n\in\N$. 
\item Same as the previous statement but only considering types in $S_1(\bar a)$.
\item If every space $S_n(T)$ is metrically compact relative to $d_\Delta$, then $T$ is $\Delta$-$\omega$-categorical.
\end{enumerate}
\end{fact}
Note: categoricity does not imply metric compactness.

\subsection{Approximate Categoricity of \DU}

\begin{prop}
Let $GH=\{d(x,y)\}$, then $GH$ is atomically complete for \DU. Therefore, $\overline{GH}$ is a distortion system for \DU.
\end{prop}

\begin{proof}
\DU{} has quantifier elimination, and the only atomic symbol is $d(x,y)$. 
\end{proof}

\begin{lem}\label{cor_construction}
Given separable $M,N\models \DU$, $\varepsilon>0$, $k\in\N$, and finite $k$-tuples $\bar m\in M$ and $\bar n\in N$ that are $\varepsilon$-isometric, there a correlation $R_\varepsilon\in cor(M,\bar m;N,\bar n)$, meaning that $(\bar m,\bar n)\in R$, with $dis_{\overline{GH}}(R_\varepsilon)\leq \varepsilon$.
\end{lem}

\begin{proof}
Fix a real number $\gamma>0$ such that 
$$\displaystyle e^{-\gamma} < \frac{d^{M}(m_i,m_j)}{d^{N}(n_i,n_j)} < e^{\gamma}$$
for $i,j<k$. Using separability of $M$ and $N$, extend the tuples $\bar m$ and $\bar n$ into countable sequences
$\{m_i\}_{i\in\N}$ dense in $M$ and $\{n_i\}_{i\in\N}$ dense in $N$. We assume that the first $k$ elements of each sequence are our original tuples.\par

\textbf{Step 1: Construction.} We construct a pair of sequences $\{\mu_i\}_{i\in\N}$ and $\{\nu_i\}_{i\in\N}$ via back-and-forth with the following properties for all $i\in\N$:

\begin{enumerate}
\item $\mu_i\in M$, and $\nu_i\in N$; 
\item $m_i\in (\mu_0,\dotsc, \mu_{2i})$, and $n_i\in (\nu_0,\dotsc, \nu_{2i+1})$;
\item $(\mu_0,\dotsc, \mu_i)$ and $(\nu_0,\dotsc, \nu_i)$ are $\varepsilon$-isometric.
\item $\displaystyle e^{-\gamma} < \frac{d^{M}(\mu_i,\mu_j)}{d^{N}(\nu_i,\nu_j)} < e^{\gamma}$ (note that this requirement is symmetric). 
\end{enumerate}

For the base case, we let $\bar m = (\mu_0,\dotsc,\mu_{k-1})$ and $\bar n =(\nu_0,\dotsc, \nu_{k-1})$. For the induction step, fix $i\in\N$, $i\geq k$, and assume we have build $(\mu_0,\dotsc, \mu_{\ell})$ and $(\nu_0,\dotsc, \nu_\ell)$ for $\ell\leq 2i-1$ satisfying the properties above. If $m_{i}$ is already in the list, then there is nothing to do. As in the proof of Lemma \ref{lem:delta_isometric_induction}, we fix $j_0\leq \ell$ that minimizes $d^M(m_i,\mu_{j_0})$. 
Consider now the open real interval 

$$D = \left((d^M(m_i,\mu_{j_0})-\varepsilon,d^M(m_i,\mu_{j_0})+\varepsilon)\cap (e^{-\gamma}d^M(m_i,\mu_{j_0}), e^{\gamma}d^M(m_i,\mu_{j_0}))\right)\setminus \{d^N(\nu_j,\nu_{j_0})\mid j\leq \ell\}$$

Note that both intervals are open, nonempty, and contain $d^M(m_i,\mu_{j_0})$, so their intersection also open nonempty. Using the density axiom of \DU, there is some $n\in N\cap D$. Let $t$ be strictly smaller than all the distances mentioned in the construction so far, and consider $B_t(n)\subset N$. By tail-density of $\{n_i\}_{i\in\N}$ in $N$, we can find some $i^*\in\N$ greater than all indices mentioned so far such that $n_{i^*}\in B_t(n)$. By the ultrametric axiom, $d^N(n,\nu_j) = d^N(n_{i^*},\nu_j)$, so we can use them interchangeably. We claim now that letting $\mu_{\ell+1}=m_i$ and $\nu_{\ell+1}= n_{i^*}$ extends the sequence preserving all the properties. The proof of $\varepsilon$-isometry is the same as in Lemma \ref{lem:delta_isometric_induction}, and the proof of Property 4 is a small variation where we have by construction that
$$e^{-\gamma} < \frac{d^{N}(\nu_{\ell+1},\nu_{j_0})}{d^{M}(\mu_{\ell+1},\mu_{j_0})} < e^{\gamma}$$
and by assumption that, for all $j\leq \ell$,
$$e^{-\gamma} < \frac{d^{N}(\nu_{j},\nu_{j_0})}{d^{M}(\mu_{j},\mu_{j_0})} < e^{\gamma}$$
so by the same argument as in Lemma \ref{isometric} we obtain Property 4. The back step is symmetric. This completes the construction.\par

\textbf{Step 2: Extension.} Now, we construct $R_\varepsilon$ as follows: $(a,b)\in R_\varepsilon$ if for some (equivalently every) function $f:\N\to \N$ we have that $a=\lim_{i\to\infty} \mu_{f(i)}$ and  $b=\lim_{i\to\infty} \nu_{f(i)}$. To show that this is a well defined notion, we prove the following claims:

\begin{claim}
For every function $f:\N\to \N$, we have that $\{\mu_{f(i)}\}_{i\in\N}$ is Cauchy in $M$ if, and only if $\{\nu_{f(i)}\}_{i\in\N}$ is Cauchy in $N$.
\end{claim}
\begin{proof}
We prove the forward direction, the backward direction being symmetric. Fix $\delta>0$; by assumption there is some $K\in\N$ such that if $K<i,j$, then $d^{M}(\mu_{f(i)},\mu_{f(j)})<\frac{\delta}{e^\varepsilon}$. Therefore, by Property 4, we have that $d^{N}(\nu_{f(i)},\nu_{f(j)})<\delta$. 
\end{proof}

\begin{claim}
Every element of $M$ can be written as the limit of a sequence $\{\mu_{f(i)}\}_{i\in\N}$ for some function $f:\N\to\N$. Similarly, every element of $N$ can be written as the limit of a sequence $\{\nu_{f(i)}\}_{i\in\N}$ for some function $f:\N\to\N$.
\end{claim}
\begin{proof}
This follows from the fact that the sequences $\{\mu_i\}_{i\in\N}$ and $\{\nu_i\}_{i\in\N}$ are dense in their respective structures. Indeed, by Property 2 of the construction, each element of the dense sequence $\{m_i\}_{i\in\N}$ appears in $\{\mu_i\}_{i\in\N}$, and similarly for $\{\nu_i\}_{i\in\N}$.
\end{proof}

\begin{claim}
The definition of $R_\varepsilon$ is well-defined.
\end{claim}

\begin{proof}
Suppose we have two functions $f,g:\N\to\N$ with $a=\lim_{i\to\infty} \mu_{f(i)}= \lim_{i\to\infty} \mu_{g(i)}$. By the previous lemma, the corresponding subsequences on of $\{\nu_i\}_{i\in\N}$ are Cauchy, so let $b_f = \lim_{i\to\infty} \nu_{f(i)}$ and $b_g = \lim_{i\to\infty} \nu_{g(i)}$. Now observe that
$$ d^{N}(b_f,b_g)\leq d^{N}(b_f,\nu_{f(i)}) + d^{N}(\nu_{f(i)},\nu_{g(i)}) + d^{N}(b_g,\nu_{g(i)})$$
where the first and last term can clearly be made arbitrarily small by taking $i\to\infty$. To see that the middle term also goes to zero, note that $d^{M}(\mu_{f(i)},\mu_{g(i)})\to 0$ as both sequences limit to $a$, so by Property 4, the same is true for $d^{N}(\nu_{f(i)},\nu_{g(i)})$. Thus, $b_f=b_g$.  
\end{proof}

\textbf{Step 3: Verification.} By the claims above $R_\varepsilon$ is a total surjective relation, so it is a correlation between $M$ and $N$.\par

Next, recall that $\bar m=(m_0,\dotsc, m_{k-1})$, $\mu_\ell=m_\ell$ for $\ell<k$ by construction, and similarly $\bar n = (\nu_0,\dotsc, \nu_{k-1})$. Therefore, the constant function $f_\ell(i)=\ell$ show that $(m_\ell,n_\ell)\in R_\varepsilon$ for $\ell<k$. Thus, $(\bar m,\bar n)\in R_\varepsilon$.\par 

It remains to show that $dis_{\overline{GH}}(R_\varepsilon)\leq \varepsilon$.  By Proposition \ref{easyClosure}, it is enough to show that $dis_{GH}(R_\varepsilon)\leq \varepsilon$. Suppose then that $(a,b),(c,d)\in R_\varepsilon$. We want to show that $|d^M(a,c)-d^N(b,d)|\leq \varepsilon$. By construction there are functions $f,g:\N\to\N$ such that 
$$a=\lim_{i\to\infty} \mu_{f(i)},\ b=\lim_{i\to\infty} \nu_{f(i)},\ c=\lim_{i\to\infty} \mu_{g(i)},\ d=\lim_{i\to\infty} \nu_{g(i)}.$$
Moreover, $|d^M(\mu_{f(i)},\mu_{g(i)})-d^N(\nu_{f(i)},\nu_{g(i)})|< \varepsilon$ by $\varepsilon$-isometry for every $i\in\N$. The result follows by continuity. 

\end{proof}

\begin{thm}\label{thm:DU_cat}
Thus, \DU{} is strongly $\overline{GH}$-$\omega$-categorical.
\end{thm}

\begin{proof}
Given separable $M,N\models \DU$ and $\varepsilon>0$, we can use Lemma \ref{cor_construction} with the empty tuples to find a correlation $R_\varepsilon\in cor(M;N)$ with $dis_{\overline{GH}}(R_\varepsilon)\leq \varepsilon$. Since this can be done for every $\varepsilon>0$, it follows that $\rho_{\overline{GH}}(M,N)=0$.
\end{proof}

Our next goal is to characterize $d_{\overline{GH}}$ in $S_n(\bar a)$ for some finite tuple $\bar a$:

\begin{lem}\label{lem:type_dis}
Let $\bar a$ be a finite set of parameters for $\DU$. Let $p,q\in S_n(\bar a)$, then $d_{\overline{GH}}(p,q)< \varepsilon$ if and only if  any (resp. all) pair(s) of realizations of $p$ and $q$ (in perhaps different models) are $\varepsilon$-isometric for all $\varepsilon>0$. Thus,
$$d_{\overline{GH}}(p,q)=\inf\{\varepsilon: \text{any/all realizations of $p$ and $q$ are $\varepsilon$-isometric}\}$$
\end{lem}

\begin{proof}
The equivalence between any and all is because $\varepsilon$-isometry is implied by a formula in the complete types (we need finitely many variables here).\par

$(\Rightarrow)$ Suppose we have types $p,q\in S_n(\bar a)$ with $d_{\overline{GH}}(p,q)<\varepsilon$. Therefore, there are models $M,N\models DU$ containing $\bar a$ and tuples $\bar m\in M$ and $\bar n\in N$ realizing $p$ and $q$ respectively such that
$$\rho_{D(\overline{GH}(\bar a),\bar c)}(M,\bar m; N, \bar n)< \varepsilon.$$
Thus, there is a correlation $R\in cor(M,\bar m;N,\bar n)$ with $dis_{D(\overline{GH}(\bar a),\bar c)}(R)< \varepsilon$. By Proposition \ref{simplify}, we get that $dis_{\overline{GH}}(R)< \varepsilon$. Thus, $\bar m$ and $\bar n$ have to be $\varepsilon$-isometric.\par

$(\Leftarrow)$ Suppose we have types $p,q\in S_n(\bar a)$ with models $M,N\models DU$ containing $\bar a$ and tuples $\bar m\in M$ and $\bar n\in N$ realizing $p$ and $q$ respectively such that $\bar m$ and $\bar n$ are $\varepsilon$-isometric. As the tuples are finite, there is some $\varepsilon_0<\varepsilon$ such that the tuples are in fact $\varepsilon_0$-isometric. Note that $M$ extends canonically into a model of the $\mathcal{L}(\bar c,\bar a$) theory $p(\bar c,\bar a)$ by interpreting $\bar c$ as $\bar m$. Similarly, $N$ extends into a model of $q(\bar c,\bar a)$. \par

Next, by Lemma \ref{cor_construction}, there is a  correlation $R\in cor(M,\bar a\bar m;N,\bar a\bar n)$ with $dis_{\overline{GH}}(R)\leq \varepsilon_0$. It remains to show that $dis_{D(\overline{GH}(\bar a),\bar c)}(R)\leq \varepsilon_0$, and it is enough to show that $dis_{D_0(\overline{GH}(\bar a),\bar c)}(R)\leq \varepsilon_0$. In moving from $\overline{GH}$ to $D_0(\overline{GH}(\bar a),\bar c)$, we added $\bar a$ as parameters, which doesn't increase the distortion as they are finitely many and already correlated in $R$, and formulas of the form $d(x,c_i)$ for $i<|\bar c|$. Hence, it is enough to show that 
$$\sup_{c\in\bar c,(u,v)\in R} | d^M(u,c_i)-d^N(v,c_i)|\leq \varepsilon_0$$
but this follows because $(\bar mu,\bar nv)\in R$, so the tuples must be $\varepsilon_0$ isometric. Thus, 
$$d_{\overline{GH}}(p,q)\leq \rho_{D(\overline{GH}(\bar a),\bar c)}(M,\bar m; N, \bar n)< \varepsilon.$$
\end{proof}

\begin{cor}
For any finite set of parameter $\bar a$ living in the monster model of $\DU$, every type in $S_n(\bar a)$ is weakly $d_{\overline{GH}}$-atomic-in-$S_n(\bar a)$, where
$$d_{\overline{GH}}(p,q)=\inf\{\varepsilon: \text{any/all realizations of $p$ and $q$ are $\varepsilon$-isometric}\}$$.
\end{cor}

\subsection{Approximate Categoricity of \UDLO}

Our result on \DU{} can be transferred into \UDLO, giving us approximate categoricity of \UDLO:

\begin{prop}
Let $OGH=\{d(x,y), r(x,y)\}$, then $OGH$ is atomically complete for \UDLO. Therefore, $\overline{OGH}$ is a distortion system for \UDLO.
\end{prop}

\begin{proof}
\UDLO\ has quantifier elimination, and the only atomic symbols are $d(x,y),r(x,y)$ and $d_\leq(a,b)$. However, by Lemma \ref{lem:ulo_dist_ray}, $d_\leq=r$. 
\end{proof}

\begin{thm}\label{thm:UDLO_cat}
Fix separable $M,N\models \UDLO$ and $\varepsilon>0$, then there is $R_\varepsilon\in cor(M,N)$ such that $dis_{\overline{OGH}}(R_\varepsilon)\leq \varepsilon$. Thus, \UDLO{} is strongly $\overline{OGH}$-$\omega$-categorical.
\end{thm}

\begin{proof}
The structure of the proof is the same as in Theorem \ref{thm:DU_cat}. The only change is that in the extension we must also preserve the order type, for which we use Lemma \ref{lem:order_isometric_induction}.
\end{proof}

Our next goal is to characterize $d_{\overline{OGH}}$ in $S_n(\bar a)$ for some finite tuple $\bar a$:

\begin{defn}\label{isometric}
    If $M, N\models \UDLO$, $\varepsilon > 0$, and $(m_1,\dots,m_k) \in M^k, (n_1,\dots,n_k) \in N^k$ are tuples of the same length, call these tuples $\varepsilon$-\emph{order isometric} if 
    \begin{enumerate}
    \item $|d^M(m_i,m_j) - d^N(n_i,n_j)|< \varepsilon$
    \item $|r^M(m_i,m_j)-r^N(n_i,n_j)|<\varepsilon$
    \end{enumerate}
    for all $1 \leq i,j \leq k$.
\end{defn}

Note that order isometric tuples might not have the same order type. The next Lemma is a generalization of Lemma \ref{lem:order_isometric_induction} to order isometric tuples:

\begin{lem}\label{lem:diff_ord_tp_induction}
    Suppose $M \vDash \ULO$, $N \vDash \UDLO$, $\varepsilon > 0$, and $(m_1,\dots,m_k), (n_1,\dots,n_k)$ are $\varepsilon$-order isometric. Then, for every $m\in M$ there is an open ball $B\subset N$ such that if $n\in B$, then  $(m_1,\dots,m_k,m)$ and $(n_1,\dots,n_k,n)$ are $\varepsilon$-order isometric.
\end{lem}
\begin{proof}
    Our strategy is similar to Lemma \ref{lem:delta_isometric_induction} and Lemma \ref{lem:order_isometric_induction}. Thus, let
    $1 \leq i_0 \leq n$ be such that $m_{i_0}$ minimizes $d(m_{i_0},m)$. Without loss of generality, we may assume that $ m_{i_0}< m$ and that $m_{i_0}$ and $m$ are consecutive elements in the order among $\{m_1,\dots,m_k\}\cup\{m\}$.\par

    We will construct an open set $U\subset\mathbb{R}$ such that if we let $n\in N$ be such that $r(n,n_{i_0})-r(n_{i_0},n)\in U$ and $t>0$ is a number smaller than all the nonzero distances among elements in this construction, then $B_t(n)$ is the desired open ball by the ultrametric inequality.\par

    First, we will try letting
    $$U = \left(\max_{i : m_i < m_{i_0}} r(n_i,n_{i_0}),\min_{i : m_i > m_{i_0}} r(n_i,n_{i_0})\right) \cap \left(d(m,m_{i_0}) - \varepsilon,d(m,m_{i_0}) + \varepsilon\right) \setminus \{d(n_i,n_{i_0}):1 \leq i \leq k\}.$$
    We claim that if $n \in N$ satisfies $r(n,n_{i_0}) \in U$, then $(m_1,\dots,m_k,m),(n_1,\dots,n_k,n)$ are $\varepsilon$-order isometric.
    By Lemma \ref{lem:delta_isometric_induction}, we see that this is a $\varepsilon$-isometry. Also, we see that if $r(n,n_{i_0}) \in U$, then for all $1 \leq i \leq k$ with $m_i < m$, we have $r(n_i,n_{i_0})<r(n,n_{i_0})$, so $n_i < n$, and similarly, if $m_i > m$, then $r(n_i,n_{i_0})>r(n,n_{i_0})$, so $n_i > n$. Thus for any $1 \leq i \leq k$, either $|r(n_i,n) - r(m_i,m)| = |0 - 0| = 0$, or $|r(n_i,n) - r(m_i,m)| = |d(n_i,n) - d(m_i,m)| < \varepsilon$, and similarly, $|r(n,n_i) - r(m,m_i)| < \varepsilon$.

    If $U$ is nonempty, this set will work, so assume $U$ is empty - this implies that 
    $$\left(\max_{i : m_i < m_{i_0}} r(n_i,n_{i_0}),\min_{i : m_i > m_{i_0}} r(n_i,n_{i_0})\right) \cap \left(d(m,m_{i_0}) - \varepsilon,d(m,m_{i_0}) + \varepsilon\right)$$
    is empty. We will check that if the interval on the left is nonempty, then the intersection is.
    To start, observe that for any $i$ with $m_i < m_{i_0}$,
    either $r(n_i,n_{i_0}) = 0$, or $n_i > n_{i_0}$, in which case by the $\varepsilon$-order isometry assumption, we see that $$r(n_i,n_{i_0}) = d(n_i,n_{i_0}) < \varepsilon < d(m,m_{i_0} + \varepsilon).$$
    Thus $\max_{i : m_i < m_{i_0}} r(n_i,n_{i_0}) < d(m,m_{i_0}) + \varepsilon$.
    Then observe that for any $i$ with $m_i > m_{i_0}$, by the assumption that $m,m_{i_0}$ are consecutive in the order, we have $m_{i_0} < m < m_i$, and thus $r(m_i,m_{i_0}) > r(m,m_{i_0})$. By the $\varepsilon$-order isometry assumption, we have $d(n_i,n_{i_0}) > d(m_i,m_{i_0}) - \varepsilon \geq d(m,m_{i_0}) - \varepsilon$,
    so $\min_{i : m_i > m_{i_0}} r(n_i,n_{i_0}) > d(m,m_{i_0}) - \varepsilon$.

    The remainder of the proof focuses on the case when
    $$\min_{i : m_i > m_{i_0}} r(n_i,n_{i_0})\leq \max_{i : m_i < m_{i_0}} r(n_i,n_{i_0}).$$
   This implies that there are $j$ and $\ell$ such that $m_j<m_{i_0}<m< m_\ell$ and $r(n_\ell, n_{i_0})\leq r(n_j,n_{i_0})$. We may assume that $j$ is such that $n_j$ is the $<$-greatest with $r(n_j,n_{i_0})$ maximal among all $m_i < m_{i_0}$. Similarly, $\ell$ is such that $n_\ell$ is the $<$-smallest element with $r(n_\ell, n_{i_0})$ minimal among all $m_i > m_{i_0}$. Consider now the set
   $$U=  \left(0, \min_{\substack{1\leq i\leq k, \\n_i\neq n_{i_0}}} d(n_i,n_{i_0})\right)\cap \left(d(m,m_{i_0}) - \varepsilon,d(m,m_{i_0}) + \varepsilon\right)\ $$
   and note that if $d(m, m_{i_{0}})<\varepsilon$, then $U$ is nonempty; also, if $n\in N$ satisfies $r(n,n_{i_0}) \in U$, then Lemma \ref{lem:delta_isometric_induction} shows that this choice makes the tuples $(m_0,\dotsc, m_k, m)$ and  $(n_0,\dotsc, n_k, m)$ $\varepsilon$-isometric. In cases (1)-(3) below, as well as in (4a), we will use this remark to find a suitable extension that preserves $\varepsilon$-isometry and then check tha the order is also approximately preserved. Case (4b) requires a different set $U$ to work, which we define there. With this in mind, we split into cases:

    \begin{enumerate}
    \item[(1)] Suppose $0 = r(n_\ell, n_{i_0}) = r(n_j,n_{i_0})$, so $n_j, n_\ell\leq n_{i_0}$: by $\varepsilon$-order isometry, both $d(m_\ell, m_{i_0})$ and $d(n_\ell,n_{i_0})$ are $<\varepsilon$, and thus so is $d(m,m_{i_0})$, so this is an $\varepsilon$-isometry by the remark above. As for the order: note that, as in the previous case, if the order does not change, then the order predicates disagree by at most $\varepsilon$. Thus, we only need to check the cases where the tuples $(m_i,m)$ and $(n_i,n)$ have different order types. By our assumptions on $j$, this can only happen when $m_{i_0}<m< m_i$ and $n_\ell \leq n_i < n_{i_0}<n$. In that case, $\varepsilon$-order isometry of the original tuples implies that both $d(m_{i_0},m_i)$ and $d(n_{i_0},n_i)$ are $<\varepsilon$ so
        $$|r(m_i,m)-r(n_i,n)|=d(m_i,m)< d(m_i,m_{i_0})<\varepsilon$$
        and 
        $$|r(m,m_i)-r(n,n_i)|=d(n,n_i)\leq \max(d(n,n_{i_0}),d(n_{i_0},n_i))<\varepsilon$$
        since $d(n,n_{i_0})\leq d(n_\ell,n_{i_0})<\varepsilon$.

    \item[(2)] Suppose $0 = r(n_\ell, n_{i_0}) < r(n_j,n_{i_0}) = \delta$, so $n_\ell\leq n_{i_0}<n_j$ and $d(n_j,n_{i_0})=\delta$. As $(m_j,m_\ell)$ and $(n_j,n_\ell)$ have different order types, $\varepsilon$-order isometry requires that their distances be $<\varepsilon$, so $\delta<\varepsilon$. By the ultrametric condition and the order, all of $m_j,m_{i_0},m,m_\ell$ are $<\varepsilon$ apart from each other, and the same holds for $n_{i_0},n_\ell,n_j$. In particular, $d(m,m_{i_0})<\varepsilon$, so the set $U$ is nonempty and does not contain any of the distances of the form $d(n_i,n_{i_0})$. Hence, if $n\in N$ satisfies $r(n,n_{i_0}) \in U$, then we are extending an $\varepsilon$-isometry. As before, to check that the order is $\varepsilon$-preserved, we need to consider the cases where $(m_i,m)$ and $(n_i,n)$ have different order types. Note that as $m_{i_0}<m$ and $n_{i_0}<n$, so nothing to check here. Moreover, $m_{i_0}<m$ are consecutive by assumption; as are $n_{i_0}<n$ by construction. Thus, the only case we need to check are:
    \begin{itemize}
    \item $m_i<m_{i_0}<m$ but $n_{i_0}<n<n_i$. As $(m_i,m_{i_0})$ and $(n_i,n_{i_0})$ have different order types, $\varepsilon$-order isometry requires that their distance be $<\varepsilon$. From this we can conclude that
    $$|r(m,m_i)-r(n,n_i)|=d(m_i,m)\leq\max(d(m_i,m_{i_0}), d(m_{i_0},m))<\varepsilon$$
    as $M$ is an ultrametric space and 
    $$|r(m_i,m)-r(n_i,n)|=d(n_i,n)<\max(d(n_i,n_{i_0}), d(n_{i_0},n))<\varepsilon$$
    since $d(n_{i_0},n)<\delta<\varepsilon$.
    \item $m_{i_0}<m< m_i$ and $n_i \leq  n_{i_0}<n$. In that case, $\varepsilon$-order isometry implies that both $d(m_{i_0},m_i)$ and $d(n_{i_0},n_i)$ are $<\varepsilon$ so
        $$|r(m_i,m)-r(n_i,n)|=d(m_i,m)< d(m_i,m_{i_0})<\varepsilon$$
        and 
        $$|r(m,m_i)-r(n,n_i)|=d(n,n_i)\leq \max(d(n,n_{i_0}),d(n_{i_0},n_i))<\varepsilon$$
        since $d(n,n_{i_0})<\delta<\varepsilon$.
    \end{itemize}

    \item[(3)] Suppose $0 < \rho = r(n_\ell, n_{i_0}) < r(n_j,n_{i_0}) = \delta$, so $n_{i_0}<n_\ell<n_j$ and $d(m_\ell,m_j)$ and $d(n_\ell,n_j)=\delta$ are both $<\varepsilon$ by $\varepsilon$ isometry. As in the previous case, this implies that all of $m_j,m_{i_0},m,m_\ell$ are $<\varepsilon$ apart from each other, and the same holds for $n_{i_0},n_\ell,n_j$. In particular, $d(m,m_{i_0})<\varepsilon$, so the set $U$
    is nonempty and does not contain any of the distances of the form $d(n_i,n_{i_0})$. Hence, if $n\in N$ satisfies $r(n,n_{i_0}) \in U$, then this extension is a $\varepsilon$-isometry. As before, to check that the order is $\varepsilon$-preserved, we need to consider the cases where $(m_i,m)$ and $(n_i,n)$ have different order types.  By the order minimality of $n_\ell$, we can rule out the case where $m_i>m> m_{i_0}$ as this implies $n_i\geq n_\ell>n$. Note that as $m_{i_0}<m$ and $n_{i_0}<n$, we only need to consider the case where $m_i<m_{i_0}<m$ but $n_{i_0}<n<n_i$, which we do as in the previous case.
    
    \item[(4)] Suppose $0 <  \rho =r(n_\ell, n_{i_0}) = r(n_j,n_{i_0})$, so $n_{i_0}<n_\ell,n_j$ and $d(n_{i_0},n_\ell)= d(n_{i_0},n_j)=\rho$: by $\varepsilon$-order isometry the distances $d(m_{i_0},m_j)$, $d(n_{i_0},n_j)$ are $<\varepsilon$, and by the ultrametric condition so is $d(n_{i_0},n_\ell)$. We have two subcases:
    \begin{itemize}
    \item[(4a)] If $n_{i_0}<n_\ell\leq n_j$, then all of the pairwise distances between $m_j,m_{i_0},m,m_\ell$ and $n_{i_0},n_\ell,n_j$ are $<\varepsilon$ and we can proceed as in the previous cases. 
    
    \item[(4b)] If $n_{i_0}<n_j<n_\ell$, then $d(m_j,m_{i_0})$ and $d(n_j,n_{i_0})$ are $<\varepsilon$. The easy case is when $d(m_{i_0},m)<\varepsilon$, which we can can handle using the same set $U$ as before.\par
    
    Otherwise, we must handle this case quite differently than all the previous ones. Note that since $d(m_{i_0},m_j)<\varepsilon$ and we have assumed that $d(m_{i_0},m)\geq\varepsilon$, then $d(m_j,m)=d(m_{i_0},m)\geq \varepsilon$. Therefore, $m_j$ is also a closest point to $m$, and we will focus on it going forward. Note that using the isometry condition we have
    $$d(n_j,n_\ell)> d(m_j,m_\ell)-\varepsilon\geq d(m_j,m)-\varepsilon.$$
    Thus, the open set
    $$U=(0,d(n_j,n_\ell))\cap(d(m_j,m)-\varepsilon,d(m_j,m)+\varepsilon)$$
    is nonempty. Using the density axioms, we can find $n\in N$ be such that $r(n,n_j)-r(n_j,n)\in U$. Since $m_j$ was also the closest point to $m$, we can use the same arguments as in Lemma \ref{lem:delta_isometric_induction} to see that this is an $\varepsilon$-isometry. By the maximality and minimality conditions on $n_j$ and $n_\ell$, we can see that no $n_i\in(n_j,n_\ell)$, so there is no $i$ such that the pairs $(m_i,m)$ and $(n_i,n)$ have different order types, so this is also an $\varepsilon$-order isometry.

    \end{itemize}
    \end{enumerate}
This completes the proof.

\end{proof}

With this result, we can now construct the desired correlations:

\begin{lem}\label{ord_cor_construction}
Given separable $M,N\models \UDLO$, $\varepsilon>0$, $k\in\N$, and finite $k$-tuples $\bar m\in M$ and $\bar n\in N$ that are $\varepsilon$-order isometric, there a correlation $R_\varepsilon\in cor(M,\bar m;N,\bar n)$, meaning that $(\bar m,\bar n)\in R$, with $dis_{\overline{GH}}(R_\varepsilon)\leq \varepsilon$.
\end{lem}

\begin{proof}
The proof structure is the same as Lemma \ref{cor_construction}. However, in the construction step, we use Lemma \ref{lem:diff_ord_tp_induction} to extend $\varepsilon$-order isometry. Thus, for an appropriately chosen $\gamma>0$, we get sequences $\{\mu_i\}_{i\in\N}$ dense in $M$ and $\{\nu_i\}_{i\in\N}$ dense in $N$ (and extending $\bar m$ and $\bar n$ respectively) such that for every $i\in N$ we have that $(\mu_0,\dotsc, \mu_i)$ and $(\nu_0,\dotsc, \nu_i)$ are $\varepsilon$-order isometric and  satisfying the same multiplicative condition as in \ref{cor_construction}.\par

The extension step follows them same strategy to get a correlation $R_\varepsilon\in cor(M,\bar m; N,\bar n)$ such that the $(a,b)\in R_\varepsilon$ if for some (equivalently every) function $f:\N\to \N$ we have that $a=\lim_{i\to\infty} \mu_{f(i)}$ and  $b=\lim_{i\to\infty} \nu_{f(i)}$.\par

The only missing part to show that $dis_{\overline{GH}}(R_\varepsilon)\leq \varepsilon$ is to check that the $r$ predicate on related pairs differs by no more than $\varepsilon$, but as in \ref{cor_construction} this follows from continuity and the fact that the sequences are $\varepsilon$-order isometric. 
\end{proof}

Thus, following the same strategy as in Lemma \ref{lem:type_dis} and Theorem \ref{fact:approx_cat} from Hanson (see {\cite[Theorem 3.13]{hansonCat}), we get that:

\begin{cor}
For any finite set of parameter $\bar a$ living in some model of $\UDLO$, every type in $S_n(\bar a)$ is weakly $d_{\overline{OGH}}$-atomic-in-$S_n(\bar a)$, where
$$d_{\overline{OGH}}(p,q)=\inf\{\varepsilon: \text{any/all realizations of $p$ and $q$ are $\varepsilon$-order isometric}\}$$.
\end{cor}
\section{Regulated Functions}\label{sec_reg}
Now that we have shown quantifier elimination in \UDLO{}, we would like to understand what the quantifier-free definable predicates are, at least in one variable. Our characterization is inspired by, and will generalize, the characterization of definable predicates in one variable in (weakly) o-minimal structures in discrete logic.

In discrete logic, a structure expanding a dense linear order is o-minimal when every definable set in dimension 1 is a finite union of intervals, or weakly o-minimal when every definable set in dimension 1 is a finite union of order-convex sets.
To generalize these properties to continuous logic, we will first need to generalize some standard analysis definitions to general linear orders.

\begin{defn}
    Let $X$ be a linear order.
    Let the set of \emph{strong step functions} on $X$ be the $\R$-linear span of the set of characteristic functions of intervals in $X \to \R$,
    and let the set of \emph{weak step functions} be the $\R$-linear span of the set of characteristic functions of order-convex sets.
\end{defn}
It is easy to see a discrete structure expanding a dense linear order is o-minimal if and only if every $\{0,1\} \subseteq \R$-valued formula is a strong step function,
and that a discrete structure expanding a dense linear order is weakly o-minimal if and only if every $\{0,1\} \subseteq \R$-valued formula is a weak step function.
If $X$ is a complete linear order, then intervals and order-convex sets coincide, and so we may speak simply of step functions, without specifying strong or weak.

Regulated functions on intervals of $\R$ were defined in \cite{bourbaki_FVR}. We may take the following fact as a definition:
\begin{fact}[{\cite[Th\'eor\`eme 3, FVR II.5]{bourbaki_FVR}}]
The space of regulated functions on an interval in $\R$ is the closure of the space of step functions under the $\sup$ norm.
\end{fact}

Over any complete linear order, such as an interval in $\R$, order-convex sets are precisely intervals, so when stating this fact, we need not distinguish between strong and weak step functions.
However, the order of an o-minimal structure need not be complete, so we will generalize regulated functions to two other definitions:

\begin{defn}\label{defn_strong_reg}
Let $X$ be a linear order, and let $f : X \to \R$.
We say that $f$ is \emph{strongly regulated} if it is a uniform limit of strong step functions, and \emph{weakly regulated} if it is a uniform limit of weak step functions.
\end{defn}

\begin{lem}\label{lem:reg_comb}
    The classes of strong/weak step functions and strongly/weakly regulated functions are closed under continuous combinations.
\end{lem}
\begin{proof}
    The classes of strong and weak step functions are closed under continuous combinations, as if $f_1,\dots,f_n$ are strong/weak step functions, then for any function $g : \R^n \to \R$, $g(f_1(t),\dots,f_n(t))$ is one as well.
    Let $g : \R^n \to \R$ be continuous, with $f_1,\dots,f_n$ strongly/weakly regulated.
    To approximate $g(f_1(t),\dots,f_n(t))$ up to $\varepsilon > 0$ with a strong/weak step function, we find $\delta$ such that $|x_i - y_i|\leq \delta$ implies $|g(x_1,\dots,x_n) - g(y_1,\dots,y_n)|\leq \varepsilon$.
    Then there is a strong/weak step function $f_i'$ for each $i$ approximating $f_i$ up to $\delta$, and then $|g(f_1(t),\dots,f_n(t)) - g(f_1'(t),\dots,f_n'(t))|\leq \varepsilon$ for all $t$.
\end{proof}

We can rephrase the regulated properties with the following useful criteria:
\begin{lem}\label{lem:reg_partition}
    Let $X$ be a linear order and $f : X \to [0,1]$.
    The following are equivalent:
    \begin{enumerate}
        \item\label{item:reg} $f$ is strongly regulated
        \item\label{item:part_const} For all $\varepsilon > 0$, there is a partition of $X$ into finitely many intervals $I_1,\dots,I_n$ such that on each $I_i$, $f$ varies by at most $\varepsilon$
        \item\label{item:part_ineq} For all $0 \leq s < t \leq 1$, there is a partition of $X$ into finitely many intervals $I_1,\dots,I_n$ such that on each $I_i$, either $f(x) < t$ or $f(x) > s$.
    \end{enumerate}
    as are the following:
    \begin{enumerate}
        \item $f$ is weakly regulated
        \item For all $\varepsilon > 0$, there is a partition of $X$ into finitely many order-convex sets $I_1,\dots,I_n$ such that on each $I_i$, $f$ varies by at most $\varepsilon$
        \item For all $0 \leq s < t \leq 1$, there is a partition of $X$ into finitely many order-convex sets $I_1,\dots,I_n$ such that on each $I_i$, either $f(x) < t$ or $f(x) > s$.
    \end{enumerate}
\end{lem}
\begin{proof}
    We will prove the equivalences for strongly regulated, as the proof for weakly regulated has the same structure.

    (\ref{item:reg}) $\implies$ (\ref{item:part_const}):
    If $f$ is strongly regulated, then for any $\varepsilon > 0$, there is a strong step function $g$ with $\sup_x|f(x) - g(x)| \leq \frac{\varepsilon}{2}$. By the nature of strong step functions, $X$ can be partitioned into intervals $I_1,\dots, I_n$ on each of which $g$ is constant. Thus for $x,y \in I_i$, 
    $$|f(x) - f(y)| \leq |f(x) - g(x)| + |f(y) - g(y)| \leq \varepsilon,$$ so $f$ varies by at most $\varepsilon$.

    (\ref{item:part_const}) $\implies$ (\ref{item:part_ineq}):
    Partition $X$ into finitely many intervals on $f$ varies by at most $\frac{t - s}{2}$.
    On each of these intervals $I$, if we had $x,y \in I$ with $f(x) \leq s$ and $f(y) \geq t$, we would have $|f(x) - f(y)|\geq t - s > \frac{t - s}{2}$, a contradiction to our assumption that $f$ varies by at most $\frac{t - s}{2}$.

    (\ref{item:part_ineq}) $\implies$ (\ref{item:reg}):
    Suppose this holds for all $s < t$, and fix $\varepsilon > 0$. We will find a strong step function $g : X \to [0,1]$ with $\sup_{x \in X} |f(x) - g(x)| \leq \varepsilon$.
    Let $n > \varepsilon^{-1}$. Then for any $0 \leq i < n$, we see that there is a finite partition of $X$ into intervals such that on each interval, either $f(x) < \frac{i + 1}{n}$ or $f(x) > \frac{i}{n}$.
    Partition $X$ into a finite partition into intervals that refines all $n$ of these partitions. Then for each interval $I$ in this partition, there are no $x_1,x_2 \in I$ and $0 \leq i < n$ with $f(x_1) \leq \frac{i}{n}$ and $f(x_2) \geq \frac{i + 1}{n}$, so in particular, there is some $j$ with $\frac{j - 1}{n} < f(x) < \frac{j + 1}{n}$ for all $x \in I$. We let $g(x) = \frac{j}{n}$ on this interval, and define it similarly for all other intervals.
    We then see that for all $x \in I$, 
    $$|f(x) - g(x)|= \left|f(x) - \frac{j}{n}\right| < \frac{1}{n} < \varepsilon.$$
\end{proof}

We claim that regulated functions are critical to understanding o-minimality in continuous logic. Any structure of discrete logic can also be understood as a metric structure. We claim that if we view an o-minimal structure of discrete logic this way, its definable predicates are precisely the strongly regulated functions.

\begin{thm}\label{thm:discrete_omin_reg}
    Suppose $M$ is a structure of discrete logic expanding a linear order.
    Then if $f : M \to [0,1]$ is strongly regulated, then using continuous logic, there is some definable predicate $\phi(x)$ with $\phi^M = f$.

    Conversely, if $M$ is o-minimal and $\phi(x)$ is a definable predicate, then $\phi^M$ is strongly regulated. If $M$ is weakly o-minimal and $\phi(x)$ is a definable predicate, then $\phi^M$ is weakly regulated.
\end{thm}
\begin{proof}
    As every interval is a definable set in a linear order, the characteristic function of every interval is given by a definable predicate. As definable predicates are also closed under convex combinations and uniform limits, any strongly regulated function is also given by a definable predicate.

    We now show that the interpretation of any definable predicate over a discrete logic $o$-minimal structure is strongly regulated.
    Because strongly regulated functions are closed under uniform limits, it suffices to show this for formulas.
    By \cite[Remark 9.21]{mtfms}, if $\phi(x)$ is an $\mathcal{L}$-formula, then $\phi^M(x)$ takes only finitely many values, and for each $s \in [0,1]$, $(\phi^M)^{-1}(\{s\})$ is a definable set.
    As any definable set is a finite union of intervals, $\phi^M$ must be a strong step function, and thus strongly regulated.

    If instead $M$ was just weakly o-minimal, we would get a weak step function, so definable predicates must be weakly regulated.
\end{proof}

\section{o-Minimal Metric Structures}
In discrete logic, o-minimal structures can be defined as expansions of a linear order, where all definable sets in one variable are already definable, without quantifiers, using only the symbol $\leq$.

To carry this analogy to expansions of metric linear orders, we now characterize the definable predicates in one variable which are quantifier-free definable using only the symbol $r(x,y)$.
These will turn out to be precisely the metric-continuous regulated functions (Theorems \ref{lem:qf_reg}, \ref{thm:reg_qf}), setting up our proposed definitions of o-minimality and weak o-minimality in Definition \ref{defn:omin}.
For structures from discrete logic, these definitions will coincide with the existing definitions by Theorem \ref{thm:discrete_omin_reg}.

\subsection{Quantifier-Free Definable Predicates}
Now that we have introduced regulated functions in the appropriate generality, we will see how they are related to quantifier-free definable functions in metric linear orders. 

\begin{lem}\label{lem:qf_reg}
    Suppose $M$ is a metric linear order, understood as a  $\{d_\leq,r\}$-structure.
    All quantifier-free definable predicates $\phi(x)$ with parameters, for $|x| = 1$, are strongly regulated.
\end{lem}
\begin{proof}
    By Lemma \ref{lem:reg_comb}, it suffices to check that atomic formulas are strongly regulated.
    In one variable, these are only $d(x;a), d_\leq(x;a), d_\leq(a;x), r(x;a), r(a;x)$.
    Each of these can be expressed in a piecewise fashion with at most two pieces, where it is nondecreasing or nonincreasing on each piece. Monotone functions are strongly regulated, and regulated functions are closed under finite piecewise definitions, finishing the proof.
\end{proof}

\begin{lem}\label{lem:piecewise_qf_def}
    Let $M$ be a metric linear order, and let $f:M \to [0,1]$.
    Suppose there are points $a_1,\dots,a_n \in M$ such that, if we abuse notation and let $a_0 = -\infty, a_{n + 1} = +\infty$,
    then on each interval $[a_i,a_{i + 1}]$ for $0 \leq i \leq n$,
    $f$ equals a quantifier-free definable predicate with parameters in the language $\{r\}$.

    Then $f$ itself is given by a quantifier-free definable predicate with parameters in the language $\{r\}$.
\end{lem}
\begin{proof}
    By induction, we can assume $n = 1$.
    We then just need to show that if 
    $\phi, \psi$ are quantifier-free definable predicates with $\phi(a_1) = \psi(a_1)$, then the function given by
    $$f(x) = \begin{cases}
        \phi(x) & x \leq a_1\\
        \psi(x) & x \geq a_1\\
    \end{cases}$$
    is also given by a quantifier-free definable predicate.

    As both $\phi$ and $\psi$ are metric-continuous, for all $\varepsilon > 0$, there exists $\delta > 0$ such that for all $x \in M$, $d(x,a_1) \leq \delta$ implies $|\phi(x) - \phi(a_1)| + |\psi(x) - \psi(a_1)|\leq \varepsilon$.
    Thus by \cite[Proposition 2.10]{mtfms},
    there is some increasing continuous function $\alpha : [0,1] \to [0,1]$ such that $\alpha(0) = 0$ and for all $x \in M$,
    $$|\phi(x) - \phi(a_1)|+ |\psi(x) - \psi(a_1)| \leq \alpha\left(d(x,a_1)\right).$$

    Now we claim that for all $x \in M$,
    $$f(x) = \min\left(\phi(x) + \alpha(r(x,a_1)), \psi(x) + \alpha(r(a_1,x)),1\right),$$
    noting that the right side is a quantifier-free definable predicate.
    Without loss of generality, we evaluate this for $x \leq a_1$, hoping to get $\phi(x)$.
    We see that
    $\phi(x) + \alpha(r(x,a_1)) = \phi(x)$ while $\psi(x) + \alpha(r(a_1,x)) = \psi(x) +\alpha(d(a_1,x))$.
    It now suffices (as $\phi(x) \leq 1$) to show that $\phi(x) \leq \psi(x) +\alpha(d(a_1,x))$.
    This follows from the construction of $\alpha$, as
    $$|\phi(x) - \psi(x)| \leq |\phi(x) - \phi(a_1)| + |\psi(x) - \psi(a_1)| \leq \alpha\left(d(x,a_1)\right).$$
    Thus the minimum is $\phi(x)$ as desired.
\end{proof}

\begin{thm}\label{thm:reg_qf}
    Suppose $M$ is a metric linear order, understood as a $\{r\}$-structure. All strongly regulated functions $M \to M$ that are continuous with respect to the metric are quantifier-free definable predicates.
\end{thm}
\begin{proof}
    Let $f : M \to M$ be strongly regulated and metric-continuous, and fix $\varepsilon > 0$. There is a strong step function $g : M \to M$ with $\sup_x |f(x) - g(x)| \leq \varepsilon$. We may assume that there are finitely many points $a_1 < \dots < a_n$ in $M$ such that $g$ is constant on the intervals $(a_i,a_{i+1})$ for $0 \leq i \leq n$, where by abuse of notation we let $a_0 = -\infty$ and $a_{n + 1} = +\infty$, and we may assume that for $1 \leq i \leq n$, $g(a_i) = f(a_i)$. Let $s_i$ denote the constant value of $g$ on the interval $(a_i,a_{i+1})$.

    We will construct a function $h : M \to M$, given by a quantifier-free definable predicate, such that $\sup_x |g(x) - h(x)|\leq \varepsilon$. This implies that $\sup_x |f(x) - h(x)|\leq 2\varepsilon$, so $f$ is a uniform limit of quantifier-free definable predicates, and is itself quantifier-free definable.
    
    By Lemma \ref{lem:piecewise_qf_def}, to construct $h$ that is quantifier-free definable, it will suffice to define $h$ by a quantifier-free definable predicate $\phi_i$ on $[a_i,a_{i + 1}]$ for each $0 \leq i \leq n$. To ensure that these definitions agree, we will define $\phi_i$ on each of these closed intervals so that $\phi_i(a_i) = g(a_i)$ and $\phi_i(a_{i + 1}) = g(a_{i + 1})$ for $1 \leq i \leq n$.
    
    Now fix $0 \leq i \leq n$. We choose the quantifier-free definable predicate $\phi_i$ in one of several ways, depending on whether the endpoints $a_i, a_{i+1}$ are limit points of the interval $(a_i, a_{i + 1})$ with respect to the metric. In each case, we must check that $\phi_i(a_i) = g(a_i)$ and $\phi_i(a_{i + 1}) = g(a_{i + 1})$, and that $\sup_{a_i < x < a_{i + 1}}|s_i - \phi_i(x)| \leq \varepsilon$.

    First, we note that if $a_i$ is an endpoint of an open interval $I$, with $a_i$ a limit point of $I$ with respect to the metric, then by continuity of $f$, $f(a_i)$ is a limit point of the image $f(I)$. As for each point $x \in I$, $|f(x) - s_i| = |f(x) - g(x)|\leq \varepsilon$, and $g$ is constant on $I$, we see that also $|g(a_i) - s_i| = |f(a_i) - s_i|\leq \varepsilon$ across $I$.

    \begin{itemize}
        \item If both of the endpoints $a_i$ and $a_{i + 1}$ are limit points of the interval, with respect to the metric, then on this interval, we define $$\phi_i(x) = g(a_i) + \left(g(a_{i + 1}) - g(a_i)\right)\cdot \max\left(\frac{d(x,a_i)}{d(a_{i+1},a_i)},1\right).$$
        This takes the right values at both $a_i,a_{i + 1}$, and is monotone (either nondecreasing or nonincreasing) on $[a_i,a_{i + 1}]$,
        so it only takes values on the closed interval bounded by $g(a_i)$ and $g(a_{i+1})$. For any $x \in (a_i,a_{i + 1})$, we know that $|g(a_i) - s_i|,|g(a_{i + 1}) - s_i| \leq \varepsilon$, so any value of $\phi_i$ between these will also satisfy $|\phi_i(x) - s_i| \leq \varepsilon$. 
        
        \item Suppose that exactly one endpoint is a limit point of the interval between them. Without loss of generality, assume it is $a_{i + 1}$. Then if we let $t = \inf_{x \in (a_i,a_{i + 1}]}d(x,a_i)$, we find that $t > 0$.
        We then define
        $$\phi_i(x) = g(a_i) + (g(a_{i + 1}) - g(a_i))\cdot \max\left(\frac{d(x,a_i)}{t},1\right).$$
        Then $\phi_i(a_i) = g(a_i)$, but for any other $x \in (a_i,a_{i + 1}]$, $d(x,a_i) \geq t$, so 
        $$\phi_i(x) = g(a_i) + (g(a_{i + 1}) - g(a_i))\cdot 1 = g(a_{i+1}).$$
        As $a_{i + 1}$ is a limit point of this interval, we see that for $x \in (a_i,a_{i + 1})$,
        $$|\phi_i(x) - s_i| = |g(a_{i + 1}) - s_i|\leq \varepsilon.$$
        
        \item If neither endpoint is a limit point of the interval between them, let $$t = \min\left(\inf_{x \in (a_i,a_{i + 1}]}d(x,a_i),\inf_{x \in [a_i,a_{i + 1})}d(x,a_{i + 1})\right).$$
        As $t > 0$, we can define
        $$\phi_i(x) = s_i + (g(a_i) - s_i)\cdot \max\left(\frac{d(x,a_i)}{t},1\right) + (g(a_{i + 1}) - s_i)\cdot\max\left(\frac{d(x,a_{i + 1})}{t},1\right).$$
        We then find that, on our closed interval,
        $$\phi_i(x) = \begin{cases}
            g(a_i) & x = a_i\\
            s_i & a_i < x < a_{i + 1}\\
            g(a_{i + 1}) & x = a_{i + 1},\\
        \end{cases}$$
        which clearly satisfies the required inequality $\sup_{x \in (a_i,a_{i + 1})}|\phi_i(x) - s_i| \leq \varepsilon$.
    \end{itemize}
\end{proof}

\subsection{Defining o-Minimality}
Now that we have characterized quantifier-free definable predicates in one variable over metric linear orders, we use this characterization to propose definitions of o-minimality and weak o-minimality for metric structures:

\begin{defn}\label{defn:omin}
    Let $\mathcal{L}$ be a language including $r$, and let $M$ be an $\mathcal{L}$-structure expanding a metric linear order. We say that $M$ is \emph{o-minimal} when either of the following two conditions holds for every definable predicate $\phi(x)$ with $|x| = 1$:
    \begin{itemize}
        \item $\phi(x)$ is equivalent to a quantifier-free definable predicate in the reduced language $\{r\}$
        \item the interpretation of $\phi(x)$ is strongly regulated.
    \end{itemize}

    We call $M$ \emph{weakly o-minimal} when the interpretation of every definable predicate $\phi(x)$ with $|x| = 1$ is weakly regulated.
\end{defn}
We see that if $X$ has the discrete metric, these agree with the existing definitions, as strongly/weakly regulated $\{0,1\}$-valued functions are exactly strong/weak step functions.

\subsection{Monotone Decomposition}
Through the definitions in terms of weak/strong step functions, it is easy to see that (weakly) regulated functions are uniform limits of sums of monotone functions. By default, the functions we get are not definable predicates, but we will now construct a decomposition of weakly regulated definable predicates into monotone definable predicates, over the same parameters.

\begin{lem}\label{lem:weak_reg_decomp}
    Let $M$ expand a metric linear order, and let $\phi(x)$ be a definable predicate with $|x| = 1$.
    
    Let $\phi_0 = \phi$,
    and then recursively define definable predicates 
    $$\psi_n(x) = \sup_{y \leq x}\phi_n(y),$$
    and
    $$\phi_{n+1}(x) = \psi_n(x) - \phi_n(x),$$
    so that
    $$M \vDash \phi(x) = \sum_{k < n}(-1)^k\psi_k(x) + (-1)^n\phi_n(x).$$
    Then also
    \begin{itemize}
        \item For all $n$, $\psi_n(x)$ is nondecreasing.
        \item If for some $\varepsilon > 0$, $b \in M^z$, and $m \in \N$, $M$ can be partitioned into $m$ order-convex sets on which $\phi(x)$ varies by at most $\varepsilon$, then $M \vDash \sup_x\phi_m(x) \leq \varepsilon$ and $M \vDash \sup_x\psi_{m + 1}(x) \leq \varepsilon$.
    \end{itemize}
\end{lem}
\begin{proof}
    By construction, each $\phi_n,\psi_n$  is also definable, and each $\psi_n(x)$ will be nondecreasing because of its definition of a supremum over an initial segment.
    It follows by a simple induction that for all $n$,
    $$\phi(x) = \left(\sum_{k = 0}^{n - 1}(-1)^k\psi_k(x)\right) + (-1)^n\phi_n(x).$$

    Now fix $\varepsilon > 0$ and some $b \in M^z$, and assume that $M$ can be partitioned into $m$ order-convex sets $U_1 < \dots < U_m$ on which $\phi(x)$ varies by at most $\varepsilon$. We will show that $M \vDash \sup_x\phi_m(x) \leq \varepsilon$.
    
    We will show by induction on $n$ that $\phi_n$ also varies by at most $\varepsilon$ on each $U_n$, and if $a \in \bigcup_{i \leq n} U_i$, then $\phi_n(a) \leq \varepsilon$. We can deduce from this that $\sup_x \phi_m(x) \leq \varepsilon$.

    If $n = 0$, this statement holds by our assumptions.
    Assume for induction that it holds for $n$.
    Then $\phi_{n + 1} (x)= \sup_{y \leq x}\phi_n(y) - \phi_n(x)$.
    If $a \in \bigcup_{i \leq n + 1}U_i$ has $\phi_{n + 1}(a) > \varepsilon$, then there must be some $b \leq a$ such that $\phi_n(b) > \phi_n(a) + \varepsilon$.
    If $b \in \bigcup_{i \leq n}U_i$, then the induction hypothesis states that $\phi_n(b) \leq \varepsilon$, which contradicts this, so we may assume $b \in U_n$, and thus $a \in U_n$.
    This contradicts the assumption that $\phi_n$ varies by at most $\varepsilon$ on $U_n$.

    Now we also show that $\phi_{n + 1}$ varies by at most $\varepsilon$ on each $U_i$, that is, if $a,b \in U_i$, then $|\phi_{n + 1}(b) - \phi_{n+1}(a)|\leq \varepsilon$. We may assume $a < b$.
    Note that
    $$|\phi_{n + 1}(b) - \phi_{n+1}(a)|
    = \left|\psi_n(b) - \phi_n(b) - \psi_n(a) + \phi_n(a)\right|,$$
    and $\psi_n$ is nondecreasing, $0 \leq \psi_n(b) - \psi_n(a)$,
    and as $|\phi_n(a) - \phi_n(b)|\leq \varepsilon$, the only way this quantity can be greater that $\varepsilon$
    is if both $\psi_n(b) > \psi_n(a)$ and $\phi_n(a) > \phi_n(b)$.
    If this occurs, $\psi_n(b) = \phi_n(c)$ for some $a < c \leq b$, and $c \in U_i$.
    We then have 
    $$\phi_n(b) < \phi_n(a) \leq \psi_n(a) < \phi_n(c),$$
    so
    $$\left|\phi_n(c) - \phi_n(b) - \psi_n(a) + \phi_n(a)\right|
    \leq \phi_n(c) - \phi_n(b) \leq \varepsilon.$$

    By construction, $\sup_x \psi_{m + 1}(x;b) = \sup_x\phi_m(x;b)$.
\end{proof}

\begin{lem}\label{lem:weak_omin_decomp}
    Let $T$ be a theory extending $\MLO$, all of whose models are weakly o-minimal, and let $\phi(x;z)$ be a definable predicate with $|x| = 1$.
    
    Let $\phi_0 = \phi$,
    and then recursively define definable predicates 
    $$\psi_n(x;z) = \sup_{y \leq x}\phi_n(y;z),$$
    and
    $$\phi_{n+1}(x;z) = \psi_n(x;z) - \phi_n(x;z),$$
    so that
    $$M \vDash \phi(x;z) = \sum_{k < n}(-1)^k\psi_k(x;z) + (-1)^n\phi_n(x;z).$$

    Then for every $\varepsilon > 0$, there is some $m$ such that $T \vDash \sup_{x,z}\phi_m(x;z) \leq \varepsilon$, $T \vDash \sup_{x,z}\psi_{m + 1}(x;z) \leq \varepsilon$.
\end{lem}
\begin{proof}
    Suppose that this fails for some $\varepsilon$.
    Then the partial type $p(z)$ consisting of all formulas of the form $\phi_m(x;z) \geq \varepsilon$ is consistent, and thus realized by $b \in M^z$ for some $M \vDash T$.
    
    The predicate $\psi(x;b)$ is weakly regulated, so Lemma \ref{lem:weak_reg_decomp} states that there is some $m$ such that $M \vDash \sup_x \psi_m(x;b) \leq \frac{\varepsilon}{2}$, contradicting the construction of $b$.
\end{proof}

\subsection{Definable Sets}
We start by providing some examples of definable sets in metric dense linear orders:
\begin{lem}\label{lem:mdlo_def_set}
    Let $M \vDash \MDLO$. If $D \subseteq M$ is the complement of a countable union of pairwise disjoint open intervals, of which only finitely many have diameter $\geq \varepsilon$ for any fixed $\varepsilon > 0$, then $D$ is definable.
\end{lem}
\begin{proof}
    First, consider a family $(I_i : i \in S)$ of pairwise disjoint open intervals in $M$, with $I_i = (a_i,b_i)$. We claim that if $D = M \setminus \bigcup_{i \in S} I_i$, then for all $c \in I_i$, 
    $$d(c,D) = d(c,M \setminus I_i) = \min(d(c,a_i),d(c,b_i)).$$
    Let $I_i = (a_i,b_i)$, and assume $a_i < c < b_i$.
    By the density of the linear order, and the fact that the intervals are pairwise disjoint, we may assume that $a_i,b_i \not \in \bigcup_{i \in S} I_i$.
    Thus $d(c,D) \leq \min(d(c,a_i),d(c,b_i))$. Also, for all $e \in D$, as $e \not \in I_i$, we either have $e \leq a_i$, in which case $d(c,e) \geq d(c,a_i)$, or $e \geq b_i$ and $d(c,e) \geq d(c,b_i)$.

    Thus if $(I_i : i \in \N)$ is a \emph{countable} family of pairwise disjoint open intervals and $D = M \setminus \bigcup_{i = 0}^\infty I_i$, we find that $d(x,D) = \sum_{i = 0}^\infty d(x,M \setminus I_i)$, and it suffices to show that this is a definable predicate.

    Given $n \in \N$, let $S_n$ be the set of all $n$ such that $d(a_n,b_n)\geq \frac{1}{n}$. Note that this set must be finite, as $d(a_n,b_n)$ is also the diameter of the interval $I_n$. Now let
    $$\phi_n(x) = \sum_{i \in S_n} d(x, M \setminus I_i),$$
    and observe that this definable predicate is the distance to the set $M \setminus \bigcup_{i \in S_n} I_i$.
    We also find that for any $c \in M$ and any $n$, 
    $$|d(c,D) - \phi_n(c)|
    = \sum_{i \not \in S_n} d(c, M \setminus I_i),$$
    which is only nonzero when $c \in I_i$ for some $i \not \in S_n$.
    We then find that $\phi_n(c) = \min(d(c,a_i),d(c,b_i)) \leq d(a_i,b_i) < \frac{1}{n}$.
    Thus $\phi_n(x)$ converges uniformly to $d(x,D)$.
\end{proof}

In any metric linear order, if $D$ is definable, and $a \not \in D$ lies in a gap in $D$, then inequalities to $a$ split $D$ into two definable sets. This is easy to do by adding a parameter for $a$, but we can avoid this in the weakly o-minimal case.
\begin{lem}\label{lem:split_def_set}
    Let $M$ be an expansion of a model of $\MLO$. If $D \subseteq M$ is an $A$-definable set and $a \in M \setminus D$, then $D \cap (-\infty, a)$ and $D \cap (a,\infty)$ are $Aa$-definable.

    If $M$ is weakly o-minimal, then these sets are $A$-definable.
\end{lem}
\begin{proof}
    We may assume that both $D \cap (-\infty, a)$ and $D \cap (a,\infty)$ are nonempty - otherwise this is trivial.

    First, we claim that there is a nondecreasing $Aa$-definable predicate $\psi(x)$
    such that
    $$\sup_{x \in D \cap (-\infty, a)}\psi(x) < 
    \inf_{x \in D \cap (a,\infty)}\psi(x),$$
    and that if $M$ is weakly o-minimal, then we can choose $\psi(x)$ that is $A$-definable.

    In the general case, we let
    $\psi(x) = r(y,a)$.
    In the weakly o-minimal case, apply Lemma \ref{lem:weak_reg_decomp} to $\phi(x) = d(x,D)$ to see that there are monotone $A$-definable predicates $\sum_{k < n}(-1)^k\psi_k(x)$ converging uniformly to $\phi(x)$.
    We claim that for some $k$, we can let $\psi(x) = \psi_k(x)$.
    To show this, we will show by contradiction that for some $k$, $\sup_{x \in D \cap (-\infty,a)}\psi_k(x) < \psi_k(a)$.
    Otherwise, fix $n$ such that
    $$\sup_x\left|\phi(x) - \sum_{k < n}(-1)^k\psi_k(x)\right|< \frac{\phi(a)}{3}.$$
    Then for every $k < n$, there is some $b_k \in D \cap (-\infty,a)$ such that $\psi_k(b_k) > \psi_k(a) - \frac{\phi(a)}{3n}$.
    Then if $b = \max_{k < n} b_k$, we have 
    $$\left|\sum_{k < n}(-1)^k\psi_k(b) - \sum_{k < n}(-1)^k\psi_k(a)\right| < n\frac{\phi(a)}{3n},$$
    from which we conclude that $\left|\phi(b) - \phi(a)\right| < \phi(a)$,
    contradicting the fact that $\phi(b) = 0$.

    We may assume that for any $b \in D$, if $b < a$, we have $\psi(b) = 0$, and if $a < b$, we have $\psi(b) = 1$.
    To do this, we can let $r = \sup_{y \in D \cap (-\infty, a)}\psi(y)$, $s =\inf_{y \in D \cap (a,\infty)}\psi(y)$, and then replace $\psi(x)$ with
    $$\max\left(\frac{\psi(x)\dot-r}{s},1\right),$$
    where $r$ and $s$ are constants, simply acting as part of a binary continuous connective.

    Now consider the definable predicate
    $$\theta(x) = \inf_{y \in D}\max\left(d(x,y),\psi(y)\right),$$
    which uses the same parameters as $d(x,D)$ and $\psi$.
    
    Given $b \in M$, $c \leq a$, we find that $\max\left(d(b,c),\psi(c)\right) = d(b,c)$,
    and given $c \in D \cap (a,\infty)$, we find that $\psi(c) = 1$, so $\max\left(d(b,c),\psi(c)\right) = 1$.
    Thus
    \begin{align*}
        \inf_{y \in D}\max\left(d(x,y),\psi(y)\right) 
        &= \min\left(\inf_{y \in D \cap (-\infty,a)}\max\left(d(x,y),\psi(y)\right), \inf_{y \in D \cap (a,\infty)}\max\left(d(x,y),\psi(y)\right)\right)\\
        &= \min\left(\inf_{y \in D \cap (-\infty,a)}d(x,y), 1\right)\\
        &= \inf_{y \in D \cap (-\infty,a)}d(x,y),
    \end{align*}
    so $D \cap (-\infty,a)$ is definable by $\theta(x)$.
\end{proof}

\begin{thm}
    If $M$ is a weakly o-minimal metric structure expanding $\UDLO$, $A \subseteq M$ and $D \subseteq M$ is compact and $A$-definable, then for any $b \in D$, $b$ is $A$-definable. Thus $\mathrm{acl}(A) = \mathrm{dcl}(A)$.
\end{thm}
\begin{proof}
    Suppose $b \in D$. Because $M \vDash \UDLO$ and $D$ is compact, $D$ has empty interior (in the order/metric topology).

    For each $n$, as $b$ is not in the interior of $D$, we may find points $a_n^-,a_n^+ \not \in D$ such that $a_n^- < b < a_n^+$ and $d(a_n^-,a_n^+) < \frac{1}{n}$.
    By applying Lemma \ref{lem:split_def_set} to $a_n^-$ and then $a_n^+$, we see that $D \cap (a_n^-, a_n^+)$ is $A$-definable by a predicate $\phi_n(x)$.

    We claim that $\phi_n(x) \to d(x,b)$ uniformly.
    For each $\varepsilon > 0$, if $\frac{1}{n}<\varepsilon$,
    then for any $c \in M$ and $e \in D \cap (a_n^-, a_n^+)$, $d(c,e) \leq d(c,b) + d(b,e) < d(c,b) + \frac{1}{n}$, so $\phi(c) = d(c,D \cap (a_n^-, a_n^+)) < d(c,b) + \varepsilon$.
    Thus $d(x,b)$ is a limit of $A$-definable predicates, so $b$ is $A$-definable.
\end{proof}

We can now describe all definable sets in weakly o-minimal expansions of $\MDLO$ in terms of their complements:
\begin{lem}\label{lem:weak_omin_def_set}
    In a weakly o-minimal expansion $M$ of an $\MDLO{}$, if $D \subseteq M$ is a definable set, then $M \setminus D$ is a countable union of order-open order-convex sets, which are disjoint and separated by elements of $D$, of which only finitely many have diameter $\geq \varepsilon$ for any fixed $\varepsilon > 0$.
\end{lem}
\begin{proof}
    As $D$ is closed, its complement is open, and an open set is a union of open intervals, which are order-convex and open. By taking unions, we may assume each open order-convex subset of $M \setminus D$ in the decomposition is maximal, which makes them disjoint and separated by elements of $D$. 

    Let $U \subseteq M \setminus D$ be an open order-convex set. Its diameter is $\sup_{x,y \in U}d(x,y)$, so if this diameter is $> \varepsilon$, then there are $a < b$ in $U$ with $d(a,b) > \varepsilon$. By the density assumptions of \MDLO{}, there is some $c \in (a,b)$ with $\frac{1}{3}d(a,b) < d(a,c) < \frac{2}{3}d(a,b)$, and the triangle inequality mandates that $\frac{1}{3}d(a,b) < d(b,c) < \frac{2}{3}d(a,b)$ also.
    Then $d(c,D) = \inf_{d_0 \in D}, d(c,d_0)$, but for each $d_0 \in D$, either $d_0 < U$, in which case $d_0 < a < c$ and $d(d_0,c) > d(a,c) > \frac{1}{3}d(a,b)$, or $d_0 > U$, in which case $c < b < d_0$, so $d(d_0,c) > d(b,c) > \frac{1}{3}d(a,b)$.
    In either case, we see that $d(c,D) \geq \frac{1}{3}d(a,b) \geq \frac{\varepsilon}{3}$.
    This tells us that in any set $U$ in the decomposition of $M \setminus D$, the predicate $d(x,D)$ reaches at least $\frac{1}{3}$ the diameter of $U$.
    As this predicate is weakly regulated, there is a maximum finite number of times it can alternate on any increasing sequence between being at least $\frac{1}{3}\varepsilon$ and being 0. Thus this maximum is also a maximum on the number of sets $U$ in the decomposition which have diameter greater than $\varepsilon$.
\end{proof}

If we strengthen our assumptions to o-minimality, we can now prove definable completeness, which will let us fully characterize definable sets.

\begin{lem}\label{lem:def_complete}
Let $M$ be an o-minimal metric structure and $D\subset M$ a definable set. If $D$ is bounded above (resp. below), then $D$ has a least upper bound (resp. greatest lower bound).
\end{lem}

\begin{proof}
As $D$ is definable, we can quantify over $D$. Thus, there is a definable predicate $\psi$ such that  
$$\psi^M(x) = \inf_{y\in D} r^M(x,y)$$
and $\psi^M$ is strongly regulated by o-minimality of $M$. Also, observe that $D$ is bounded above in $M$ exactly when $M\models \sup_x \psi(x) >0$, so we will assume so and break into two cases:\\\par

\textbf{Case 1:} For every $\epsilon>0$, there is some $x\in M$ such that $0<\psi^M(x)<\epsilon$.\par
We begin by constructing a sequence $\{x_n\}\subset M$ such that $0<\psi^M(x_n)<2^{-n}$ and $x_{n+1}<x_n$. We claim that the sequence is Cauchy in the metric: take $m>n$ so that $x_m<x_n$ and fix some $d\in D$ with $d(x_n, y)<2^{-n}$ (we can do this because $\psi^M(x_n)<2^{-n}$). As $d(\cdot, y)$ is nondecreasing on $[y,\infty)$ and $y<x_m<x_n$, we must have that $d(x_m,y)<2^{-n}$. Thus, 
$$d(x_m,x_n)\leq d(x_n,y)+d(x_m,y)< 2^{-n+1}.$$
Next, let $m=\lim x_n$. It follows that $\psi^M(m)=0$, so $m\in D$ as we can find elements of $D$ arbitrarily close to $m$ in the metric and $D$ is metrically closed. We claim that $m$ is the maximum of $D$: otherwise, we could find some $z\in D$ so that $m<z$. By the construction of the Cauchy sequence, we must also have that $m<z<x_n$ for all $n$. But since $d(\cdot, x_n)$ is nonincreasing on $(-\infty, x_n)$, we must have that $x_n\to z$ on the metric, so $z=\lim x_n$ and $z=m$, a contradiction.\\\par

\textbf{Case 2:} There is an $\epsilon>0$ such that $\psi^M(x)<\epsilon$ implies $\psi^M(x)=0$ for every $x\in M$.\par
Since $\psi^M$ is strongly regulated, we can partition $M$ into finitely many intervals $I_0<I_1<\dotsc<I_n$ so that on each interval either $\psi^M<2\epsilon/3$ or $\psi^M>\epsilon/3$. By assumption, the first case implies $\psi^M=0$.\par
Let $I_k$ be the rightmost interval such that $\psi^M=0$. Thus, $I_{k+1}$ is the leftmost interval so that $\psi^M>\epsilon/3$. In particular, $D<I_{k+1}$. There are a few possibilities here:
\begin{itemize}
\item $I_k$ is closed on the right. Then $I_k$ has a maximum $m$, so $m\in D$ and it is the maximum of $D$. 
\item $I_k$ is open on the right, and $I_{k+1}$ is closed on the left with a minimum $m$. Then $D$ has no maximum. However, $m$ is a upper bound of $D$ and it is the least upper bound as $m$ is the leftmost element of the leftmost interval above $D$. 
\item $I_k$ is open on the right and $I_{k+1}$ is open on the left. Then, there is a point $p$ such that $I_k<p<I_{k+1}$, but this contradicts the partition of $M$ into intervals and our choice of intervals. 
\end{itemize}
\end{proof}

\begin{thm}\label{thm:omin_def_set}
    In an o-minimal expansion $M$ of an $\MDLO$, a set $D \subseteq M$ is definable if and only if $M \setminus D$ is a countable union of disjoint open intervals, of which only finitely many have diameter $\geq \varepsilon$ for any fixed $\varepsilon > 0$.
\end{thm}
\begin{proof}
    One direction is Lemma \ref{lem:mdlo_def_set}.

    Now assume that $D \subseteq M$ is definable.
    By Lemma \ref{lem:weak_omin_def_set}, write $D = M \setminus \bigcup_{i = 0}^\infty U_i$, where each $U_i$ is an open order-convex set,
    and for each distinct nonempty $U_i, U_j$, there is some $e \in D$ between them.

    Now fix $i$, and we will show that $U_i$ is an open interval. It suffices to show that if it is bounded above/below, it has a least upper/greatest lower bound. We may also assume $U_i$ is nonempty, so let $a \in U_i$.
    By Lemma \ref{lem:split_def_set}, both $D \cap (-\infty,a)$ and $D \cap (a,\infty)$ are definable. If $U_i$ is bounded above, then $D \cap (a,\infty)$ is nonempty and bounded below by $a$, so by Lemma \ref{lem:def_complete}, it has a greatest lower bound, call this $b$.
    
    We claim that $b$ is the least upper bound for $U_i$.
    If $c \in U_i$, then by order convexity, $(c,a)$ or $(a,c) \subseteq U_i$, so in either case, $c < b$, so $b$ is an upper bound.
    To see it is the least upper bound, note that if $a < c < b$, then the entire interval $(c,b)$ is less than $D \cap (a,\infty)$, so it does not intersect $D$. IT is thus contained in $U_i$, so $c$ cannot be an upper bound for $U_i$. 

    Similarly, we can show that if $U_i$ is bounded below, it has a greatest lower bound, and we conclude that $U_i$ is an interval.
\end{proof}

We have classified definable subsets of an o-minimal expansion of a metric dense linear order, similarly to the classical case, where such sets are exactly finite unions of intervals. However, not everything works as smoothly as in the classical case.

In classical o-minimality, one of the most basic facts is that a definable set in dimension 1 either contains an open interval or is finite.
We could characterize containing an open interval as having nonempty interior in the order topology, and the standard analog of finiteness in continuous logic is compactness with respect to the metric. In this subsection, we demonstrate that in an o-minimal metric structure, it is possible to construct a definable set in dimension 1 with empty interior that is not compact. This suggests that some basic tools of o-minimality will have to be developed differently.

\begin{thm}
    Let $M \vDash \UDLO$ be separable.
There is a definable subset $D \subseteq M$ that has empty interior but is not compact.
\end{thm}
\begin{proof}
Let $(p_i : i \in \N)$ be dense in $M$.
We will construct a sequence $(I_i : i \in \N)$ of (possibly empty) intervals $I_i = (a_i,b_i)$ recursively, such that for all $n$, either $I_n$ is nonempty with $a_n = p_n$, or $p_n \in \bigcup_i I_i$. We will then define $D = M \setminus \bigcup_{i = 0}^\infty I_i$.
To make this definable, it will suffice for the intervals to be pairwise disjoint, with the diameter of $I_n$ bounded by $\frac{1}{n + 2}$.

To start defining our intervals, let $I_0 = (a_0,b_0)$, where $a_0 = p_0$, $b_0 > a_0$ and $d(a_0,b_0)< \frac{1}{2}$.
Given $I_0,\dots,I_{n-1}$, if $p_n \in \bigcup_{i < n} [a_i,b_i]$, then let $I_n = \emptyset$. In this case, we will end up with $p_n \in \bigcup_i I_i$.
Otherwise, let $a_n = p_n$, let $\varepsilon = d\left(p_n, \bigcup_{i < n} [a_i,b_i]\right)$, and let $b_n > p_n$ be such that $d(a_n,p_n) < \min\left(\varepsilon, \frac{1}{n + 2}\right)$.
This ensures that all of the intervals are disjoint, and that the diameter of $I_n$ is at most $\frac{1}{n + 2}$, ensuring definability.

\begin{claim}
    $D$ has empty interior.
\end{claim}
\begin{proof}
    Equivalently, $\bigcup_{i = 0}^\infty I_i$ is dense.
    We show this by showing that each $p_n$ is contained in the closure of this set. This happens because either $p_n \in \bigcup_{i = 0}^\infty I_i$ directly, or $I_n$ is nonempty and $p_n = a_n$, in which case $p_n \in [a_n,b_n]$, which is contained in the closure of $\bigcup_{i = 0}^\infty I_i$.
\end{proof}

\begin{claim}
    $D$ is not compact.
\end{claim}
\begin{proof}
    We will show that every ball of radius $\frac{3}{4}$ contains a point of $D$.
    As $M$ can be partitioned into infinitely many disjoint such balls, this gives an open cover of which no finite subset covers $D$.

    Each ball of radius $\frac{3}{4}$ contains a point $p_n$ from the dense sequence. If $I_n \neq \emptyset$, then $a_n = p_n \in E$, and we are done. Otherwise, there is some $m < n$ with $I_m$ nonempty such that $p_n \in [a_m,b_m]$. As $d(a_m,b_m) < \frac{1}{2}$, $d(a_m,a_n) < \frac{1}{2}$, so $a_m = p_m$ is also in this ball, and in $D$.
\end{proof}
\end{proof}

\subsection{Definable Functions}
We now observe that definable functions in o-minimal structures themselves satisfy an analog of Lemma \ref{lem:reg_partition}:

\begin{lem}\label{lem:func_reg}
    Let $f : M \to M$ be definable in an o-minimal metric structure $M$. Then for all $s < t$ in $M$, there is a partition of $X$ into finitely many intervals such that on each interval $I$ in the partition, either $f(x) > s$ or $f(x) < t$.
\end{lem}
\begin{proof}
    Fix $s < t$ in $M$. We wish to show that $M$ can be partitioned into finitely many intervals such that on each interval, either $f(x) > s$ or $f(x) < t$.

    To do this, we consider the predicate $\phi(x) = r(s,f(x))$, which is definable as $f(x)$ is a definable function. As this is a definable predicate, it is a strongly regulated function itself, so we can partition $M$ into finitely many intervals such that on each interval, $\phi(x) > 0$ or $\phi(x) < d(s,t)$.

    Now on any interval in the partition, if $\phi(x) > 0$, then $f(x) > s$. If instead, $\phi(x) < d(s,t)$, then either $f(x) \leq s$, or $\phi(x) = d(r,f(x)) < d(s,t)$, and in either case, we can deduce that $f(x) < t$.
\end{proof}

\subsection{Distality}
As in classical logic, weak o-minimality (on the level of theories) is enough to imply distality.
\begin{thm}\label{thm:distal}
    If $T$ is a theory expanding $\MLO$ such that all models of $T$ are weakly $o$-minimal, then $T$ is distal.
\end{thm}
\begin{proof}
    By \cite[Theorem 5.22]{anderson1}, it suffices to show that if $\mathcal{I} = (a_i : i \in \Q)$ is an indiscernible sequence over $M \vDash T$, and $b \in M$ is a singleton such that $\mathcal{I}$ with $a_0$ removed is indiscernible over $b$, then $\mathcal{I}$ is indiscernible over $b$.
    Fix $\phi(x;y)$, and we wish to show that $i \mapsto \phi(a_i;b)$ is constant.
    We know that $\phi(a_0;y)$, treated as a predicate on $y$, is weakly regulated, so by Lemma \ref{lem:weak_omin_decomp}, we can decompose $\phi(x;y)$ as a continuous combination of nondecreasing predicates.
    Thus we may assume $\phi(x;y)$ itself is nondecreasing in $y$ for every value of $x$.

    For contradiction, without loss of generality, we may assume that $\phi(a_0;b) < \phi(a_1;b)$.
    Fix real numbers $t_0, t_1$ with
    $\phi(a_0;b) < t_0 < t_1 < \phi(a_1;b)$,
    and define a new predicate
    $$\theta(x_0,x_1) = \sup_y\min\left(t_0 \dot- \phi(x_0;y),\phi(x_1;y)\dot-t_1\right).$$
    Note that $\theta(x_0,x_1) > 0$ is equivalent to there existing some value of $y$ such that $$\phi(x_0;y) < t_0< t_1 < \phi(x_1;y).$$
    In particular,
    $\theta(a_0,a_1) > 0$,
    so by indiscernibility, $\theta(a_{-1},a_0) > 0$, so there is some $b'$ such that $$\phi(a_{-1};b') < t_0 < t_1 < \phi(a_0;b').$$
    Because $\phi(a_1;y)$ is monotone and $$\phi(a_{-1};b') < t_1 < \phi(a_1;b) = \phi(a_{-1};b),$$
    we see that $b' < b$, but this means that 
    $t_1 < \phi(a_0;b') \leq \phi(a_0;b) < t_0$,
    a contradiction since we assume $t_0 < t_1$.
\end{proof}

\section{Cyclic Orders}

In order to correctly describe the order on ordered real closed metric valued fields, we will need to describe the ordering on a projective line. The ordering on such a projective line is most naturally described as a cyclic (or circular) ordering, rather than linear. Cyclic orders are usually defined in terms of the strict relation (analogous to $<$) instead of the nonstrict relation (analogous to $\leq$). We recall that definition:

\begin{defn}
    A relation $\cyc(x,y,z)$ on a set $M$ is a \emph{strict cyclic order} relation when it satisfies the following axioms for all $w,x,y,z$:
    \begin{itemize}
        \item if $\cyc(x,y,z)$, then $\cyc(y,z,x)$
        \item if $\cyc(x,y,z)$, then $\neg\cyc(x,z,y)$
        \item if $\cyc(w,x,y)$ and $\cyc(w,y,z)$,
        then $\cyc(w,x,z)$
        \item if $x,y,z$ are distinct, then $\cyc(x,y,z)$ or $\cyc(x,z,y)$.
    \end{itemize}
\end{defn}

One can define the nonstrict version equivalently by $\ceq(x,y,z) \iff \neg\ceq(x,z,y)$ or $$\ceq(x,y,z) \iff \cyc(x,y,z) \textrm{ or }x = y\textrm{ or }x = z\textrm{ or }y = z.$$
It is straightforward to translate the theory of cyclic orders to this language:

\begin{defn}
    A relation $\ceq(x,y,z)$ on a set $M$ is a nonstrict cyclic order when it satisfies the following axioms for all $w,x,y,z$:
    \begin{itemize}
        \item $\ceq(x,x,y)$
        \item if $\ceq(x,y,z)$, then $\ceq(y,z,x)$
        \item if $\ceq(x,y,z)$ and $\ceq(x,z,y)$, then $x = y, x = z,$ or $y = z$
        \item if $\ceq(w,x,z)$, then either $\ceq(w,x,y)$ or $\ceq(w,y,z)$
        \item either $\ceq(x,y,z)$ or $\ceq(x,z,y)$.
    \end{itemize}
\end{defn}

We can consider cyclic orders on complete metric spaces as metric structures in the language $\Lceq$ consisting of just the distance predicate $d_\ceq$ to the set $\{(x,y,z) \in M^3: \ceq(x,y,z)\}$. As in the case of linear orders, we will also ask for a convexity condition, using the definition of order-convex sets from \cite{kulpmac}:

\begin{defn}
    If $M$ is cyclically ordered, a set $S \subseteq M$ is \emph{order-convex} when for every $a, b \in S$, either $\ceq(a,c,b)$ implies $c \in S$, or $\ceq(b,c,a)$ implies $c \in S$.

    A \emph{metric cyclic order} is a complete metric space equipped with a cyclic order with respect to which each open ball is order-convex.
\end{defn}

\begin{lem}
    In a metric cyclic order, the set $\ceq$ of triples is metric-closed.
\end{lem}
\begin{proof}
    This amounts to showing the set $\cyc$ of triples is open. Suppose that $\cyc(a,b,c)$. Then there is some  
\end{proof}

\subsection{Inequalities}
Before attempting to axiomatize metric cyclic orders, we need to establish some basic inequalities that the metric and the function $d_\ceq$ satisfy. To state this inequalities, we need the notion of cyclic order for tuples:
\begin{defn}
    Say a tuple $(a_1,\dots,a_n)$ of elements of a cyclic order is \emph{in cyclic order} when for any $i,j,k \in \{1,\dots,n\}$ that are themselves in cyclic order (giving $\{1,\dots,n\}$ the standard cyclic order), $\ceq(a_i,a_j,a_k)$.
\end{defn}

When dealing with four points $(w,x,y,z)$ in cyclic order, we imagine them as a quadrilateral of points arranged around a circle.
In that visualization, the following lemma states that a diagonal of this quadrilateral incident to vertex $w$ is at least as long as one of the sides incident to $w$.
\begin{lem}\label{lem:cyclic_four_points}
    If $(w,x,y,z)$ in a metric cyclic order $M$ are in cyclic order, then $$d(w,y) \geq \min\left(d(w,x),d(w,z)\right).$$
\end{lem}
\begin{proof}
    If not, then we may choose a real number $t$ with $d(w,y) < t < \min\left(d(w,x),d(w,z)\right)$.
    However, then the ball of radius $t$ around $w$ contains $w,y$ but not $x,z$, and is thus not order-convex.
\end{proof}

In fact, this is enough to guarantee convexity.
\begin{lem}\label{lem:conv_cond}
    Suppose $\ceq$ is a nonstrict cyclic order on $M$, equipped with a metric $d$, such that for all $(w,x,y,z) \in M^4$ in cyclic order, $$d(w,y) \geq \min\left(d(w,x),d(w,z)\right).$$
    Then $M$ with $d$ and $\ceq$ is a metric cyclic order.
\end{lem}
\begin{proof}
    Suppose that the convexity condition fails. Then there is an open ball, say $B_\varepsilon(w)$ for $\varepsilon > 0, w\in M$, that is not order-convex. Thus there are $u,v \in B_\varepsilon(w)$ and $m_1,m_2 \in M$ with $\ceq(u,v,m_1)$ and $\ceq(v,u,e_2)$ such that $d(w,m_1),d(w,m_2) \geq \varepsilon$.
    This means that $(u,e_2,v,e_1)$ is in cyclic order, so depending on whether $\ceq(w,e_1,e_2)$, either $(w,e_2,v,e_1)$ or $(w,e_1,u,e_2)$ is.
    Whichever of these tuples is in cyclic order, we rename $(w,x,y,z)$.
    Then we have $d(w,y) < \varepsilon$, but $d(w,x),d(w,z) \geq \varepsilon$, contradicting our assumption.
\end{proof}

\begin{lem}\label{lem:cyclic_four_points}
    If $(w,x,y,z)$ are in cyclic order, then $$d(w,y) \geq \min\left(d(w,x),d(w,z)\right).$$
\end{lem}
\begin{proof}
    If not, then we may choose a real number $t$ with $d(w,y) < t < \min\left(d(w,x),d(w,z)\right)$.
    However, then the ball of radius $t$ around $w$ contains $w,y$ but not $x,z$, and is thus not order-convex.
\end{proof}

In order to axiomatize metric cyclic orders, we first need to understand what the predicate $d_\ceq(x,y,z)$ will do when $\ceq(x,y,z)$ is false. To do this outside the ultrametric case, we need to build on Definition \ref{defn:d_delta}:

\begin{defn}
    Define the formulas $$d_\Delta(x,y) = \inf_z(\max(d(x,z),d(y,z)))$$
    and 
    $$d_{\Delta_3}(x,y,z) = \min(d_\Delta(x,y),d_\Delta(y,z),d_\Delta(z,x)).$$
\end{defn}
\begin{lem}
    The formula $d_{\Delta_3}$ is the distance predicate to the definable set $\Delta_3$ of triples $(a,b,c) \in M^3$ such that $a,b,c$ are not all distinct.
\end{lem}
\begin{proof}
    The set $\Delta_3$ is the union of three definable sets, each of which is a permutation of the variables of $\Delta \times M$, so it is definable, and its distance predicate is just the minimum of the three distance predicates to the permutations of $\Delta \times M$.
\end{proof}

The following lemma says that for any quadruple in cyclic order, the maximum of the diagonal distances of the quadrilateral is at least the distance from any side to $\Delta$.
\begin{lem}\label{lem:cyclic_four_points_diag}
    If $(w,x,y,z)$ are in cyclic order, then 
    $$\max\left(d(w,y),d(x,z)\right) \geq d_\Delta(y,z).$$
\end{lem}
\begin{proof}
    By Lemma \ref{lem:cyclic_four_points}, we see that either $d(w,y) \geq d(y,z)$ or $d(w,y) \geq d(x,y)$. If the former is true, then
    $$\max\left(d(w,y),d(x,z)\right) \geq d(w,y) \geq d(y,z) \geq d_\Delta(y,z).$$
    In the latter case, where $d(w,y) \geq d(x,y)$,
    we find that
    $$d_\Delta(y,z) \leq \max(d(x,y),d(x,z)) \leq \max\left(d(w,y),d(x,z)\right).$$    
\end{proof}

\begin{lem}\label{lem:cyclic_distance}
    In a metric cyclic order, either $(x,y,z) \in \ceq$ or
    $$d_\ceq(x,y,z) = \begin{cases}
        0 & \textrm{if }\ceq(x,y,z)\\
        d_{\Delta_3}(x,y,z) & \textrm{otherwise}.
    \end{cases}$$
\end{lem}
\begin{proof}
    As $\Delta_3 \subseteq \ceq \subseteq M^3$, we find that $d_\ceq(x,y,z) \leq d_{\Delta_3}(x,y,z)$. Thus it suffices to show that for every $(x,y,z)$, either $\ceq(x,y,z)$ or for every $(u,v,w) \in \ceq$,
    $$d_{\Delta_3}(x,y,z) \leq d((u,v,w),(x,y,z)).$$

    Fix some $(u,v,w) \in \ceq$.
    Either $\ceq(u,v,z)$ or $\cyc(u,z,v)$.
    In the former case, as $d((u,v,z),(x,y,z)) \leq d((u,v,w),(x,y,z))$, we may assume that $w = z$.
    By the same argument, we see that either
    \begin{itemize}
        \item $u = x$ or $\cyc(v,x,w)$
        \item $v = y$ or $\cyc(w,y,u)$
        \item $w = z$ or $\cyc(u,z,v)$.
    \end{itemize}
    We now do casework on the 8 options from these three dichotomies. By symmetry, we only need to consider how many of the equalities hold.
    
    If the equality holds in all cases, we see that $\ceq(x,y,z)$ holds.

    If two of the equalities hold, we assume that $u = x, v = y,$ and $\cyc(u,z,v)$.
    In this case, $d((u,v,w),(x,y,z)) = d(w,z)$. As $(u,z,v,w) = (x,z,y,w)$ is in cyclic order, we see that by Lemma \ref{lem:cyclic_four_points}, 
    $$d(w,z) \geq \min(d(y,z),d(x,z))\geq d_{\Delta_3}(x,y,z),$$
    because $d(x,z)$ is the distance to $(x,y,x) \in \Delta_3$ and $d(y,z)$ is the distance to $(x,y,y) \in \Delta_3$.
    
    If one of the equalities holds, we assume that $u = x, \cyc(w,y,u),$ and $\cyc(u,z,v)$.
    In this case, $d((u,v,w),(x,y,z)) = \max(d(v,y),d(w,z))$. The sequence $(w,y,x,z,v)$, and thus its subsequence $(w,y,z,v)$ is in cyclic order, so by Lemma \ref{lem:cyclic_four_points_diag}, $$\max\left(d(v,y),d(w,z)\right) \geq d_\Delta(y,z) \geq d_{\Delta_3}(x,y,z).$$
    
    If none of the equalities hold, then the $\cyc$ relation holds in all cases, then we see that the sequence $(u,z,v,x,w,y,u)$ is in cyclic order, which also implies that $\ceq(x,y,z)$ holds.
\end{proof}

\begin{lem}\label{lem:cyclic_closed}
    In an ultrametric cyclic order, the set of triples satisfying $\ceq$ is metric-closed.
\end{lem}
\begin{proof}
    In general, the zero-set of $d_\ceq$ is the metric-closure of the set of triples satisfying $\ceq$.
    By Lemma \ref{lem:cyclic_distance}, we see that if $d_\ceq(a,b,c) = 0$, then $(a,b,c) \in \ceq$, because otherwise, we would have $d_\Delta(a,b,c) = 0$, so three elements $a,b,c$ are not distinct, which still implies $(a,b,c) \in \ceq$.
    Thus this set equals its closure, the zero-set of $d_\ceq$.
\end{proof}

\subsection{Axiomatizing Metric Cyclic Orders}
We can now axiomatize metric cyclic orders.
The first two axioms are to ensure that $d_\ceq$ is the distance predicate to its zero set. These are instances of conditions $E_1$ and $E_2$ from \cite[Theorem 9.12]{mtfms}, and in that notation, $\bar x$ and $\bar y$ are triples of variables, while $d(\bar x,\bar y) = \max(d(x_1,y_1), d(x_2,y_2), d(x_3,y_3))$.

Because of these axioms, we will be able to assume that $\ceq$ is definable, and thus by \cite[Theorem 9.17]{mtfms}, we can quantify over $\ceq$ in the subsequent axioms.

\begin{lem}
    A metric structure in the language $\Lceq$ is the structure given by a metric cyclic order if and only if it satisfies the following theory, \MCO{}:
    \begin{enumerate}[label=(\arabic*)]
        \item\label{item:mco_dist1} $\sup_{\bar x}\inf_{\bar y}\max(d_\ceq(\bar y),|d_\ceq(\bar x) - d(\bar x, \bar y)|) = 0$
        \item\label{item:mco_dist2} $\sup_{\bar x}|d_\ceq(\bar x) - \inf_{\bar y}\min(d_\ceq(\bar y) - d(\bar x, \bar y),1)| = 0$
        \item\label{item:mco_cyc} $\sup_{x,y,z} |d_\ceq(x,y,z) - d_\ceq(y,z,x)| = 0$ (Cyclicity)
        \item\label{item:mco_asym_refl} $\sup_{x,y,z} |(d_\ceq(x,y,z) + d_\ceq(y,x,z)) - d_{\Delta_3}(x,y,z)| = 0$ (Antisymmetry, Reflexivity)
        \item\label{item:mco_tot} $\sup_{x,y,z} \min\{d_\ceq(x,y,z), d_\ceq(y,x,z)\} = 0$ (Totality)
        \item\label{item:mco_trans} $\sup_{w,x,z \in \ceq}\min(d_\ceq(w,x,y),d_\ceq(w,y,z)) = 0$ (Transitivity)
        \item\label{item:mco_conv} $\sup_{w,x,y,z}\min\left(\ceq(w,y,x),\ceq(w,z,y),\left(\min\left(d(w,x),d(w,z)\right)\dot{-}d(w,y)\right)\right) = 0$ (Convexity)
    \end{enumerate}
\end{lem}
\begin{proof}
    First, assume that $M$ is a metric cyclic order with metric $d$ and nonstrict cyclic ordering $\ceq$. If we interpret $d_\ceq$ as the distance predicate to the closed set $\ceq$, it will satisfy Axioms \ref{item:mco_dist1} and \ref{item:mco_dist2}.
    
    Axiom \ref{item:mco_cyc} follows from the cyclic symmetry of $\ceq$.
    By Lemma \ref{lem:cyclic_distance}, $d_\ceq(x,y,z) = 0$ is equivalent to $\ceq(x,y,z)$, so totality of $\ceq$ implies Axiom \ref{item:mco_tot}.
    
    By antisymmetry and reflexivity, either $(x,y,z)$ represents a tuple of nondistinct elements, in which case $d_\ceq(x,y,z) = d_\ceq(y,x,z) = 0$ and  Axiom \ref{item:mco_asym_refl} holds, or it represents distinct elements, in which case precisely one of $\ceq(x,y,z), \ceq(y,x,z)$ holds. By Lemma \ref{lem:cyclic_distance}, for one of these, $d_\ceq$ is 0, and for the other, it is $d_{\Delta_3}(x,y,z)$, satisfying Axiom \ref{item:mco_asym_refl}.

    To check transitivity, we quantify over $\ceq$. Axiom \ref{item:mco_trans} states that for all $w,x,z$, if $\ceq(w,x,z)$, then $\ceq(w,x,y)$ or $\ceq(w,y,z)$, the contrapositive of the transitivity of $\cyc$.

    Given $w,x,y,z$, either $\ceq(w,y,x), \ceq(w,z,y)$, or we have both $\cyc(w,x,y)$ and $\cyc(y,z,w)$, from which we can deduce that $(w,x,y,z)$ is in cyclic order, and then deduce by Lemma \ref{lem:cyclic_four_points} that $d(w,y) \geq \min(d(w,x),d(w,z))$, allowing us to deduce Axiom \ref{item:mco_conv}.

    Now let $(M;d,d_\ceq)$ be a metric structure satisfying these axioms.
    By Axioms \ref{item:mco_dist1} and \ref{item:mco_dist2}, $d_\ceq$ is the distance predicate to its zero set, the definable relation $\ceq$, which we must show is a metric cyclic order with $d$.
    By Axiom \ref{item:mco_cyc} this relation is cyclic, and for $a,b,c \in M$, by Axiom \ref{item:mco_tot}, we see that either $\ceq(a,b,c)$ or $\ceq(b,a,c)$ holds, so $\ceq$ is appropriately total.
    
    By Axiom \ref{item:mco_asym_refl} and totality, for each $a,b,c \in M$, as $\ceq(a,b,c) + \ceq(b,a,c) = d_{\Delta_3}(a,b,c)$ and one of these values is 0, the other must be $d_{\Delta_3}(a,b,c)$. Thus both of these are equal to 0 if and only if $d_{\Delta_3}(a,b,c) = 0$, which happens precisely when $a,b,c$ are not distinct. This makes $\ceq$ reflexive and antisymmetric.

    Axiom \ref{item:mco_trans} shows that for any $a,b,c,d \in M$, if $\ceq(a,b,d)$, then either $\ceq(a,b,c)$ or $\ceq(a,c,d)$, satisfying the transitivity axiom.

    Thus far we have checked that $\ceq$ is a nonstrict cyclic order.
    Axiom \ref{item:mco_conv} implies that for any pairwise distinct $(a,b,c,d) \in M^4$ in cyclic order, the condition of Lemma \ref{lem:conv_cond} holds. It holds trivially if they are not pairwise distinct, so $d,\ceq$ satisfy the convexity condition, and this is a metric cyclic order.
\end{proof}

\subsection{Rolling Up Linear Orders}

Given a linear order $(M,\leq)$, it can be interpreted as a cyclic order by defining
$$\ceq(x,y,z) \iff (x \leq y \leq z) \lor (y \leq z \leq x) \lor (z \leq x \leq y).$$
We call this procedure \emph{rolling up} a linear order.
In the case of metric linear orders, we can again define the roll-up as a metric cyclic order structure:

\begin{lem}
    If $M$ is a metric linear order,
    then the roll-up of $M$ with the same metric is a metric cyclic order, which is a reduct of the metric linear order structure.
\end{lem}
\begin{proof}
    To check that the roll-up is a metric cyclic order, it suffices to check the cyclic convexity property, in the form given in Lemma \ref{lem:conv_cond}.
    Suppose $(w,x,y,z) \in M^4$ are in cyclic order - we wish to check that $d(w,y) \geq \min(d(w,x),d(w,z))$.
    Some cyclic permutation of $(w,x,y,z)$ is in increasing order. In all four of these cases, we see that either $w \leq x \leq y$, so $d(w,x) \leq d(w,y)$ by the metric linear order convexity property, or $y \leq z \leq w$, so $d(w,z) \leq d(w,y)$.

    For the reduct claim, we check that $d_\ceq(x,y,z)$ is definable in $\Lleq$, by proving that it equals
    $$\min\left(d_\leq(x,y),d_\leq(y,z),d_\leq(x,z)\right)
    + \min\left(d_\leq(y,z),d_\leq(z,x),d_\leq(y,x)\right)
    + \min\left(d_\leq(z,x),d_\leq(x,y),d_\leq(z,y)\right).$$
    If any of these three terms is nonzero, it is equal to $d_{\Delta_3}(x,y,z)$. The first term is nonzero when $y < x, z < y,$ and $z < x$ - that is, $z < y < x$.
    By symmetry, the second term is nonzero when $x < z < y$, and the third is nonzero when $y < x < z$.
    These are mutually incompatible conditions, and one of them occuring is equivalent to $\cyc(y,x,z)$.
    Thus the sum is equal to 0 when $\ceq(x,y,z)$, and otherwise, $d_{\Delta_3}(x,y,z)$.
\end{proof}

The converse is not true - not every metric cyclic order is the roll-up of a metric linear order.
\begin{eg}
    The roll-up of the linear order $[0,1)$ is the same cyclic order as the obvious cyclic order on the unit circle.

    This cyclic order is compatible with the metric $d(x,y) = \min(|x - y|,1 - |x - y|)$, which corresponds (up to a constant multiple) to geodesic distance around the unit circle.

    This cyclic order, with this metric, is not the roll-up of any metric linear order.
\end{eg}
\begin{proof}
    By the least upper bound property of the reals, if $M$ is the roll-up of a linear order, that linear order has either a minimum or a maximum element, but not both.
    This linear order is isomorphic, by a shift/rotation, to $[0,1)$ or $(0,1]$, in such a way that the induced metric remains $d(x,y) = \min(|x - y|,1 - |x - y|)$.
    This is not a valid metric linear order, as the open ball $B_{0.2}(0.1)$ is not convex.
\end{proof}

We can recover this converse in the ultrametric case.
\begin{lem}\label{lem:uco_roll-up}
    If $M$ is a ultrametric cyclic order, then $M$ is the roll-up of a metric linear order.
\end{lem}
\begin{proof}
    This is trivial if $|M| < 2$.

    We can first extend $M$ to a larger ultrametric cyclic order - if this larger ultrametric cyclic order has a compatible linear order, then the induced linear order on $M$ will also be compatible with the metric cyclic order on $M$.

    Because of this, if $D = \sup_{x,y \in M} d(x,y)$, then we may assume, by passing to an elementary extension, that $M$ contains elements $a,b \in M$ with $d(a,b) = D$, while still satisfying $\sup_{x,y} d(x,y) \leq D$.

    We then can partition $M$ into the order-convex pieces $B_D(a)$ (the open ball of radius $D$) and $M \setminus B_D(a)$. We give $B_D(a)$ a compatible linear ordering with $x \leq y \iff \ceq(x,y,b)$, and give $M \setminus B_D(a)$ a compatible linear ordering with $x \leq y \iff \ceq(x,y,a)$.
    We can then extend these to a linear order on $M$ by declaring
    $$B_D(a) < (M \setminus B_D(a)).$$
    This linear order is compatible with the cyclic order on $M$. To see that it is also compatible with the metric on $M$, let $B\subseteq M$ be an open ball.
    By the ultrametric property, the ball $B$ is either disjoint from $B_D(a)$ or comparable to it. Thus either $B \subseteq B_D(a), B \subseteq M \setminus B_D(a)$, or $B_D(a) \subsetneq B$, in which case $B$ has radius greater than $D$, so $B = M$.
    In all three cases, $B$ is order-convex.
\end{proof}

\subsection{Ultrametric Dense Cyclic Orders}
We now introduce a completion of the theory of ultrametric cyclic orders, analogous to $\UDLO$.

\begin{defn}
    Let $\UDCO$ be the theory consisting of $\MCO$, $\DU$, and the following axioms, for every $p \in \Q \cap (0,1)$:
    $$\sup_{u,v}\inf_{x,y,z \in \ceq} d(u,x) + d(v,y) + |d(u,z) - pd(u,v)| = 0,$$
    $$\sup_{u,v}\inf_{x,y,z \in \ceq} d(u,x) + d(v,y) + |d(v,z) - pd(u,v)| = 0.$$
\end{defn}

\begin{lem}\label{lem:udco_cond}
    Suppose that $(M,d,\ceq) \vDash \MCO \cup \DU$.
    Then $(M,d,\ceq)\vDash \UDCO$ if and only if for each $a,b \in M$, the sets $\{d(a,c) : \ceq(a,b,c)\} \cap [0,d(a,b)], \{d(b,c) : \ceq(a,b,c)\} \cap [0,d(a,b)]$ are both dense in $[0,d(a,b)]$.
\end{lem}
\begin{proof}
    First, we will check that this condition is sufficient for $(M,d,\ceq)\vDash \UDCO$.
    Without loss of generality, we show that for $p \in \Q \cap (0,1)$, $a,b \in M$ and $\varepsilon > 0$,
    $$\inf_{x,y,z \in \ceq} d(a,x) + d(b,y) + |d(a,z) - pd(a,b)| < \varepsilon.$$
    By our density condition, there is some $c \in M$ with $\ceq(a,b,c)$ such that $|d(a,c) - pd(a,b)| < \varepsilon$.
    Thus we see that 
    $$\inf_{x,y,z \in \ceq} d(a,x) + d(b,y) + |d(a,z) - pd(a,b)| \leq d(a,a) + d(b,b) + |d(a,c) - pd(a,b)| <\varepsilon.$$

    Now assume $(M,d,\ceq)\vDash \UDCO$, and without loss of generality, we will show that for all $a,b \in M$, the set $\{d(a,c) : \ceq(a,b,c)\} \cap [0,d(a,b)]$ is dense in $[0,d(a,b)]$.
    This is trivial if $a = b$, so assume $d(a,b) > 0$.
    Fix $p \in \Q \cap (0,1)$, and $0 < \varepsilon < \min\left(\frac{p}{2}d(a,b),(1-p)d(a,b)\right).$
    The axioms of $\UDCO$ tell us that
    there exist $a',b',c \in M$ with $\ceq(a',b',c)$ and
    $$d(a,a') + d(b,b') + |d(a,c) - pd(a,b)| < \varepsilon.$$
    
    It suffices to show that $\ceq(a,b,c)$, as $|d(a,c) - pd(a,b)| < \varepsilon$, which would assure density.
    By the ultrametric inequality, as $d(a,a') < d(a,b)$, we see that $d(a,b') = d(a,b)$.
    We see that $$d(a,c) > pd(a,b) - \varepsilon > \frac{p}{2}d(a,b) > \varepsilon > d(a,a').$$

    We claim that $\ceq(a,b',c)$. If this were not the case, then either $(a,b',a',c)$ or $(a',b',a,c)$ would be in cyclic order, but $d(a,a') < d(a,c),d(a,b').$

    Also, $d(a,c) < (1-p)d(a,b) + \varepsilon < d(a,b)$,
    so $d(b,c) = d(a,b) > d(b,b')$, so by a similar convexity argument, as $\ceq(a,b',c)$, we also have $\ceq(a,b,c).$
\end{proof}

\begin{lem}\label{lem:udlo_iff_udco}
    If $(M,d,d_\ceq)$ is the roll-up of a metric linear order $(M,d,d_\leq)$, then $(M,d,d_\leq) \vDash \UDLO$ if and only if $(M,d,d_\ceq) \vDash \UDCO$.
\end{lem}
\begin{proof}
    First, fix an ultrametric dense linear order $(M,d,d_\leq)$. We know that $(M,d,d_\ceq) \vDash \MCO \cup \DU$. Without loss of generality, fix $a,b \in M$, and we will show that $\{d(a,c) : \ceq(a,b,c)\} \cap [0,d(a,b)]$ is dense in $[0,d(a,b)]$
    By $\UDLO$'s density axioms, for any $p \in [0,d(a,b)]$ and $\varepsilon > 0$, there is some $c \in M$ with $c < a$ and $|d(a,c) - p| < \varepsilon$, and it suffices to show that $\ceq(a,b,c)$.
    If $a \leq b$, then we have $c < a \leq b$ so $\ceq(a,b,c)$.
    If $b < a$, then choose $0 < \varepsilon < d(a,b) - p$,
    so that $d(a,c) < d(a,b)$. Then convexity ensures that $b < c < a$, so $\ceq(a,b,c)$.

    Now fix $(M,d,d_\leq)$ and assume that $(M,d,d_\ceq) \vDash \UDCO$. It is easy to see that $(M,d,d_\leq) \vDash \DU \cup \ULO$, so without loss of generality, it suffices to show that for every $a \in M$ and $0 < s < t < 1$, there is some $c \in M$ with $a < c$ and $s < d(a,c) < t$.
    As $(M,d,d_\leq) \vDash \DU$, fix $b \in M$ with $d(a,b) > t$.
    Then by Lemma \ref{lem:udco_cond}, there is $c$ with $\ceq(a,c,b)$ and $s < d(a,c) < t$, and we can show that $a < c$. If $a < b$, we see that $a < c < b$. If $b < a$, then $c \leq a$ would imply $d(a,b) \leq d(c,a)$, a contradiction to $d(a,c) < t < d(a,b)$, so $a < c$.
\end{proof}

\begin{thm}\label{thm:UDCO_complete}
    The theory $\UDCO$  is complete.
\end{thm}
\begin{proof}
    Every model $(M,d,d_\ceq) \vDash \UDCO$ is a reduct of a model of the complete theory $\UDLO$, and thus satisfies the complete theory of $d_\ceq$-reducts of $\UDLO$.
\end{proof}

\subsection{Cyclic o-Minimality}
We now define regulated functions for the cyclic order context, and characterize definable predicates in dimension 1 over $\UDCO$.

\begin{defn}
    Let $X$ be a cyclic order.
    
    A \emph{cyclic interval} is a set of the form $\{c : \ceq(a,c,b)\}$ or $\{c : \cyc(a,c,b)\}$.

    Let the set of \emph{cyclic strong step functions} on $X$ be the $\R$-linear span of the set of characteristic functions of intervals in $X \to \R$,
    and let the set of \emph{cyclic weak step functions} be the $\R$-linear span of the set of characteristic functions of cyclically order-convex sets.

    Then let the set of \emph{cyclically strongly regulated} functions be the closure of the set of cyclic strong step functions under uniform limits, and let the set of \emph{cyclically weakly regulated} functions be the closure of the set of cyclic strong step functions under uniform limits.

    If $X$ is a linear order, say that $f : X \to \R$ is cyclically strongly/weakly regulated on $X$ when $f$ is cyclically strongly/weakly regulated on the roll-up of $X$.

    If $M$ is a structure expanding a model of $\MCO$, then we call $M$ \emph{cyclically weakly o-minimal} when the interpretation of every definable predicate in one variable with parameters is cyclically weakly regulated, and \emph{cyclically o-minimal} when the interpretation of every definable predicate in one variable with parameters is cyclically strongly regulated.
\end{defn}

\begin{lem}
    If $X$ is a dense linear order without endpoints, then $f: X \to \R$ is cyclically strongly regulated if and only if it is strongly regulated and $\lim_{x \to -\infty} f(x) = \lim_{x \to +\infty} f(x)$.
\end{lem}
\begin{proof}
    Cyclic intervals in the roll-up of $X$ are all unions of at most two intervals, so any cyclic strong step function is also a strong step function, and cyclically strongly regulated functions are strongly regulated.

    Any interval of $X$ \emph{not} containing an initial or final segment of $X$ is of the form $(a,b)$ or $[a,b]$ and thus is a cyclic interval of the roll-up.
    Any cyclic interval of the roll-up which contains an initial segment of $X$ also contains a final segment of $X$, and vice versa. (This relies on the fact that $X$ lacks endpoints.) Thus the cyclic intervals of the roll-up are precisely $X$, $\emptyset$, intervals of $X$ that are neither initial nor final segments, and unions of a nonempty initial ray and a nonempty final ray of $X$.
    
    From this, we see that a cyclic strong step function on the roll-up is precisely a strong step function $f$ where $\lim_{x \to -\infty} f(x) = \lim_{x \to +\infty} f(x)$,
    so any cyclically strongly regulated function $f$ satisfies this also.

    Let $f: X \to \R$ be a strongly regulated function satisfying $\lim_{x \to -\infty} f(x) = \lim_{x \to +\infty} f(x)$.
    For every $\varepsilon > 0$, there is a strong step function $g : X \to \R$ with $\sup_x|f(x) - g(x)|< \varepsilon$, and thus $\left|\lim_{x \to -\infty} g(x) - \lim_{x \to +\infty} g(x)\right| < 2\varepsilon$.
    Thus we could replace the value of $g$ on one of these rays with the value on the other, and get a cyclic strong step function $h$, without ever changing the value by more than $2\varepsilon$.
    Thus $\sup_x|f(x) - h(x)|< \varepsilon$, from which we conclude that $f$ is cyclically strongly regulated.
\end{proof}

\begin{thm}
    If $M \vDash \UDCO$, then the interpretations of definable predicates with parameters in one variable are precisely the cyclically strongly regulated functions which are continuous with respect to the metric, which are all given by quantifier-free definable predicates with parameters.

    In particular, $M$ is cyclically o-minimal.
\end{thm}
\begin{proof}
    As $M$ is a reduct of some model of $\UDLO$, all interpretations of definable predicates with parameters in one variable are strongly regulated with respect to that linear order, and all are continuous with respect to the metric.
    
    Let $\phi(x)$ be an $\Lceq$-predicate with parameters in $M$. It remains to show that if $\psi(x)$ is the $\Lleq$-predicate with the same parameters with the same interpretation as $\phi(x)$ with respect to this linear order, then $\lim_{+\infty}\psi(x) = \lim_{-\infty}\psi(x)$.
    This is equivalent to showing that
    $$(M,d,d_\leq) \vDash \inf_x\inf_y\sup_{w : w \leq x} \sup_{z : y \leq z}|\psi(w) - \psi(z)| = 0.$$

    Fix an elementary extension $(N,d,d_\leq) \succeq (M,d,d_\leq)$ such that there are $a,b \in N$ with $d(a,b) = 1$. It suffices to show that $$(N,d,d_\leq) \vDash \inf_x\inf_y\sup_{w : w \leq x} \sup_{z : y \leq z}|\psi(w) - \psi(z)| = 0.$$
    
    Examining the proof of Lemma \ref{lem:uco_roll-up}, we see that $(N,\ceq)$ can be partitioned into two disjoint cyclically order-convex sets $A, B$ such that either defining $\forall a \in A, \forall b \in B, a \leq b$ or defining $\forall a \in A, \forall b \in B, b \leq a$ induces a distinct valid $\UDLO$ structure on $N$. Thus in at least one of these linear orders, there is a cut placing a final segment of our initial linear order next to an initial segment of the initial linear order. As the limits of $\psi(x)$ on either side of the cut are equal, we find that in the original linear order,
    $$(N,d,d_\leq) \vDash \inf_x\inf_y\sup_{w : w \leq x} \sup_{z : y \leq z}|\psi(w) - \psi(z)| = 0.$$
\end{proof}
 
\section{Ordered Metric Valued Fields}
Prior to this paper, there was one notable example in the continuous logic literature of a class of ordered structures: ordered metric valued fields, as introduced in \cite[Section 3]{mvf}. We will examine ordered metric valued fields, and in particular, their model completion, ordered real closed metric valued fields, or the theory \ORCMVF.
There are two ways to consider these as metric structures: the projective line, as in \cite{mvf}, or a valuation ring (for a valuation coarsening the metric), as in \cite{tilting}.

The projective line has the advantage of being axiomatizable without requiring a coarsening of the valuation - see \cite[Prop. 1.2]{mvf}. It will turn out to be best understood as a metric cyclic order, in which context it is o-minimal (Theorem \ref{thm:ORCMVF_reg}).

The valuation ring has the advantage of being simpler to define, and will turn out to be a metric linear order. Restricting the domain to the valuation ring also defines additional structure, as the valuation ring is not otherwise definable in the projective line. We combine this framework with that of metric linear orders to define a theory $\ORCMVR$ of ordered real closed metric valued fields, for which we prove quantifier elimination. Models of this theory turn out to be only \emph{weakly} $o$-minimal (Theorem \ref{thm:ORCMVR_omin}).

We observe that the two complete theories of projective lines of metric valued fields introduced in \cite{mvf} are tamer than the analogous theories in classical logic: $\Th{ACMVF}$ is stable, while $\Th{ACVF}$ is only NIP, and $\ORCMVF$ is $o$-minimal while $\mathsf{RCVF}$ is only weakly $o$-minimal.
The same does not occur for the valuation ring framework, where both theories reproduce the wilder classical behavior.

This phenomenon can be at least partially explained in terms of the Ax-Kochen-Ershov principle. Such a principle is elaborated for the projective line in \cite{ake}, where the value group is always a subgroup of the positive reals. To develop an analogous principle for the valuation ring framework will require studying the value groups as metric linear orders.

\subsection{Projective Line}

\begin{lem}\label{lem:rcmvf_c_pred}
    In a model of \RCMVF{},
    $$
    \inf_t\norm{
    (x - y)(y - z)(z - x)
    - (x^2 + 1)
    (y^2 + 1)
    (z^2 + 1)
    t^2
    } = 
    \begin{cases}
        0 & \ceq(x,y,z) \\
        d(x,y)d(x,z)d(y,z) & \textrm{otherwise}.
    \end{cases}$$
\end{lem}
\begin{proof}
    Firstly, we can expand this in homogeneous coordinates as
    \begin{align*}
        &\inf_{[t:t^*]}\abs{
    (xy^* - x^*y)(yz^* - y^*z)(zx^* - z^*x)(t^*)^2
    -(x^2+(x^*)^2)(y^2+(y^*)^2)(z^2+(z^*)^2)t^2
    }\\
    =& \inf_{[t:t^*]}\abs{
    P^h(x,x*,y,y^*,z,z^*)(t^*)^2
    -Q^h(x,x^*)Q^h(y,y^*)Q^h(z,z^*)t^2
    },
    \end{align*}
    where $P(x,y,z) = (x - y)(y - z)(z - x)$ and $Q(x) = x^2 + 1$.

    For any $[x : x^*]$, the binomial $Q^h(x,x^*)$ will always be positive,
    and $|Q^h(x,x^*)| = |x^2 + (x^*)^2| = \max(|x|,|x^*|)^2 = 1$.
    
    If $P^h(x,x*,y,y^*,z,z^*) \geq 0$, then it can be written as $u^2$,
    and $Q^h(x,x^*)Q^h(y,y^*)Q^h(z,z^*)$ can be written as $v^2$ where $|v| = 1$, so if we let $t = [u:v]$,
    we find that
    $$P^h(x,x*,y,y^*,z,z^*)(t^*)^2
    +Q^h(x,x^*)Q^h(y,y^*)Q^h(z,z^*)t^2
    = u^2v^2
    -v^2u^2 = 0.$$

    If $P^h(x,x*,y,y^*,z,z^*) < 0$, then for any $[t:t^*]$,
    \begin{align*}
        &\abs{
    P^h(x,x*,y,y^*,z,z^*)(t^*)^2
    -Q^h(x,x^*)Q^h(y,y^*)Q^h(z,z^*)t^2
    }\\
    =& \max\left(\abs{
    P^h(x,x*,y,y^*,z,z^*)}|t^*|^2,\abs{Q^h(x,x^*)}\abs{Q^h(y,y^*)}\abs{Q^h(z,z^*)}|t|^2
    \right),\\
    \end{align*}
    so as either $|t| = 1$ or $|t^*| = 1$, this quantity is minimized in $[t:t^*]$ when $t = 0$ and $|t| = 1$,
    with the value
    $$|P^h(x,x*,y,y^*,z,z^*)| =
    |xy^* - x^*y||yz^* - y^*z||zx^* - z^*x|
    = d(x,y)d(y,z)d(z,x).$$

    It now suffices to show that $\ceq([x:x^*],[y:y^*],[z:z^*])$ is equivalent to $P^h(x,x*,y,y^*,z,z^*) \geq 0$.
    If none of these points are infinite, we may assume that $x^*,y^*,z^*$ are all positive, and then we see that the sign of $P^h(x,x*,y,y^*,z,z^*)$ is the same as the sign of $P\left(\frac{x}{x^*},\frac{y}{y^*},\frac{z}{z^*}\right)$, and one can check by casework that this is nonnegative if and only if $\ceq\left(\frac{x}{x^*},\frac{y}{y^*},\frac{z}{z^*}\right)$ holds.

    If multiple points are infinite, or in general, if any two points are equal, then $\ceq([x:x^*],[y:y^*],[z:z^*])$ holds and $P^h(x,x*,y,y^*,z,z^*) = 0$.

    If only one point is infinite, we may assume that $[z:z^*] = [1,0]$ and $x^*,y^* > 0$.
    Then we find that
    $$P^h(x,x*,y,y^*,z,z^*) = (x^*y^*)^2\left(\frac{y}{y^*}-\frac{x}{x^*}\right),$$
    which is nonnegative if and only if
    $\frac{y}{y^*}\geq \frac{x}{x^*}$,
    which is precisely the condition under which
    $\ceq\left(\frac{x}{x^*},\frac{y}{y^*},\infty\right)$ holds.
\end{proof}

\begin{thm}
    In a model of \RCMVF{},
    the relation $\ceq(x,y,z)$ is a definable set.
\end{thm}
\begin{proof}
    The set $C$ is definable if there is some definable predicate $\phi(x,y,z)$ such that if $\ceq(x,y,z)$, then $\phi(x,y,z) = 0$, and otherwise, $d((x,y,z),C) \leq \phi(x,y,z)$.

    We will use
    $$\phi(x,y,z) = \sqrt[3]{\inf_t\norm{
    (x - y)(y - z)(z - x)
    - (x^2 + 1)
    (y^2 + 1)
    (z^2 + 1)
    t^2
    }}.$$
    
    By Lemma \ref{lem:rcmvf_c_pred},
    $$\phi(x,y,z) =
    \begin{cases}
        0 & \ceq(x,y,z) \\
        \sqrt[3]{d(x,y)d(x,z)d(y,z)} & \textrm{otherwise},
    \end{cases}$$
    and if $(x,y,z)\not\in C$, then 
    $$d((x,y,z),C) \leq \min(d(x,y),d(x,z),d(y,z)) \leq \sqrt[3]{d(x,y)d(x,z)d(y,z))} = \phi(x,y,z).$$
\end{proof}

This means that any model of \RCMVF{} (or, indeed, of $\ORCMVF$) can be expanded by definitions to an expansion of a metric cyclic order.
Because this metric cyclic order is definable from the metric valued field structure, we will just call that resulting structure the reduct.

\begin{lem}\label{lem:ORCMVF_UDCO}
    The reduct of any $M \vDash \RCMVF{}$ to $\Lceq$ models \UDCO{}.
\end{lem}
\begin{proof}
    We already know that the reduct structure is an ultrametric cyclic order. To show density, let $a, b \in M$ - we wish to show that the sets $\{d(a,c) : c \in \ceq(a,b,c) \}$ and $\{d(a,c) : c \in \ceq(a,c,b)\}$ are both dense in $[0,d(a,b)]$.

    Because $[a:a^*] \mapsto [a^*:a]$ is a cyclic-order-preserving isometry of $M$, we may assume that $|a| \leq 1$ and $a^* = 1$. 
    We start by showing the density of $\{d(a,c) : c \in \ceq(a,c,b)\}$.
    Now let $0 < s < t < d(a,b)$. By the density of the valuation in a real closed metric valued field, choose $q$ nonnegative such that $s < |q| < t$, and let $c = [a + q:1]$.
    We find that $\norm{c} = |a + q| \leq \max(|a|,|q|) \leq 1$, and $d(a,c) = |ac^* - ca^*| = |a - (a + q)| = |q|$, so $s < d(a,c) < t$.
    Because $q$ is nonnegative, $a \leq c$.
    To show that $\ceq(a,c,b)$, we will see that either $b \leq a$, $b = \infty$, or $c \leq b$. The set $S$ of elements $[\alpha : \alpha^*]$ with $|\alpha^*| = 1$ is order-convex, so if $b$ is outside this set, we may assume that either $b = \infty$, $b \leq S$, or $S \leq b$, and in all cases we are done.
    Thus we may assume that $b = [b : 1]$, and we may also assume that $a < b$. In this regime, we find that $d(a,b) = |b - a|$. As $|c - a| < |b - a|$ and both differences are positive, we find that $a < c < b$.

    Similarly, we could show the density of $\{d(a,c) : c \in \ceq(a,c,b)\}$ by setting $c = [a - q:1]$.
\end{proof}

\begin{thm}\label{thm:ORCMVF_reg}
    Every model of \ORCMVF{} is a cyclically o-minimal expansion of \UDCO{}.
\end{thm}
\begin{proof}
    Let $M \vDash \ORCMVF{}$, let $\phi(x;\bar y)$ be a definable predicate with $|x| = 1$, and let $\bar b \in M^{\bar y}$. We wish to show that $\phi(x;\bar b)$ is strongly regulated.

    By \cite[Theorem 3.11]{mvf}, \ORCMVF{} has quantifier elimination. Thus, as in the proof of \cite[Theorem 3.12]{mvf} we may assume $\phi(x;\bar b)$ is an atomic predicate - these take the form $\norm{P(x,\bar b)}$ or $\langle P(x,\bar b)\rangle$.
    Each of these equals, in turn, a continuous combination of formulas of the form $|x - c|$ or $\langle x - c\rangle$, so it suffices to show these predicates are cyclically strongly regulated. (Note that the exact continuous combination this equals and the new parameters $c$ may depend on $\bar b$). We will show that these are cyclically strongly regulated using Lemma \ref{lem:reg_partition}.
    
    Fix $0 \leq s < t \leq 1$.
    We  first claim that there is some $a \in M$ with $s < |a| < t$ - in general, the set of valuations of elements of $M$ is dense in $[0,1]$.
    To show this, we first note that by one of the axioms in \cite[Definition 3.5]{mvf}, for any $\delta > 0$, there is an element $h \in M$ with $\left||h| - \frac{1}{2}\right| < \delta$ - all we need is that $0 < |h| < 1$.
    There must exist some rational $\frac{p}{q}$ such that $s < |h|^{\frac{p}{q}} < t$. We may assume $q$ is odd.
    As $M$ is a real closed valued field, there is some $a \in M$ with $a^q = h^p$, so we see that $|a| = |h|^{\frac{p}{q}}$, and $s < |a| < t$ as desired.

    Now consider $|x - c|$. We may partition $M$ into three cyclic intervals: $(\infty,c - a),[c - a,c + a],(c + a, \infty]$. Then on the first and last piece of this partition, we have $|x - c| \geq |a| > s$, while on the middle part, we have $|x - c| \leq |a| < t$.

    Meanwhile, we can understand $\langle x - c\rangle$ best as a piecewise function.
    The immediate definition is that
    $$\langle x - c\rangle 
    = \begin{cases}
        0 & x - c \geq 0\\
        \norm{x -c} \wedge \norm{(x - c)^*} & x - c < 0,
    \end{cases}$$
    but this amounts to 
    $$\langle x - c\rangle 
    = \begin{cases}
        0 & x - c \geq 0\\
        |x - c| & -1 \leq x - c < 0\\
        |(x - c)|^{-1} & \infty < x - c < 1.
    \end{cases}$$
    As all three pieces are strongly regulated functions, and the pieces of the definitions are intervals, the piecewise definition is as cyclically strongly regulated as well.
\end{proof}

\subsection{Valuation Ring}

Now we compare with the formalism for metric valued fields introduced in \cite{tilting}.
In this context, the metric structure is just the valuation ring, so we will refer to such structures as \emph{metric valuation rings}.

\begin{defn}
    Let $\Lvr$ be the metric language $\{0,1,+,-,\cdot, D\}$ of valued rings, and let $\Lovr = \Lring \cup \Lleq$.

    We use $|x|$ as notation in this language for $d(x,0)$.

    Let $\MVR$ be the theory of metric valuation rings, referred to in \cite[Lemma 3.3]{tilting} as $\MVF$,
    together with the axiom $D(x,y) = \inf_z|y - xz|$.
\end{defn}

\begin{fact}[{\cite[Lemma 3.3]{tilting}}]
    The models of $\MVR$ are precisely of the form $(\mathcal{O}(K),|x - y|,0,1,+,-,\cdot,D)$ where $\mathcal{O}(K)$ is the valuation ring of a valued field $(K,v)$, and $|\cdot|$ is an absolute value corresponding to a valuation coarsening $v$, and $D(x,y) = \inf_{z \in \mathcal{O}(K)}|y - xz|$.
\end{fact}

First, let us make some observations about $D(x,y)$ and the value group $v\mathcal{O}$. It is mentioned in \cite{tilting} that $|x - y\mathcal{O}(K)^\times| = \max(D(x,y),D(y,x))$ is the pseudometric on the imaginary corresponding to the value group $v\mathcal{O}$. Here, we observe that $D(x,y)$ makes this imaginary a metric linear order.
\begin{lem}\label{lem:D_pred}
    For all $x,y \in \mathcal{O}(K)$, $$D(x,y) = \begin{cases}
        0 & \textrm{if }v(x) \leq v(y) \\
        |y| & \textrm{otherwise}
    \end{cases}$$
\end{lem}
\begin{proof}
    If $v(x) \leq v(y)$, then $x$ divides $y$, so $D(x,y) = 0$.
    If $v(x) > v(y)$, then for every $z \in \mathcal{O}(K)$, 
    $v(xz) \geq v(x)$, and thus $v(y - xz) = v(y)$, so also, $|y - xz| = |y|$.
    Thus $D(x,y) = \inf_z|y - xz| = |y|$.
\end{proof}
\begin{thm}\label{thm:value_ulo}
    The imaginary $v\mathcal{O}$ is a metric linear order, and for all $x,y$, $r(v(x),v(y)) = D(x,y)$, so this metric linear order structure is interpretable. In particular, $v\mathcal{O} \vDash \ULO$.
\end{thm}
\begin{proof}
    First, observe that if $v(w) = v(x), v(y) = v(z)$, then $D(x,y) = D(w,z)$, so this definition of $r$ is a definable predicate on this imaginary.

    We know that $v\mathcal{O}$ is linearly ordered, and that $D(x,y) = 0$ for $v(x) \leq v(y)$.
    Meanwhile, when $v(x) > v(y)$, we see that $d(v(x),v(y)) = \max(D(x,y),D(y,x)) = D(x,y)$, so $D(x,y)$ has the correct piecewise definition to be the predicate $r$. It now suffices to show that for all $a \in \mathcal{O}$, $r(v(x),v(a)) = D(x,a)$ is nondecreasing in $v(x)$, which is true because it equals $0$ for $v(x) \leq v(a)$, and $|a|$ when $v(x) > v(a)$.

    Lastly, we note that this metric is an ultrametric, as given $x,y,z$ with $v(x) < v(y) < v(z)$, we have
    $$d(v(x),v(y)) = |y| \leq |z| = d(v(x),v(z)) = d(v(y),v(z)).$$
\end{proof}

We now recall some facts about formally real and ordered valued fields before defining formally real and ordered metric valuation rings.
\begin{fact}[{\cite[Lemma 3.2]{mvf}}]\label{fact:frspec}
    A valued field is formally real if and only if it satisfies
    $$\forall \bar x, v\left(\sum_i x_i^2\right) = \min_i v\left(x_i^2\right).$$

    A valued field equipped with an ordering is an ordered valued field if and only if for all $x,y \geq 0$, $v(x + y) = \min(v(x),v(y)).$
\end{fact}
It is straightforward to see that if a field is equipped with two valuations, one coarsening the other, and is formally real or ordered with respect to the finer one, it is also formally real or ordered with respect to the coarser one.

\begin{lem}
    If $M \vDash \MVR$, then the fraction field of $M$ is formally real if and only if for every $n$, $M$ satisfies the following conditions:
    \begin{align*}
        &\sup_{\bar x} \min_i D\left(\sum_{i = 1}^n x_i^2,x_i^2\right) = 0\\
        &\sup_{\bar x} D\left(x_i^2,\sum_{i = 1}^n x_i^2\right) = 0,
    \end{align*}
    and if $M$ is an $\Lovr$-structure also satisfying $\MLO$, then the ordering on $M$ extends to a valued field ordering on the fraction field of $M$ if and only if $M$ satisfies the conditions
    \begin{align*}
        &\sup_{x,y \geq 0} \min\left(D(x + y,x),D(x + y,y)\right) = 0\\
        &\sup_{x,y \geq 0} D(x,x + y) = 0.
    \end{align*}

    We define the $\Lvr$-theory $\Th{FRMVR}$ of formally real metric valued fields and the $\Lovr$-theory $\Th{OMVR}$ of ordered metric valued fields accordingly.
\end{lem}
\begin{proof}
    It is easy to see by scaling that if the fraction field of $M$ is not formally real or not validly ordered, the witnesses in Lemma \ref{fact:frspec} can be chosen from the valuation ring.

    The axioms then simply ensure that the valuations of the appropriate sums are the minima of the valuations of the elements they are summing, and these properties extend to the coarser absolute value valuation as well.
\end{proof}

\begin{fact}[{\cite[Theorem 1]{rcvr}}]\label{fact:rcvr_spec}
    An ordered ring $M$ is the valuation ring of an ordered real closed valued field if and only if the following hold:
    \begin{itemize}
        \item For all $0 \leq a \leq b$ in $M$, $b$ divides $a$
        \item Every nonnegative element of $M$ is a square.
        \item Every monic polynomial of odd degree in $M$ has a root in $M$.
    \end{itemize}
\end{fact}

The first of these properties is already covered by the axioms of an ordered valuation ring, allowing us to axiomatize ordered real closed metric valuation rings accordingly:
\begin{lem}\label{lem:orcmvr_spec}
    The fraction field of a model $M \vDash \OMVR$ is a real closed valued field with respect to both valuations if and only if $M$ satisfies the following additional conditions, with the latter holding for each $n \in \N$:
    \begin{align*}
        &\sup_{x \geq 0}\inf_y d(x,y^2) = 0\\
        &\sup_{x_1,\dots,x_{2n}}\inf_y \left|y^{2n+1} + \sum_{i = 0}^{2n} x_iy^i\right| = 0.\\
    \end{align*}
\end{lem}
\begin{proof}
    This follows by Fact \ref{fact:rcvr_spec}.
\end{proof}
\begin{defn}
    Let $\ORCMVR$ be the theory consisting of $\OMVR$, the additional conditions of Lemma \ref{lem:orcmvr_spec}, and the axiom 
    $$\inf_x\abs{|x| - \frac{1}{2}} = 0$$
    implying that $|\cdot|$, and therefore $v$, are nontrivial on $M$.
\end{defn}

\begin{lem}
    Any model of $\ORCMVR$ is also a model of $\UDLO$.
\end{lem}
\begin{proof}
    Such a model is clearly an ultrametric linear order, so it suffices to prove density.

    We observe that as the absolute value on $M$ is nontrivial, there is some $a \in M$ with $0 <|a| <1$, and optionally switching $a$ with $-a$, we assume $0 \leq a$.
    Thus by exponentiating and taking roots of $a$, we may find positive elements of valuations $|a|^{q}$ for all $q \in \Q \cap (0,\infty)$, which are dense in $(0,1)$.

    Thus for any $b \in M$, and any $0 \leq s < t \leq 1$, we may find an element $c = a^q$ such that $s < |c|< t$, and see that $s < d(b,b + c) < t$, and $b < b + c$. Similarly, $s < d(b,b - c) < t$, while $b - c < b$.
\end{proof}

Before proving quantifier elimination for the theory $\ORCMVR$, we characterize embeddings of (ordered) metric valued rings, in terms of embeddings with respect to the discrete language $\Lrm{div}$ of rings with the divisibility symbol $x|y$, interpreted as $v(y) \geq v(x)$.
\begin{lem}
    Embeddings of metric valuation rings are exactly embeddings in the discrete language $\mathcal{L}_{\mathrm{div}}$ that also preserve $|\cdot|$.

    If $f : M \to N$ is an $\mathcal{L}_{\mathrm{div}}$-embedding, and there is $m \in M$ such that $0 < |m| < 1$ and $|f(m)| = |m|$, then $f$ is an embedding of metric valuation rings.

    Additionally, embeddings of ordered metric valuation rings are just order-preserving embeddings of metric valuation rings.
\end{lem}
\begin{proof}
    An embedding of metric valuation rings is definitionally a map preserving the ring structure, the metric, and the additional symbol $D$. Preserving the metric is equivalent to preserving $|\cdot|$, and preserving $D$ is equivalent to preserving both $|\cdot|$ and divisibility by Lemma \ref{lem:D_pred}.

    Let $f : M \to N$ be an $\mathcal{L}_{\mathrm{div}}$-embedding, and let $m \in M$ be such that $0 < |m| < 1$ and $|f(m)| = |m|$. Because this embedding preserves divisibility, we see that for all $a,b \in M$, if $|a| < |b|$, then $b | a$, so $f(b) | f(a)$, and thus $|f(b)| \leq |f(a)|$.
    
    Now fix $a \in M$, and we will show that $|f(a)| = |a|$. If $|a| = 0$, then $|f(a)| = |f(0)| = 0$. If $|a| = 1$, then for all $k \in \N$, $|m| < |a^k|$, so $|f(m)| \leq |f(a)|^k$. This can only happen if $|f(a)| = 1$.
    Otherwise, fix $0 < s < |a| < t < 1$, and we will show that $s < |f(a)| < t$. There are then positive rationals $p,q \in \Q \cap (0,\infty)$ with $s < |m|^p < |a| < |m|^q < t$. Let $i \in \N$ be such that $i > 0$ and $pi,qi$ are integers. Then 
    $|m|^{pi} < |a|^i < |m|^{qi}$, from which we can deduce that $s < |f(m)|^{pi} \leq |f(a)|^i \leq |f(m)|^{qi} < t$. 

    The last claim is true as the only additional symbol is the metric linear order symbol $r$, which is preserved as order and metric are.
\end{proof}

\begin{lem}
    The theory $\ORCMVR$ admits quantifier elimination.
\end{lem}
\begin{proof}
    To do this, assume that $M, N \vDash \ORCMVR$, that $A$ is a substructure of $M$ with $f : A \hookrightarrow N$, and that $N$ is $|M|^+$-saturated.
    We will show that $f$ extends to an embedding of $M$ into $N$, analogously to the proof of \cite[Prop. 3.7]{tilting} for the algebraically closed case.

    We start by extending $A$ and $f$ to an embedding from a substructure of $M$ on which $|\cdot|$ is not trivial.
    Suppose $c \in M$ satisfies $0 < |c| < 1$, and $0 < c$. Then to correctly extend the embedding to $c$, and the substructure generated by $A$ and $c$, it suffices to map $c$ to some $d \in N$ where $|d| = |c|$ and $0 < d$. This is possible by saturation of $N$.
    By performing this extension first, we may assume $|\cdot|$ is not trivial on $A$, and $f$ preserves $|\cdot|$, so any order-preserving $\mathcal{L}_{\mathrm{div}}$-embedding $M \to N$ extending $f$ will automatically be an embedding of ordered metric valuation rings.

    We now define $N_1 \supseteq N$, a larger, $|M|^+$-saturated model of the discrete language of ordered real closed valued fields. To do this, if $K_N$ is the fraction field of $N$, let $N_1$ be the valuation ring of the Hahn series field $K_N((\Gamma))$, where $\Gamma$ is an $|M|^+$-saturated divisible ordered abelian group, taking the natural valuation in $v(N) \times \Gamma$. The Hahn series field $N_1$ is saturated, as it is maximally complete, the value group is saturated, and the residue field is saturated.

    By \cite{rcvr}, the theory of real closed valuation rings (valuation rings of real closed valued fields) in the (discrete) language of ordered valued rings admits quantifier elimination. Thus $f$ extends to an embedding $g : M \to N_1$ of real closed valued rings.
    By composing this with the quotient map $N_1 \to N$, quotienting out by the ideal $I = \{x : v(x) > v(N)\}$, we get a map $h : M \to N$.
    To show this is an embedding, it suffices to show that $g(M) \cap I = \{0\}$. Suppose for contradiction that $m \in M$ is nonzero with $m \in I$. Choose $n \in N, a \in A$ with $0 < |n| < |a| < |m|$.
    Then $|n| < |g(a)|$, and $v(g(m)) < v(g(a))$, so $v(g(m)) < v(n)$, contradicting the assumption that $m \in I$.
\end{proof}

\begin{thm}\label{thm:ORCMVR_omin}
    Every model of $\ORCMVR$ is weakly o-minimal, but not o-minimal.
\end{thm}
\begin{proof}
    To show $M \vDash \ORCMVR$ is not o-minimal, let $a \in M$ be such that $a > 0$ and $0 < |a| < 1$. Then the ideal $aM$ is definable by the predicate $d(x,aM) = \inf_y d(x,ay)$. We will show that $d(x,aM)$ is not strongly regulated.

    The ideal $aM$ consists of exactly those elements $x$ of $M$ with $v(x) \geq v(a)$.
    Thus for any $b \in M \setminus aM, c \in aM$, we have that $b - c \not \in aM$, so $v(b - c) < v(a)$ and thus $d(b,c) = |b - c| \geq |a|$, so $d(x,aM) >0$ implies $d(x,aM) \geq |a|$.
    If $d(x,aM)$ is strongly regulated, $M$ can be partitioned into finitely many intervals on which either $d(x,aM) > 0$ or $d(x,aM) < |a|$, but we see that the latter condition implies $x \in aM$, so $aM$ must itself be a finite union of intervals.
    
    The ideal $aM$ is order-convex and bounded above, so it suffices to show that it is not an interval. Assume for contradiction that $b$ is a least upper bound.
    If $b \in aM$, we can contradict $b$ being an upper bound by finding that $b < b + b \in aM$.
    If $b \not \in aM$, then as before, we see that $v(b - a) < v(a)$, so $b - a \not \in aM$. However, $0 < b - a < b$, so by order-convexity of $aM$, we can conclude that $b - a$ is a smaller upper bound for $aM$.

    To show $M \vDash \ORCMVR$ is weakly o-minimal, it suffices to show that if $\phi(x)$ is an atomic predicate with parameters in $M$ with $|x| = 1$, then $\phi(x)$ is weakly regulated.
    All atomic $\Lleq$-predicates will be strongly regulated, as $\ORCMVR$ implies $\UDLO$, leaving us to check $D(p(x,\bar a),q(x,\bar a))$ for polynomials $p,q$ and parameters $\bar a$.

    For a given value of $x$, by Lemma \ref{lem:D_pred} either $v(p(x,\bar a)) \leq v(q(x,\bar a))$, in which case $D(p(x,\bar a),q(x,\bar a)) = 0$, or $D(p(x,\bar a),q(x,\bar a)) = |q(x,\bar a)|$. We will show that the former condition holds on a finite union of order-convex sets. This means that $M$ can be partitioned into a finite union of order-convex sets on which this predicate either equals 0, or the strongly regulated function $|q(x,\bar a)|$. The latter order-convex sets can then be further partitioned into a finite union of order-convex sets on which that function is approximately constant, showing that $D(p(x,\bar a),q(x,\bar a))$ is weakly regulated.

    Note that any boolean combination of order-convex sets is a finite union of order-convex sets. Thus if we write $v(p(x,\bar a)) = \min_{i \leq m} f_i(x)$, $v(q(x,\bar a)) = \min_{j \leq n} g_j(x)$, then it suffices to show that the set $\{x \in M : f_i(x)\leq g_j(x)\}$ is itself a boolean combination of order-convex sets.

    We can factor $p(x, \bar a)$ into polynomials of the forms $c$, $x - c$, and $(x - c)^2 + b$, where $c$ is any constant and $b$ is a positive constant.
    As $v((x - c)^2 + b) = \min(2v(x - c),v(b))$,
    we find that $v(p(x,\bar a))$ is a sum of valuations of the forms $v(c), v(x - c), \min(2v(x - c),v(b))$, and is thus the minimum of a finite collection of sums of valuations of the form $v(c)$ and $v(x - c)$. Thus we may assume $v(p(x, \bar a))$ and $v(q(x, \bar a))$ are each such a sum.

    Write $v(p(x, \bar a)) = v(b_0) + \sum_{i = 1}^m v(x - b_i)$, $v(q(x, \bar a)) = v(c_0) + \sum_{j = 1}^n v(x - c_j)$.

    Let $C = \{b_i: 1 \leq i \leq m\} \cup \{c_j : 1 \leq j \leq n\}$.

    For each $c^* \in C$, on the order-convex set where $v(x - c^*) \geq v(x - c)$ for all other $c \in C$, we see that $v(x - c) = \min(v(x - c^*),v(c - c^*))$.
    Thus $M$ can be partitioned into a finite collection of order-convex sets on which $v(p(x, \bar a)), v(q(x,\bar a))$ are each the minimum of a finite collection of valuations of the form $kv(x - c^*) + v(c)$ for some $1 \leq i_0 \leq n$, $k \in \N$, $c \in M$. Thus we may now reduce to the case where each of $v(p(x,\bar a)),v(q(x,\bar a))$ is of this form, with $v(p(x,\bar a)) = jv(x - c^*) + v(b)$, $v(q(x,\bar a)) = kv(x - c^*) + v(c)$.
    
    In this case, $v(p(x, \bar a)) \leq v(q(x, \bar a))$ precisely when $(k - j)v(x - c^*) \leq v(b) - v(c)$. If $j = k$, this does not depend on $x$, and otherwise, this occurs when $v(x - c^*) \leq \frac{v(b) - v(c)}{k - j}$, which is the complement of an order-convex set containing $c^*$.
\end{proof}

\bibliographystyle{plainurl}
\bibliography{ref.bib}

\end{document}